\documentclass[11pt]{amsart}

\newtheorem{thm}{Theorem}[section]
\newtheorem*{thm*}{Theorem}
\newtheorem{lem}[thm]{Lemma}
\newtheorem*{lem*}{Lemma}

\newtheorem{cor}[thm]{Corollary}

\newtheorem{prop}[thm]{Proposition}

\theoremstyle{definition}

\newtheorem*{case*}{Case}

\newtheorem{defn}[thm]{Definition}
\newtheorem*{defn*}{Definition}
\newtheorem{exmp}[thm]{Example}
\newtheorem*{exmp*}{Example}

\newtheorem{step}{Step}\renewcommand{\thestep}{}

\theoremstyle{remark}

\newtheorem{case}{Case}\renewcommand{\thecase}{}

\newtheorem{rmk}[thm]{Remark}
\newtheorem*{rmk*}{Remark}

\makeatletter
\def\alphenumi{
  \def\theenumi{\alph{enumi}}
  \def\p@enumi{\theenumi}
  \def\labelenumi{(\@alph\c@enumi)}}
\makeatother

\makeatletter
\def\thecase{\@arabic\c@case}
\makeatother

\makeatletter
\def\thestep{\@arabic\c@step}
\makeatother

\newcount\hh
\newcount\mm
\mm=\time
\hh=\time
\divide\hh by 60
\divide\mm by 60
\multiply\mm by 60
\mm=-\mm
\advance\mm by \time
\def\hhmm{\number\hh:\ifnum\mm<10{}0\fi\number\mm}

\let\oldmarginpar\marginpar
\renewcommand\marginpar[1]{\-\oldmarginpar[\raggedleft\footnotesize #1]%
{\raggedright\footnotesize #1}}

\renewcommand\emptyset{\varnothing}

\newcommand\CC{\mathbb{C}}

\newcommand\NN{\mathbb{N}}

\newcommand\RR{\mathbb{R}}

\newcommand\fb{{\mathfrak{b}}}

\newcommand\fK{{\mathfrak{K}}}

\newcommand\sC{{\mathscr{C}}}

\newcommand\sO{{\mathscr{O}}}

\newcommand\sQ{{\mathscr{Q}}}

\newcommand\sS{{\mathscr{S}}}

\newcommand\sU{{\mathscr{U}}}
\newcommand\sV{{\mathscr{V}}}

\newcommand\eps{\varepsilon}

\newcommand\less{\setminus}

\newcommand\diam{\operatorname{diam}}
\DeclareMathOperator{\mydirac}{\slashed{\partial}}
\newcommand\dist{\operatorname{dist}}

\newcommand{\esssup}{\operatornamewithlimits{ess\ sup}}

\DeclareMathOperator{\Int}{int}

\newcommand\loc{\operatorname{loc}}

\newcommand\supp{\operatorname{supp}}

\newcommand\tr{\operatorname{tr}}

\numberwithin{equation}{section}

\usepackage{amssymb}
\usepackage{graphicx}
\usepackage{hyperref}
\usepackage{mathrsfs}
\usepackage[usenames]{color}
\usepackage{slashed}
\usepackage{url}
\usepackage{verbatim}

\hypersetup{pdftitle={Perturbations of local maxima and comparison principles for boundary-degenerate linear differential equations}}
\hypersetup{pdfauthor={Paul M. N. Feehan}}

\begin{document}

\title[Maximum principles for boundary-degenerate differential equations] {Perturbations of local maxima and comparison principles for boundary-degenerate linear differential equations}

\author[Paul M. N. Feehan]{Paul M. N. Feehan}
\address{Department of Mathematics, Rutgers, The State University of New Jersey, 110 Frelinghuysen Road, Piscataway, NJ 08854-8019}
\email{feehan@math.rutgers.edu}

%TODO Remove hours and minutes for arXiv and journal versions and fix date
%\date{\today{ }\hhmm}
\date{This version: April 23, 2020, incorporating final galley proof corrections. \emph{Transactions of the American Mathematical Society}, in press.}

\begin{abstract}
We develop strong and weak maximum principles for boundary-degenerate elliptic and parabolic linear second-order partial differential operators, $Au := -\tr(aD^2u)-\langle b, Du\rangle + cu$, with \emph{partial Dirichlet boundary conditions}. The coefficient, $a(x)$, is assumed to vanish along a non-empty open subset, $\partial_0\sO$, called the \emph{degenerate boundary portion}, of the boundary, $\partial\sO$, of the domain $\sO\subset\RR^d$, while $a(x)$ is non-zero at any point of the \emph{non-degenerate boundary portion}, $\partial_1\sO := \partial\sO\less\overline{\partial_0\sO}$. If an $A$-subharmonic function, $u$ in $C^2(\sO)$ or $W^{2,d}_{\loc}(\sO)$, is $C^1$ up to $\partial_0\sO$ and has a strict local maximum at a point in $\partial_0\sO$, we show that $u$ can be perturbed, by the addition of a suitable function $w\in C^2(\sO)\cap C^1(\RR^d)$, to a strictly $A$-subharmonic function $v=u+w$ having a local maximum in the interior of $\sO$. Consequently, we obtain strong and weak maximum principles for $A$-subharmonic functions in $C^2(\sO)$ and $W^{2,d}_{\loc}(\sO)$ which are $C^1$ up to $\partial_0\sO$. Points in $\partial_0\sO$ play the same role as those in the interior of the domain, $\sO$, and only the non-degenerate boundary portion, $\partial_1\sO$, is required for boundary comparisons. Moreover, we obtain a comparison principle for a solution and supersolution in $W^{2,d}_{\loc}(\sO)$ to a unilateral obstacle problem defined by $A$, again where only the non-degenerate boundary portion, $\partial_1\sO$, is required for boundary comparisons. Our results extend those in \cite{DaskalHamilton1998, Epstein_Mazzeo_2013, Feehan_maximumprinciple, Feehan_parabolicmaximumprinciple}, where $\tr(aD^2u)$ is in addition assumed to be continuous up to and vanish along $\partial_0\sO$ in order to yield comparable maximum principles for $A$-subharmonic functions in $C^2(\sO)$, while the results developed here for $A$-subharmonic functions in $W^{2,d}_{\loc}(\sO)$ are entirely new. Finally, we obtain analogues of all the preceding results for parabolic linear second-order partial differential operators, $Lu := -u_t - \tr(aD^2u)-\langle b, Du\rangle + cu$.
\end{abstract}

% AMS 2010 subject classifications (used in AMS journals)
%
% 35B50  	Maximum principles
% 35B51  	Comparison principles
% 35D40  	Viscosity solutions
% 35J70  	Degenerate elliptic equations
% 35J86  	Linear elliptic unilateral problems and linear elliptic variational inequalities
% 35K65  	Degenerate parabolic equations
% 35K85  	Linear parabolic unilateral problems and linear parabolic variational inequalities
% 35R35  	Free boundary problems
% 35R45  	Partial differential inequalities
%
% 47J20  	Variational and other types of inequalities involving nonlinear operators (general)
%
% 49J20  	Optimal control problems involving partial differential equations
% 49J40  	Variational methods including variational inequalities
%
% 60J60  	Diffusion processes

\subjclass[2010]{Primary 35B50, 35B51, 35J70, 35K65; secondary 35D40, 35J86, 35K85}

% AMS keywords (used in AMS journals)
\keywords{Comparison principles, boundary-degenerate elliptic differential operators, boundary-degenerate parabolic differential operators, maximum principles, obstacle problems, viscosity solutions}

% Acknowledge support
\thanks{The author was partially supported by NSF grant DMS-1237722 and a visiting faculty appointment in the Department of Mathematics at Columbia University.}

\maketitle
%TODO Remove both for journal version
\tableofcontents
\listoffigures

\section{Introduction}
\label{sec:Introduction}
%TODO Remove ``classical'' unless we say what it means
The classical weak maximum principle for an elliptic, possibly degenerate, linear, second-order partial differential operator in non-divergence form, $Au=-\tr(aD^2u)-\langle b, Du\rangle + cu$, provides uniqueness of solutions, $u$, to boundary value problems on an open subset $\sO\subset\RR^d$ with Dirichlet condition prescribed on the \emph{full boundary}, $\partial\sO$, when the solutions belong to $C^2(\sO)$ or $W^{2,d}_{\loc}(\sO)$ \cite{GilbargTrudinger}, or $C(\sO)$ if interpreted in the viscosity sense \cite{Crandall_Ishii_Lions_1992}. As noted\footnote{The idea appears to owe its origin to the article \cite{Keldys_1951} by M. V. Keldy{\v{s}}.}
by G. Fichera \cite{Fichera_1956, Fichera_1960} and O. A. Ole{\u\i}nik and E. V. Radkevi{\v{c}} \cite{Oleinik_Radkevic, Radkevich_2009a, Radkevich_2009b}, one can obtain uniqueness of solutions to boundary value problems with Dirichlet condition prescribed only along a \emph{part of the boundary}, $\partial_1\sO:=\partial\sO\less\overline{\partial_0\sO}$, for some non-empty, open subset $\partial_0\sO\subseteqq\partial\sO$, when the coefficient $a(x)$ vanishes along $\partial_0\sO$ (we call such an operator, $A$, \emph{boundary-degenerate}) and the \emph{Fichera function}\footnote{Namely, $\fb := (b^k-a^{kj}_{x_j})n_k$, where $(n_1,\ldots,n_d)$ is the inward-pointing unit normal vector field along $\partial_0\sO$ \cite[Equation (1.1.3)]{Radkevich_2009a}.}
$\fb$ defined by $A$ and $\partial_0\sO$
obeys the Fichera sign condition \cite[p. 308]{Radkevich_2009a} along $\partial_0\sO$. When the operator, $A$, is given in divergence form, so one can define a weak solution, $u \in W^{1,2}(\sO)$, to a boundary value problem, one can also obtain uniqueness of solutions with partial Dirichlet data when the Fichera sign condition holds along $\partial_0\sO$ \cite{Fichera_1956, Fichera_1960, Oleinik_Radkevic, Radkevich_2009a, Radkevich_2009b}. See Figure \ref{fig:domain}.

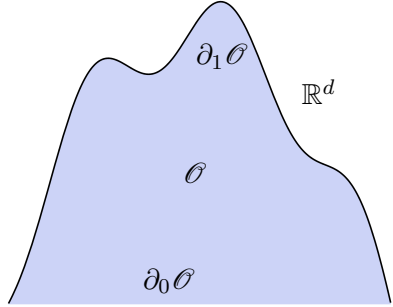
\begin{figure}
 \centering
 \begin{picture}(210,210)(0,0)
 \put(0,0){\includegraphics[scale=0.6]{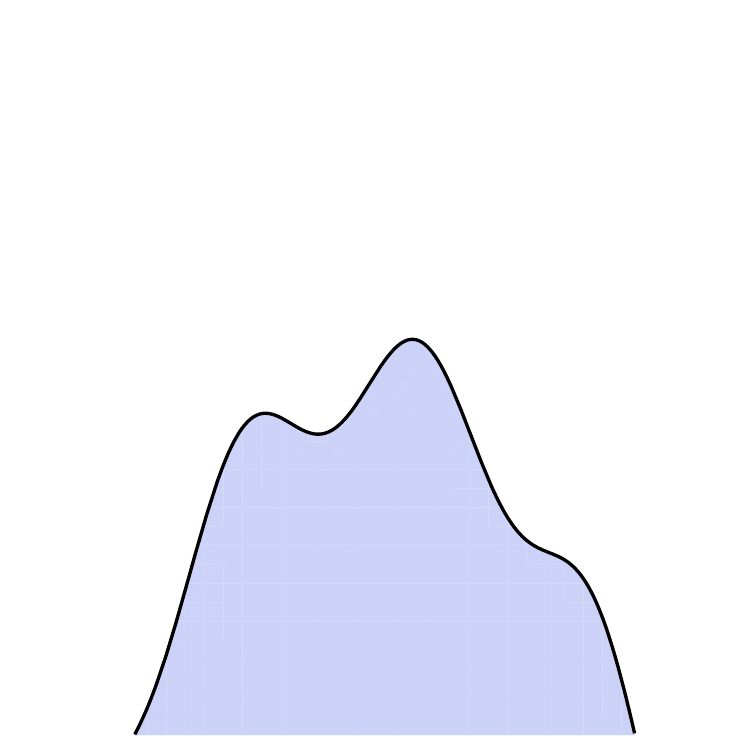}}
 \put(90,10){$\partial_0\sO$}
 \put(105,50){$\sO$}
 \put(110,95){$\partial_1\sO$}
 \put(150,80){$\RR^d$}
 \end{picture}
 \caption[A domain and its `degenerate' and `non-degenerate' boundaries.]{A subdomain, $\sO\subset\RR^d$, and its `degenerate' and `non-degenerate' boundaries, $\partial_0\sO$ and $\partial_1\sO$. In maximum and comparison principles, the degenerate boundary portion, $\partial_0\sO$, plays the same role as the interior of the domain, $\sO$.}
 \label{fig:domain}
\end{figure}

However, as observed by P. Daskalopoulos and R. Hamilton \cite{DaskalHamilton1998} in their work on the porous medium equation, when the solution, $u$, has sufficient regularity up to the degenerate boundary portion, $\partial_0\sO$, one obtains uniqueness of solutions to a boundary value problem with Dirichlet condition prescribed only along $\partial_1\sO$, \emph{irrespective} of whether or not the Fichera condition is obeyed along $\partial_0\sO$. Indeed, upon closer examination, the relevant regularity condition is that $u$ belong to $C^2_s(\underline\sO)$, namely (see \cite{Feehan_maximumprinciple})
\begin{subequations}
\label{eq:C2s_elliptic}
\begin{gather}
\label{eq:C2s_elliptic_C2interior_C1uptoboundary}
u \in C^2(\sO)\cap C^1(\underline\sO),
\\
\label{eq:C2s_elliptic_traD2u_continuous_upto_boundary}
\tr(aD^2u) \in C(\underline\sO), \quad\hbox{and}
\\
\label{eq:C2s_elliptic_traD2u_zero_boundary}
\tr(aD^2u) = 0 \quad\hbox{on } \partial_0\sO,
\end{gather}
\end{subequations}
where $\underline\sO = \sO\cup\partial_0\sO$.
This is the degenerate-boundary regularity condition for $u$ assumed in the weak maximum principles of Daskalopoulos and Hamilton \cite[Theorem I.3.1]{DaskalHamilton1998} and Epstein and Mazzeo \cite[Proposition 4.1.1]{Epstein_Mazzeo_annmathstudies}. Clearly, the condition $u \in C^2_s(\underline\sO)$ lies strictly between the conditions $u\in C^2(\sO)\cap C^1(\underline\sO)$ and $u \in C^2(\underline\sO)$. The `second-order boundary condition' \eqref{eq:C2s_elliptic_traD2u_zero_boundary} is a property of a function $u\in C^{2+\alpha}_s(\underline\sO)$ for $\alpha\in (0,1)$ (see \cite[pp. 901--902]{DaskalHamilton1998} for the definition of this H\"older space). However, the vanishing condition \eqref{eq:C2s_elliptic_traD2u_zero_boundary} is implied by the conditions \eqref{eq:C2s_elliptic_C2interior_C1uptoboundary} and \eqref{eq:C2s_elliptic_traD2u_continuous_upto_boundary} alone when $a(x)$ has suitable degeneracy and continuity properties near $\partial_0\sO$; see Lemma \ref{lem:Boundary_properties_C2s_functions}.

The significance of the conditions \eqref{eq:C2s_elliptic_C2interior_C1uptoboundary} and \eqref{eq:C2s_elliptic_traD2u_continuous_upto_boundary} is that they that ensure $Au \in C(\underline\sO)$ and the vanishing property \eqref{eq:C2s_elliptic_traD2u_zero_boundary} then allows one to deduce versions of the weak maximum principle for boundary-degenerate operators with relative ease --- at least for $u \in C^2_s(\underline\sO)$; the case $u \in W^{2,d}_{\loc}(\sO)\cap C^1(\underline\sO)$ is less straightforward. This paradigm, for $u \in C^2_s(\underline\sO)$, was developed further by the author in \cite{Feehan_maximumprinciple, Feehan_parabolicmaximumprinciple} for elliptic and parabolic boundary value and obstacle problems and closely related ideas have appeared in work of Daskalopoulos and the author \cite{Daskalopoulos_Feehan_statvarineqheston, Daskalopoulos_Feehan_optimalregstatheston} on the Heston stochastic volatility model \cite{Heston1993} in mathematical finance, the work of E. Ekstr\"om and J. Tysk \cite{Ekstrom_Tysk_bcsftse} on interest-rate models in mathematical finance, and the work of C. L. Epstein and R. Mazzeo \cite{Epstein_Mazzeo_2010, Epstein_Mazzeo_annmathstudies, Epstein_Mazzeo_2013} on Wright-Fisher diffusion models in mathematical biology. An important variant of this concept, for weak solutions to the porous medium equation, was developed by H. Koch \cite{Koch} and was explored further by the author in \cite{Feehan_maximumprinciple}.

However, it is a delicate question to determine the \emph{minimal regularity condition} on the solution, $u$, up to the degenerate boundary portion, $\partial_0\sO$, needed to ensure uniqueness when a Dirichlet boundary condition is only imposed along the non-degenerate boundary portion, $\partial_1\sO$. If we ask for too much regularity (such as $C^2$ up to the boundary), the boundary value problem may not be well-posed; if we ask for too little regularity (such as $C^0$ up to the boundary), we are led back to the imposition of a classical Dirichlet condition on the full boundary, without physical motivation provided by applications or that limits the solution to be at most continuous up to the boundary. As we have hinted, the case of $u \in W^{2,d}_{\loc}(\sO)\cap C^1(\underline\sO)$ is more difficult, where we only have $\tr(aD^2u) \in L^d_{\loc}(\sO)$ and thus suitable replacements for \eqref{eq:C2s_elliptic_traD2u_continuous_upto_boundary} and \eqref{eq:C2s_elliptic_traD2u_zero_boundary} are less obvious: fortunately, as we shall see, the latter two conditions are actually \emph{redundant} when the coefficients $a$ and $b$ of $A$ satisfy some mild regularity conditions along $\partial_0\sO$.

Given an open subset $\sO\subset\RR^d$, we shall say
that\footnote{The terminology appears to be due to Littman \cite{Littman_1959}.}
a function $u \in C^2(\sO)$ (respectively, $W^{2,d}_{\loc}(\sO)$) is \emph{(strictly) $A$-subharmonic} if $Au\leq 0$ (respectively, $Au < 0$) (a.e.) on $\sO$. Similarly, given an open subset $\sQ\subset\RR^{d+1}$, we shall say that a function $u \in C^2(\sQ)$ (respectively, $W^{2,d+1}_{\loc}(\sQ)$) is \emph{(strictly) $L$-subharmonic} if $Lu\leq 0$ (respectively, $Lu < 0$) (a.e.) on $\sQ$.  (The parabolic function space notation is explained in \S \ref{subsec:Parabolic_sobolev_embedding}.)

The purpose of this article is to explain how one may weaken the degenerate-boundary regularity condition for an $A$-subharmonic function $u \in C^2_s(\underline\sO)$ (as in \cite{Feehan_maximumprinciple, Feehan_parabolicmaximumprinciple}) to $C^2(\sO)\cap C^1(\underline\sO)$ at the expense of imposing mild regularity conditions on the coefficients $a$ and $b$ of $A$ near $\partial_0\sO$. The result will be a collection of useful strong and weak maximum principles, complementing those in \cite{Feehan_maximumprinciple, Feehan_parabolicmaximumprinciple}, in the elliptic case for $A$-subharmonic functions in $C^2(\sO)\cap C^1(\underline\sO)$ and $W^{2,d}_{\loc}(\sO)\cap C^1(\underline\sO)$ and in the parabolic case for $L$-subharmonic functions in $C^2(\sQ)\cap \sC^1(\underline\sQ)$ and $W^{2,d+1}_{\loc}(\sQ)\cap \sC^1(\underline\sQ)$ (see Definition \ref{defn:C1_C2_function_parabolic}). There are no analogues of the maximum principles that we develop in this article for $A$-subharmonic functions in $W^{2,d}_{\loc}(\sO)\cap C^1(\underline\sO)$ or $L$-subharmonic functions in $W^{2,d+1}_{\loc}(\sQ)\cap \sC^1(\underline\sQ)$.

Furthermore, we obtain comparison principles for solutions and supersolutions in $W^{2,d}_{\loc}(\sO)$ or $W^{2,d+1}_{\loc}(\sQ)$ to obstacle problems for boundary-degenerate elliptic or parabolic operators, respectively, in non-divergence form and with Dirichlet boundary conditions prescribed only along $\partial_1\sO$ or $\mydirac_1\sQ$. This application is particularly appealing since the optimal interior regularity of solutions to obstacle problems is $C^{1,1}(\sO)=W^{2,\infty}_{\loc}(\sO)$ in the elliptic case or $C^{1,1}(\sQ)=W^{2,\infty}_{\loc}(\sQ)$ in the parabolic case, so maximum principles for $A$ or $L$-subharmonic functions in $C^2(\sO)$ or $C^2(\sQ)$ are not directly applicable.

To achieve these enhanced maximum principles, we shall use two key ideas which appear fundamentally more difficult than those required for $u$ in $C^2_s(\underline\sO)$ (respectively, $C^2_s(\underline\sQ)$):
\begin{enumerate}
\item A Hopf boundary point lemma developed by the author in \cite{Feehan_maximumprinciple} (respectively, \cite{Feehan_parabolicmaximumprinciple}) for a boundary-degenerate elliptic (respectively, parabolic) linear, second-order partial differential operator, and
\item A construction in this article of a perturbation of an $A$-subharmonic (respectively, $L$-subharmonic) function, $u$, having a strict local maximum at point in the degenerate boundary portion, $\partial_0\sO$ (respectively, $\mydirac_0\!\sQ)$), by a function $w \in C^2(\sO)\cap C^1(\RR^d)$ (respectively, $C^2(\sQ)\cap \sC^1(\RR^{d+1})$) such that $v=u+w$ is a strictly $A$-subharmonic (respectively, $L$-subharmonic) function with a local maximum in the interior, $\sO$ (respectively, $\sQ$).
\end{enumerate}
This perturbation construction leads to a contradiction to the classical maximum principle and we then obtain the desired conclusions.

While the focus of this article is on the development of strong and weak maximum principles for subsolutions in $C^2(\sO)$ or $W^{2,d}_{\loc}(\sO)$ to \emph{linear} boundary-degenerate elliptic equations in non-divergence form (and similarly for boundary-degenerate parabolic equations), it appears likely that our approach can be extended to give comparison principles for viscosity subsolutions and supersolutions to \emph{fully nonlinear} equations on $\sO$ with fully nonlinear boundary conditions imposed only on $\partial_1\sO$, provided the concept of viscosity solution \cite[Definition 2.2]{Crandall_Ishii_Lions_1992} is appropriately modified. Similarly, one can expect analogues of the classical Aleksandrov-Bakelman-Pucci estimates (compare \cite[Theorem 9.1]{GilbargTrudinger} for functions in $W^{2,d}_{\loc}(\sO)$ or \cite[Theorem 5.8]{Han_Lin_2011} for viscosity subsolutions), where the role of the full boundary, $\partial\sO$, would be replaced by the non-degenerate boundary portion, $\partial_1\sO$. These ideas will be developed in a separate article.

\subsection{Motivation from one-dimensional examples}
\label{subsec:Motivation}
We shall illustrate the essence of the preceding ideas with the aid of the homogeneous \emph{Kummer ordinary differential equation} \cite[\S 13]{Olver_Lozier_Boisvert_Clark} on the interval $(0,\ell)$, for $0<\ell<\infty$,
\begin{equation}
\label{eq:Kummer_homogeneous_differential_equation}
Au(x) := -xu''(x) - (b-x)u'(x) + au(x) = 0, \quad x\in (0, \ell),
\end{equation}
where $a\geq 0$ and $b \geq 0$ are constants. Equation \eqref{eq:Kummer_homogeneous_differential_equation} has two fundamental solutions, $M(a,b,x)$ and $U(a,b,x)$, called \emph{confluent hypergeometric functions}. Corollary \ref{cor:Confluent_hypergeometric_function_boundary_regularity} says that $M \in C^\infty[0,\infty)$, when $b>0$, while $U\in C[0,\infty)$ and $U\notin C^1[0,\infty)$ when $a>0$ and $b>0$. Therefore, requiring that $u \in C^2(0,\ell)\cap C^1[0,\ell)\cap C[0,\ell]$ and $u(\ell)=0$ ensures that if $u$ solves \eqref{eq:Kummer_homogeneous_differential_equation}, then $u\equiv 0$ on $(0,\ell)$.

The boundary regularity condition \eqref{eq:C2s_elliptic} for $u$, namely $u\in C^2_s[0,\ell)$, means that $u\in C^2(0,\ell)\cap C^1[0,\ell)$ and
\begin{equation}
\label{eq:C2s_dimension_one_left_endpoint}
xu'' \in C[0,\ell).
\end{equation}
The vanishing property \eqref{eq:C2s_elliptic_traD2u_zero_boundary}, namely
$$
xu''(x) = 0 \quad\hbox{at } x=0
$$
is implied in this case by the conditions \eqref{eq:C2s_elliptic_C2interior_C1uptoboundary} and \eqref{eq:C2s_elliptic_traD2u_continuous_upto_boundary}; see Lemma \ref{lem:Boundary_properties_C2s_functions}. When $b>0$, a simple maximum principle argument shows that if $a\geq 0$ and $Au<0$ on $[0,\ell)$ or $a > 0$ and $Au\leq 0$ on $[0,\ell)$, then $u \in C^2_s[0,\ell)$ cannot have a positive local maximum at $x=0$. For the case $a>0$, taking the limit of Equation \eqref{eq:Kummer_homogeneous_differential_equation} as $x\downarrow 0$ yields
$$
-bu'(0) + au(0) \leq 0,
$$
or $u'(0) \geq au(0)/b > 0$, a contradiction. Thus, one obtains uniqueness of solutions, $u \in C^2_s[0,\ell)\cap C[0,\ell]$, to \eqref{eq:Kummer_homogeneous_differential_equation}, without relying on the asymptotic expansions for $u(x)$ as $x\downarrow 0$ implicit in the assertions of Corollary \ref{cor:Confluent_hypergeometric_function_boundary_regularity}.

The distinction between the boundary regularity conditions, $u \in C^2(0,\ell)\cap C^1[0,\ell)$ and $u \in C^2_s[0,\ell)$, is of course sharp, as elementary examples illustrate (see Appendix \ref{sec:Properties_confluent_hypergeometric_functions}). The main contribution of this article is to develop a broader theory for maximum principles based on the phenomena illustrated by the Kummer equation in dimension one: \emph{the boundary regularity condition \eqref{eq:C2s_dimension_one_left_endpoint} is not required for maximum principles for $A$-subharmonic functions $u \in C^2(0,\ell)\cap C^1[0,\ell)\cap C[0,\ell]$ to \eqref{eq:Kummer_homogeneous_differential_equation}}. We note that the coefficients of $u''$ and $u'$ of $A$ in \eqref{eq:Kummer_homogeneous_differential_equation} are smooth up to $x=0$.

The \emph{hypergeometric differential equation} on the interval $(0,1)$ (see \cite[\S 15.10]{Olver_Lozier_Boisvert_Clark}),
\begin{equation}
\label{eq:Hypergeometric_homogeneous_differential_equation}
Au(x) := -x(1-x)u''(x) - (c-(a+b+1)x)u'(x) + abu(x) = 0, \quad x\in (0, 1),
\end{equation}
where $a,b,c \in \RR$ are constants, provides another example. When none of $c$, $c-a-b$, and $a-b$ are integer, Equation \eqref{eq:Hypergeometric_homogeneous_differential_equation} has fundamental solutions $f_1, f_2$ \cite[\S 15.10]{Olver_Lozier_Boisvert_Clark} which may be expressed in terms of hypergeometric functions \cite[\S 15.1 and \S 15.2]{Olver_Lozier_Boisvert_Clark}. However, the complicated dependence of the functions $f_1, f_2$ on the parameters $a,b,c$ make any explicit analysis (comparable to that for the confluent hypergeometric functions) of their asymptotic behavior as $x\downarrow 0$ or $x\uparrow 1$ very difficult, if not impossible, except for simple choices of the parameters.

Nevertheless, as a consequence of the weak maximum principles developed in this article, we also obtain uniqueness of a solution, $u\in C^2(0,1)\cap C^1[0,1]$, to \eqref{eq:Hypergeometric_homogeneous_differential_equation} without imposing any boundary conditions at $x=0$ or $x=1$, for suitable parameters, $a,b,c$. In this case, the boundary regularity condition \eqref{eq:C2s_elliptic} for $u$, namely $u\in C^2_s[0,1]$, means that $u\in C^2(0,1)\cap C^1[0,1]$ and
\begin{equation}
\label{eq:C2s_dimension_one_both_endpoints}
xu'' \in C[0,1],
\end{equation}
and thus, by Lemma \ref{lem:Boundary_properties_C2s_functions},
$$
xu''(x) = 0 \quad\hbox{at } x = 0 \quad\hbox{and}\quad (1-x)u''(x) = 0 \quad\hbox{at } x = 1.
$$
Our uniqueness result for \eqref{eq:Hypergeometric_homogeneous_differential_equation} is stronger than those suggested by the proofs of \cite[Theorem I.3.1]{DaskalHamilton1998} and \cite[Proposition 4.1.1]{Epstein_Mazzeo_annmathstudies} for a solution $u\in C^2_s[0,1]$, since we can omit the boundary regularity condition  \eqref{eq:C2s_dimension_one_both_endpoints} and instead achieve uniqueness of a solution $u\in C^2(0,1)\cap C^1[0,1]$.

Before outlining the main results of this article in more detail, we first discuss the elliptic and parabolic operators we shall consider, along with their associated boundary value and obstacle problems.

\subsection{Boundary value problems for second-order partial differential operators}
Let $\sS(d)\subset\RR^{d\times d}$ denote the subset of symmetric matrices and $\sS^+(d)\subset\sS(d)$ denote the subset of non-negative definite, symmetric matrices. We first discuss the elliptic problems.

\subsubsection{Elliptic operators and boundary problems}
Let $\sO\subset\RR^d$, $d\geq 1$, be an open, possibly unbounded subset with topological boundary $\partial\sO$. Given a map
%COMMENT We're now defining a,b,c initially on \sO only
$$
a:\sO \to \sS^+(d),
$$
we call
\begin{equation}
\label{eq:Degeneracy_locus_elliptic}
\partial_0\sO := \Int\left\{x\in\partial\sO: \lim_{\sO\ni y\to x}a(y) = 0\right\}
\end{equation}
the \emph{degenerate boundary portion} (slightly abusing terminology) defined by $a:\sO \to \sS^+(d)$, where $\Int X$ denotes the interior of a subset $X$ of a topological space. Throughout this article, we shall allow $\partial_0\sO$ to be non-empty and denote
\begin{equation}
\label{eq:Domain_plus_degenerate_boundary_elliptic}
\underline\sO := \sO\cup\partial_0\sO,
\end{equation}
and call
\begin{equation}
\label{eq:Nondegeneracy_locus_elliptic}
\partial_1\sO := \partial\sO \less \overline{\partial_0\sO} = \Int\left\{x\in\partial\sO: \lim_{\sO\ni x'\to x}a(x') \neq 0\right\},
\end{equation}
the \emph{non-degenerate boundary portion} defined by $a:\sO \to \sS^+(d)$.

Given a vector field
%COMMENT We define a,b,c initially on \sO only
$b:\sO\to\RR^d$, and a function $c:\sO\to [0,\infty)$, we shall derive maximum principles for the operator,
%COMMENT We could just define c this way once and for all and make an assumption so we don't have to keep repeating?
\begin{equation}
\label{eq:Generator}
Au := -\tr(aD^2u) - \langle b,Du\rangle + cu,
\end{equation}
where $D^2u$ and $Du$ denote the Hessian matrix and gradient of a suitably regular function $u$ on $\sO$, respectively, so $\tr(aD^2u) = a^{ij}u_{x_ix_j}$ and $\langle b,Du\rangle = b^iu_{x_i}$, where $x=(x_1,\ldots,x_d)$ and Einstein's summation convention is used. We suppose that the coefficients, $a,b,c$, are defined on $\sO$ in the case of maximum principles for $A$-subharmonic functions in $C^2(\sO)$ and are measurable and defined a.e. on $\sO$ in the case of maximum principles for those in $W^{2,d}_{\loc}(\sO)$.

In older literature, $A$ in \eqref{eq:Generator} is called an elliptic linear second-order partial differential operator with non-negative characteristic form \cite{Oleinik_Radkevic}. We say that $A$ is \emph{boundary degenerate} when $\partial_0\sO$ is non-empty, recalling that the term `degenerate'\footnote{A fully nonlinear equation, $F(\cdot,u,Du,D^2u)=0$ on $\sO$, is \emph{degenerate elliptic} in the terminology of \cite[Equation (0.3)]{Crandall_Ishii_Lions_1992} if $F(x,r,p,X)$ is a decreasing function of $X\in\sS(d)$, the vector subspace of symmetric matrices in $\RR^{d\times d}$. When $F$ is also linear, that is, $F(x,r,p,X) = -\tr(a(x)X)-\langle b(x),p\rangle+c(x)$, the equation $F=0$ on $\sO$ is degenerate elliptic if and only if $a(x)\geq 0$ for all $x\in\sO$, and thus also degenerate elliptic if $F$ is uniformly elliptic on $\sO$ in the terminology of \cite[Example 1.2]{Crandall_Ishii_Lions_1992}, so $\lambda_0I\leq a \leq \Lambda_0I$ on $\sO$ (where $I$ is the identity matrix and $\lambda_0, \Lambda_0$ are positive constants) or uniformly and strictly elliptic in the terminology of \cite[p. 31]{GilbargTrudinger}.}
has a different meaning in \cite[p. 2]{Crandall_Ishii_Lions_1992}, where an operator which is strictly elliptic is merely a particular type of degenerate-elliptic operator.

We shall consider the question of uniqueness of solutions to the elliptic equation,
\begin{equation}
\label{eq:Elliptic_equation}
Au = f \quad \hbox{(a.e.) on }\sO,
\end{equation}
and obstacle problem,
\begin{equation}
\label{eq:Elliptic_obstacle_problem}
\min\{Au-f, \ u-\psi\} = 0 \quad \hbox{a.e. on }\sO,
\end{equation}
with \emph{partial} Dirichlet boundary condition,
\begin{equation}
\label{eq:Elliptic_boundary_condition}
u = g \quad \hbox{on } \partial_1\sO,
\end{equation}
for $u$, given a suitably regular source function $f$ on $\sO$, boundary data $g$ on $\partial_1\sO$, and obstacle function $\psi$ on $\sO$ which is compatible with $g$ in the sense that
\begin{equation}
\label{eq:Elliptic_boundarydata_obstacle_compatibility}
\psi\leq g \quad\hbox{on } \partial_1\sO.
\end{equation}
Next, we discuss parabolic versions of these questions.

\subsubsection{Parabolic operators and boundary problems}
\label{subsubsec:Parabolic_operator_boundary_problem}
For the reader's convenience, we recall the background from \cite[Section 1.1]{Feehan_parabolicmaximumprinciple}. Consider a possibly non-cylindrical open subset $\sQ\subset\RR^{d+1}$ with topological boundary $\partial \sQ$, where $d\geq 1$. Given $P^0=(t^0,x^0)\in\RR^{d+1}$ and $R>0$, define
%COMMENT Note that t^0 excluded, like Lieberman
\begin{equation}
\label{eq:Parabolic_cylinder}
Q_R(P^0) := \left\{(t,x)\in \RR^{d+1}: \max\left\{|x-x^0|, \, |t-t^0|^{1/2}\right\} < R, \ t>t^0\right\},
\end{equation}
where our time convention is opposite to that of Lieberman \cite[p. 5]{Lieberman} since we consider terminal rather than initial boundary problems in this article. Following \cite[p. 7]{Lieberman}, we call
\begin{equation}
\label{eq:Parabolic boundary}
\mydirac \sQ := \{P^0=(t^0,x^0)\in\partial \sQ: Q_\eps(P^0)\cap \sQ \neq \emptyset, \ \forall\,\eps > 0\}
\end{equation}
the \emph{parabolic boundary} of $\sQ$.

We identify the vector spaces $\RR$ and $\RR^d$ with the hyperplanes $\RR\times\{0\}$ and $\{0\}\times\RR^d\subset\RR^{d+1}$ of temporal and spatial vectors, respectively. When the boundary of $\sQ$ has a tangent plane at a point $P \in \partial \sQ$, we write the \emph{inward}-pointing  normal vector as $n_0(P)e_0 + \vec n(P)\in\RR^{d+1}$, where $\vec n(P) = \sum_{i=1}^dn^i(P)e_i$ and $e_0,e_1,\ldots,e_d$ is the standard basis of $\RR^{d+1}$. Given a map
%COMMENT We define a,b,c initially on \sQ only
$$
a:\sQ \to \sS^+(d),
$$
we call
%COMMENT Note that we could replace \partial \sQ below by \mydirac\!\sQ.
\begin{equation}
\label{eq:Degeneracy_locus_parabolic}
\mydirac_0\!\sQ := \Int\left\{P\in\partial \sQ: \lim_{\sO\ni P'\to P}a(P') = 0\right\}\cap \Int\{P\in\partial \sQ: n_0(P)=0\},
\end{equation}
the \emph{degenerate parabolic boundary} (again slightly abusing terminology) defined by $a:\sQ \to \sS^+(d)$. Throughout this article we shall allow $\mydirac_0\!\sQ$ to be non-empty and denote
\begin{equation}
\label{eq:Domain_plus_degenerate_boundary_parabolic}
\underline \sQ := \sQ\cup\mydirac_0\!\sQ.
\end{equation}
We also call
%COMMENT Note that we could replace \partial \sQ below by \mydirac\!\sQ.
\begin{equation}
\label{eq:Parabolic_nondegeneracy_locus}
\mydirac_1\!\sQ := \Int\left\{P\in\partial \sQ: \lim_{\sO\ni P'\to P}a(P') \neq 0\right\}\cup \Int\{P\in\partial \sQ: n_0(P)\neq 0\}
\end{equation}
the \emph{non-degenerate parabolic boundary} defined by $a:\sQ \to \sS^+(d)$.

\begin{exmp}[Boundaries for parabolic cylinders]
\label{exmp:Boundary_cylinder_parabolic}
\cite[Example 1.3]{Feehan_parabolicmaximumprinciple}
Given a parabolic cylinder, $\sQ=(0,T)\times\sO=\sO_T$, for some $T>0$ and open subset $\sO\subset\RR^d$, then
\begin{align*}
\mydirac \sQ &= \left(\{T\}\times\partial\sO\right) \cup \left(\{T\}\times\sO\right) \cup \left((0,T)\times\partial\sO\right)
\\
&= \left(\{T\}\times\bar\sO\right) \cup \left((0,T)\times\partial\sO\right),
\end{align*}
where (in the terminology of \cite[p. 7]{Lieberman}), the subset $\{T\}\times\sO$ is the \emph{top} of $\sQ$, and $(0,T)\times\partial\sO$ is the \emph{side} of $\sQ$, and $\{T\}\times\partial\sO$ is the \emph{corner} of $\sQ$.

Unlike its elliptic counterpart, we note that the non-degenerate boundary portion,
\begin{equation}
\label{eq:Nondegenerate_boundaryportion_nonempty_parabolic}
\mydirac_1\!\sQ \hbox{ is always non-empty}.
\end{equation}
For example, when $\sQ=(0,T)\times\sO$, then
$$
\mydirac_1\!\sQ \supset \{T\}\times\sO,
$$
since $n_0(P)=-1$ when $P\in \{T\}\times\sO$, and again keeping in mind our convention of considering terminal, rather than initial boundary problems, in this article because of their association with optimal stopping problems in probability theory.
Furthermore, when $a(t,x)$ is independent of $t\in\RR$ and we write $a(t,x)=a(x)$, for all $(t,x)\in\sQ$, then we can define $\partial_0\sO$ and $\partial_1\sO$ as in the elliptic case and observe that
$$
\mydirac_0\!\sQ = (0,T)\times\partial_0\sO
\quad\hbox{and}\quad
\underline \sQ = (0,T)\times (\sO\cup \partial_0\sO) = (0,T)\times\underline\sO = \underline\sO_T,
$$
while
$$
\mydirac_1\!\sQ = \left(\{T\}\times(\sO\cup\partial_1\sO)\right) \cup \left((0,T)\times\partial_1\sO\right).
$$
This concludes our example.
\end{exmp}

In the sequel, we shall allow $\sQ\subset\RR^{d+1}$ to be an arbitrary open subset. Given a vector field
$b:\sQ\to\RR^{d+1}$, and a function $c:\sQ\to \RR$, we shall derive maximum principles for the operator,
\begin{equation}
\label{eq:Generator_parabolic}
Lu := -u_t - \tr(aD^2u) - \langle b,Du\rangle + cu,
\end{equation}
where $D^2u$ and $Du$ denote the Hessian matrix and gradient of a suitably regular function $u$ on $\sQ$ with respect to the spatial coordinates, respectively. We suppose that the coefficients, $a,b,c$, are defined on $\sQ$ in the case of maximum principles for $L$-subharmonic functions in $C^2(\sQ)$ and are measurable and defined a.e. on $\sQ$ in the case of maximum principles for those in $W^{2,d+1}_{\loc}(\sQ)$.

In older literature, $L$ in \eqref{eq:Generator_parabolic} is called a parabolic linear second-order partial differential operator with non-negative characteristic form \cite{Oleinik_Radkevic}. We shall call $L$ \emph{boundary degenerate} when $\mydirac_0\!\sQ$ is non-empty, again noting the distinction between the way we use the term `degenerate' here and the sense in which this term is used in \cite{Crandall_Ishii_Lions_1992}, where an operator which strictly parabolic is merely a particular type of degenerate parabolic operator.

We shall consider the question of uniqueness of solutions to the parabolic equation,
\begin{equation}
\label{eq:Parabolic_equation}
Lu = f \quad \hbox{(a.e.) on }\sQ,
\end{equation}
and obstacle problem,
\begin{equation}
\label{eq:Parabolic_obstacle_problem}
\min\{Lu-f, \ u-\psi\} = 0 \quad \hbox{a.e. on }\sQ,
\end{equation}
with \emph{partial} Dirichlet boundary condition,
\begin{equation}
\label{eq:Parabolic_boundary_condition}
u = g \quad \hbox{on } \mydirac_1\!\sQ,
\end{equation}
for $u$, given a suitably regular source function $f$ on $\sQ$, boundary data $g$ on $\mydirac_1\!\sQ$, and obstacle function $\psi$ on $\sQ$ which is compatible with $g$ in the sense that
\begin{equation}
\label{eq:Parabolic_boundarydata_obstacle_compatibility}
\psi\leq g \quad\hbox{on } \mydirac_1\!\sQ.
\end{equation}
We now discuss some of the properties of the coefficients of the operators $A$ and $L$.

\subsection{Properties of the coefficients of the elliptic and parabolic operators}
\label{subsec:Properties_coefficients}
It is convenient to collect here most of the conditions which we shall impose in various sections of this article on the coefficients of the elliptic operator, $A$, in \eqref{eq:Generator} or parabolic operator, $L$, in \eqref{eq:Generator_parabolic}. However, we emphasize that when stating any specific result we \emph{only} impose those few selected conditions that are \emph{explicitly referenced}, unless it we state that the condition holds throughout the article.

\subsubsection{Properties of the coefficients of the elliptic operator}
\label{subsubsec:Properties_coefficients_elliptic}
Consider the coefficients of the elliptic operator, $A$, in \eqref{eq:Generator}. Let $\lambda(x)$ denote the smallest eigenvalue of the matrix, $a(x)$, for each $x\in\sO$, and let
$$
\lambda_*:\sO\to[0,\infty)
$$
be the lower semi-continuous envelope\footnote{When $f:X\to[0,\infty)$ is a measurable function on a measure space $(X,\Sigma,\mu)$, then $f_*:X\to[0,\infty)$ is the largest lower-semicontinuous function on $X$ such that $f_*\leq f$ $\mu$-a.e. on $X$.}
of the resulting least eigenvalue function, $\lambda:\sO\to\RR$, for $a:\sO\to\sS^+(d)$. To achieve certain intermediate results, we may require that $a:\sO\to\sS^+(d)$ be \emph{locally strictly elliptic on the interior}, $\sO$, in the sense that \cite[p. 31]{GilbargTrudinger}
\begin{equation}
\label{eq:a_locally_strictly_elliptic}
\lambda_* > 0 \quad\hbox{on } \sO \quad\hbox{(interior local strict ellipticity)}.
\end{equation}
We again remind the reader that conditions such as \eqref{eq:a_locally_strictly_elliptic} are not assumed unless explicitly referenced and, indeed, the main results of this article allow $\lambda_*\equiv 0$ on $\sO$.

Throughout the article we shall require that\footnote{It is likely that $C^1$ would suffice, but an assumption that $\partial_0\sO$ is a boundary portion of class $C^{1,\alpha}$ simplifies the proofs --- see Lemma \ref{lem:Straightening_boundary}.}
\begin{equation}
\label{eq:C1alpha_degenerate_boundary_elliptic}
\partial_0\sO \quad\hbox{is } C^{1,\alpha},
\end{equation}
for some $\alpha\in(0,1)$ and let $\vec n$ denote the \emph{inward}-pointing unit normal vector field along $\partial_0\sO$.

The vector field, $\vec n:\partial_0\sO\to\RR^d$, may be extended to a tubular neighborhood $N(\partial_0\sO)$ of $\partial_0\sO \subset \underline\sO$. We can then split the vector field, $b:N(\partial_0\sO)\to\RR^d$, into its normal and tangential components, with respect to the extended vector field, $\vec n:N(\partial_0\sO)\to\RR^d$, so
\begin{equation}
\label{eq:b_splitting_elliptic}
b^\perp := \langle b, \vec n \rangle \quad\hbox{and}\quad b^\parallel := b - b^\perp \vec n \quad\hbox{on } N(\partial_0\sO).
\end{equation}
We may require that the vector field $b^\perp$ obey one of the following conditions,
\begin{align}
\label{eq:b_perp_nonnegative_boundary_elliptic}
b^\perp &\geq 0 \quad \hbox{on } \partial_0\sO \quad\hbox{or}
\\
\label{eq:b_perp_positive_boundary_elliptic}
\tag{\ref*{eq:b_perp_nonnegative_boundary_elliptic}$'$}
b^\perp &> 0 \quad \hbox{on } \partial_0\sO.
\end{align}
Similarly, we may require that the function $c$ obey one of the following conditions,
\begin{align}
\label{eq:c_nonnegative_elliptic}
%COMMENT We need c \geq 0 on \underline\sO in C^2_s papers, but not C^1 papers
%TODO Commas before "or"
c &\geq 0 \quad \hbox{(a.e.) on } \sO \quad\hbox{or}
\\
\label{eq:c_lower*_positive_elliptic}
\tag{\ref*{eq:c_nonnegative_elliptic}$'$}
c_* &> 0 \quad \hbox{(a.e.) on } \sO,
\end{align}
where $c_*:\sO\to[0,\infty)$ is the lower semicontinuous envelope of $c$.

We may also require that $c$ is (essentially) locally bounded above on $\underline \sO$, that is, setting $c^+ = c \vee 0$ on $\sO$,
\begin{equation}
\label{eq:c_locally_bounded_above_near_boundary_elliptic}
%COMMENT The definition below does not imply that c is defined on \partial_0\sO itself.
c^+ \in L^\infty_{\loc}(\underline\sO),
\end{equation}
where we slightly abuse notation by writing $w\in L^\infty_{\loc}(\underline\sO)$ as an abbreviation for saying that $w$ is a locally bounded function on $\underline\sO$, irrespective of whether $w$ is measurable or everywhere-defined.
%COMMENT For the sake of simplicity, we assume bounded domains
We shall often require that the coefficient $a$ be continuous along $\partial_0\sO$,
\begin{equation}
\label{eq:a_continuous_degen_boundary_elliptic}
a \in C(\partial_0\sO;\sS^+(d)).
\end{equation}
We say that the coefficients, $a:N(\partial_0\sO)\to\sS^+(d)$ and $b^\parallel:N(\partial_0\sO)\to\RR^d$, are \emph{locally Lipschitz along $\partial_0\sO$} if, for each precompact open subset, $\Sigma\Subset\partial_0\sO$, there are positive constants, $K=K(\Sigma)$ and $\delta=\delta(\Sigma)$, such that
\begin{align}
\label{eq:a_locally_lipschitz_elliptic}
|a(x) - a(x^0)| &\leq K|x-x^0|,
\\
\label{eq:b_tangential_locally_lipschitz_elliptic}
|b^\parallel(x) - b^\parallel(x^0)| &\leq K|x-x^0|, \quad\hbox{for } x^0\in\Sigma \hbox{ and (a.e.) } x\in N_\delta(\Sigma),
\end{align}
where, for any $\Sigma\subseteqq\partial_0\sO$,
\begin{equation}
\label{eq:Tubular_neighborhood}
N_\delta(\Sigma) := \{x \in \underline\sO: \dist(x,\Sigma) < \delta\},
\end{equation}
and $\dist(x,x^0) := |x-x^0|$ denotes the Euclidean distance function between points in $\RR^d$. Since $a=0$ on $\partial_0\sO$ by definition \eqref{eq:Degeneracy_locus_elliptic}, then \eqref{eq:a_locally_lipschitz_elliptic} is equivalent to
$$
|a(x)| \leq K\dist(x,\Sigma), \quad\hbox{for (a.e.) } x\in N_\delta(\Sigma).
$$
Of course, the conditions \eqref{eq:a_locally_lipschitz_elliptic} and \eqref{eq:b_tangential_locally_lipschitz_elliptic} would be consequences of assumptions that $a \in C^{0,1}(\underline\sO;\sS^+(d))$ and $b^\parallel \in C^{0,1}(\underline\sO;\RR^d)$, but such conditions are stronger than one often wishes to impose.

\subsubsection{Properties of the coefficients of the parabolic operator}
\label{subsubsec:Properties_coefficients_parabolic}
Compare \cite[Section 1.2]{Feehan_parabolicmaximumprinciple} for a related description of coefficient properties. Consider the coefficients of the parabolic operator, $L$, in \eqref{eq:Generator_parabolic}. Let $\lambda(P)$ denote the smallest eigenvalue of the matrix, $a(P)$, for each $P\in\sQ$, and let
$$
\lambda_*:\sQ\to[0,\infty)
$$
be the lower semi-continuous envelope of the resulting least eigenvalue function, $\lambda:\sQ\to[0,\infty]$, for $a:\sQ\to\sS^+(d)$. To achieve certain intermediate results, we may require that $a:\sQ\to\sS^+(d)$ be \emph{locally strictly parabolic on the interior}, $\sQ$, in the sense that
\begin{equation}
\label{eq:a_locally_strictly_parabolic}
\lambda_* > 0 \quad\hbox{on } \sQ \quad\hbox{(interior local strict parabolicity)}.
\end{equation}
Throughout the article we shall require that\footnote{Again, it is likely that $C^1$ would suffice, but an assumption that $\mydirac_0\!\sQ$ is a boundary portion of class $C^{1,\alpha}$ simplifies the proofs --- see Lemma \ref{lem:Straightening_boundary}.}
\begin{equation}
\label{eq:C1alpha_degenerate_boundary_parabolic}
\mydirac_0\!\sQ \quad\hbox{is } C^{1,\alpha},
\end{equation}
and let $\vec n$ denote the \emph{inward}-pointing unit normal vector field along $\mydirac_0\!\sQ$.

The vector field, $\vec n:\mydirac_0\!\sQ\to\RR^d$, may be extended to a tubular neighborhood $N(\mydirac_0\!\sQ)$ of $\mydirac_0\!\sQ \subset \underline\sO$, recalling that $n_0=0$ along $\mydirac_0\!\sQ$ by definition \eqref{eq:Degeneracy_locus_parabolic}. We can then split the vector field, $b:N(\mydirac_0\!\sQ)\to\RR^d$, into its normal and tangential components, with respect to the extended vector field, $\vec n:N(\mydirac_0\!\sQ)\to\RR^d$, so
\begin{equation}
\label{eq:b_splitting_parabolic}
b^\perp := \langle b, \vec n \rangle \vec n \quad\hbox{and}\quad b^\parallel := b - b^\perp \quad\hbox{on } N(\mydirac_0\!\sQ).
\end{equation}
We may require that the vector field $b^\perp$ obey one of the following conditions,
\begin{align}
\label{eq:b_perp_nonnegative_boundary_parabolic}
b^\perp &\geq 0 \quad \hbox{on } \mydirac_0\!\sQ, \quad\hbox{or}
\\
\label{eq:b_perp_positive_boundary_parabolic}
\tag{\ref*{eq:b_perp_nonnegative_boundary_parabolic}$'$}
b^\perp &> 0 \quad \hbox{on } \mydirac_0\!\sQ.
\end{align}
Similarly, we may require that the function $c$ obey one of the following conditions,
\begin{align}
\label{eq:c_nonnegative_parabolic}
%COMMENT We need c \geq 0 on \underline\sQ in C^2_s papers, but not C^1 papers
c &\geq 0 \quad \hbox{(a.e.) on } \sQ \quad\hbox{or}
\\
\label{eq:c_lower*_positive_parabolic}
\tag{\ref*{eq:c_nonnegative_parabolic}$'$}
c_* &> 0 \quad \hbox{(a.e.) on } \sQ,
\end{align}
where $c_*:\sQ\to[0,\infty)$ is the lower semicontinuous envelope of $c$.

We may require that $c$ is (essentially) locally bounded above on $\underline \sQ$, that is, setting $c^+ = c \vee 0$ on $\sQ$,
\begin{align}
\label{eq:c_locally_bounded_above_domain_plus_degen_boundary_parabolic}
%COMMENT The definition below does not imply that c is defined on \mydirac_0\!\sQ itself.
c^+ &\in L^\infty_{\loc}(\underline\sQ).
\end{align}
We shall often require that the coefficient $a$ be continuous along $\mydirac_0\!\sQ$,
\begin{equation}
\label{eq:a_continuous_degen_boundary_parabolic}
a \in C(\mydirac_0\!\sQ;\sS^+(d)).
\end{equation}
We say that the coefficients, $a:N(\mydirac_0\!\sQ)\to\sS^+(d)$, respectively $b^\parallel:N(\mydirac_0\!\sQ)\to\RR^d$, are \emph{locally Lipschitz along $\mydirac_0\!\sQ$} if for each precompact open subset, $\Sigma\Subset\mydirac_0\!\sQ$, there are positive constants, $K=K(\Sigma)$ and $\delta=\delta(\Sigma)$, such that
%COMMENT Note choice of the factor |t-t^0|^{1/2} below to match our definition of C^{0,1}(\sQ)
\begin{align}
\label{eq:a_locally_lipschitz_parabolic}
|a(t, x) - a(t^0,x^0)| &\leq K\left(|t-t^0|^{1/2}+|x-x^0|\right),
\\
\label{eq:b_tangential_locally_lipschitz_parabolic}
|b^\parallel(t,x) - b^\parallel(t^0,x^0)| &\leq K\left(|t-t^0|^{1/2}+|x-x^0|\right),
\\
\notag
&\qquad\hbox{for all } (t^0,x^0)\in\Sigma \hbox{ and (a.e.) } (t,x)\in N_\delta(\Sigma),
\end{align}
where
$$
N_\delta(\Sigma) := \{P \in \underline \sQ: \dist_p(P,\Sigma) < \delta\},
$$
and $\dist_p((t,x),(t^0,x^0)) := |t-t^0|^{1/2}+|x-x^0|$ denotes the parabolic distance function between points in $\RR^{d+1}$. Since $a=0$ on $\mydirac_0\!\sQ$ by definition \eqref{eq:Degeneracy_locus_parabolic}, then \eqref{eq:a_locally_lipschitz_parabolic} is equivalent to
$$
|a(t,x)| \leq K\dist_p((t,x),\Sigma), \quad\hbox{for (a.e.) } (t,x)\in N_\delta(\Sigma).
$$
As in the elliptic case, the conditions \eqref{eq:a_locally_lipschitz_parabolic} and \eqref{eq:b_tangential_locally_lipschitz_parabolic} would be consequences of assumptions that $a \in C^{0,1}(\underline\sQ;\sS^+(d))$ and $b^\parallel \in C^{0,1}(\underline\sQ;\RR^d)$, but such conditions are stronger than one often wishes to impose.

\subsection{Summary of main results}
\label{subsec:Summary}
We begin with a statement of the principal technical tool used to prove the main results of this article for elliptic partial differential operators --- the concept of perturbing a strict local maximum point for an $A$-subharmonic function on the degenerate boundary portion to a nearby maximum point in the interior for a strictly $A$-subharmonic function obtained from the original by adding a suitable smooth function.

Given $1\leq p\leq\infty$, we let $W^{2,p}(\sO)$ (respectively, $W^{2,p}_{\loc}(\sO)$) denote the Sobolev space of measurable functions, $u$ on $\sO$, such that $u$ and its weak first and second derivatives, $u_{x_i}$ and $u_{x_ix_j}$ for $1\leq i,j\leq d$, belong to $L^p(\sO)$  (respectively, $L^p_{\loc}(\sO)$) \cite[\S 3.1]{Adams_1975}.

\begin{thm}[Perturbing a degenerate-boundary local maximum for an $A$-subharmonic function to an interior local maximum]
\label{thm:Viscosity_maximum_elliptic}
Let $\sO\subset\RR^d$ be an open subset and $A$ be as in \eqref{eq:Generator}. Assume that the coefficients of $A$ obey \eqref{eq:a_locally_strictly_elliptic}, \eqref{eq:c_nonnegative_elliptic}, \eqref{eq:b_perp_positive_boundary_elliptic}, \eqref{eq:c_locally_bounded_above_near_boundary_elliptic}, \eqref{eq:a_continuous_degen_boundary_elliptic}, \eqref{eq:a_locally_lipschitz_elliptic}, and \emph{one} of the following two conditions:
\begin{gather}
\label{eq:b_tangential_component_zero_degenerate_boundary_elliptic}
b^\parallel \hbox{ obeys }\eqref{eq:b_tangential_locally_lipschitz_elliptic}
\quad\hbox{and}\quad
b^\parallel = 0 \quad\hbox{on }\partial_0\sO, \quad\hbox{or}
\\
\label{eq:b_C2_elliptic}
\tag{\ref*{eq:b_tangential_component_zero_degenerate_boundary_elliptic}$'$}
b\in C^2(\underline\sO;\RR^d).
\end{gather}
Assume that $u \in C^2(\sO)\cap C^1(\underline\sO)$
(respectively,\footnote{Our hypothesis that $u\in C^1(\underline\sO)$ could be relaxed to the slightly more technical hypothesis that $Du$ be defined and continuous along $\partial_0\sO$.}
$u\in W^{2,d}_{\loc}(\sO)\cap C^1(\underline\sO)$). Finally, assume that $Au\leq 0$ (a.e.) on $\sO$. Suppose $u$ has a strict local maximum at $\bar x^0 \in \partial_0\sO$ relative to $\sO$, that is, $u<u(\bar x^0)$ on $B_\delta(\bar x^0)\cap\sO$ for a small enough $\delta>0$; we allow the sign of $u(\bar x^0)$ to be arbitrary when $c=0$ on $\sO$ and require that $u(\bar x^0)\geq 0$ when $c\geq 0$ on $\sO$. Then there are a connected open subset, $\sU\subset\RR^d$, such that $\bar x^0 \in \bar \sU \subset \underline\sO$ and a function $w\in C^2(\sO)\cap C^1(\RR^d)$, such that $w(\bar x^0)=0$ and $Aw < 0$ (a.e.) on $\sU$, so $v := u+w$ obeys
$$
Av < 0 \quad\hbox{(a.e.) on } \sU,
$$
and $v$ on $\bar \sU$ attains a maximum $v(x^0) > u(\bar x^0)$ at an interior point $x^0\in \sU$.
\end{thm}

\begin{rmk}[Relaxing the regularity condition on $b$]
\label{rmk:Viscosity_maximum_elliptic_relaxing_b_regularity}
While the proofs of the strong maximum principles (Theorem \ref{thm:Strong_maximum_principle_C2_elliptic} for $A$-subharmonic functions in $C^2(\sO)$ and \ref{thm:Strong_maximum_principle_W2d_elliptic} for $A$-subharmonic functions in $W^{2,d}_{\loc}(\sO)$) requires that $b$ obeys both \eqref{eq:b_tangential_locally_lipschitz_elliptic} and \eqref{eq:b_tangential_component_zero_degenerate_boundary_elliptic} or that $b$ obeys \eqref{eq:b_C2_elliptic}, these constraints can be removed in later versions of our weak maximum principles (Theorems \ref{thm:Weak_maximum_principle_C2_elliptic_domain_relaxed} and \ref{thm:Weak_maximum_principle_W2d_elliptic_domain_relaxed}).
\end{rmk}

\begin{rmk}[Structure of the perturbation]
\label{rmk:Viscosity_maximum_elliptic_w_structure}
In the special case that $\sO\subset\RR^{d-1}\times\RR_+$ and $\partial_0\sO = \Int\{\bar\sO\cap(\RR^{d-1}\times\{0\})\}$, the function $w$ in Theorem \ref{thm:Viscosity_maximum_elliptic} is a quadratic polynomial.
\end{rmk}

As we shall discover, Theorem \ref{thm:Viscosity_maximum_elliptic} --- when coupled with our version of the Hopf boundary-point Lemma \ref{lem:Degenerate_hopf_lemma_elliptic} for a boundary-degenerate elliptic operator $A$ --- provides a powerful tool for the development of strong maximum principles for $A$-subharmonic functions in $C^2(\sO)\cap C^1(\underline\sO)$ or $W^{2,d}_{\loc}(\sO)\cap C^1(\underline\sO)$; the proof of the strong maximum principle \cite[Theorem 4.8]{Feehan_maximumprinciple} for $A$-subharmonic functions in $C^2_s(\underline\sO)$ indicates that the Hopf boundary-point lemma alone is insufficient to obtain the enhanced versions of the strong maximum principle given in this article (Theorems \ref{thm:Strong_maximum_principle_C2_elliptic} and \ref{thm:Strong_maximum_principle_W2d_elliptic}).

We also have a parabolic analogue of Theorem \ref{thm:Viscosity_maximum_elliptic}; see \S \ref{subsec:Parabolic_sobolev_embedding} for definitions of spaces of continuous functions, H\"older spaces, and Sobolev spaces suitable for functions $u(t,x)$, where $(t,x) \in \RR^{d+1}$, in the context of parabolic partial differential equations.

\begin{thm}[Perturbing a degenerate-boundary local maximum for an $L$-subharmonic function to an interior local maximum]
\label{thm:Viscosity_maximum_parabolic}
Let $\sQ\subset\RR^{d+1}$ be an open subset and $L$ be as in \eqref{eq:Generator_parabolic}. Assume that the coefficients of $L$ obey \eqref{eq:a_locally_strictly_parabolic}, \eqref{eq:c_nonnegative_parabolic}, \eqref{eq:b_perp_positive_boundary_parabolic}, \eqref{eq:c_locally_bounded_above_domain_plus_degen_boundary_parabolic}, \eqref{eq:a_continuous_degen_boundary_parabolic}, \eqref{eq:a_locally_lipschitz_parabolic}, and \emph{one} of the following two conditions:
\begin{gather}
\label{eq:b_tangential_component_zero_degenerate_boundary_parabolic}
b^\parallel \hbox{ obeys }\eqref{eq:b_tangential_locally_lipschitz_parabolic}
\quad\hbox{and}\quad
b^\parallel = 0 \quad\hbox{on }\mydirac_0\!\sQ, \quad\hbox{or}
\\
\label{eq:b_C2_parabolic}
\tag{\ref*{eq:b_tangential_component_zero_degenerate_boundary_parabolic}$'$}
b\in C^2(\underline\sQ;\RR^d).
\end{gather}
Assume that $u \in C^2(\sQ)\cap \sC^1(\underline \sQ)$
(respectively,\,\footnote{Our hypothesis that $u\in \sC^1(\underline\sQ)$ could be relaxed to the slightly more technical hypothesis that $u_t$ and $Du$ be defined and continuous along $\mydirac_0\!\sQ$.}
$u\in W^{2,d+1}_{\loc}(\sQ)\cap \sC^1(\underline \sQ)$). Finally, assume that $Lu\leq 0$ (a.e.) on $ \sQ$. Suppose $u$ has a strict local maximum at $\bar P^0 \in \mydirac_0\!\sQ$ relative to $\sQ$, that is, $u<u(\bar P^0)$ on $B_\delta(\bar P^0)\cap\sQ$ for a small enough $\delta>0$; we allow the sign of $u(\bar P^0)$ to be arbitrary when $c=0$ on $\sQ$ and require that $u(\bar P^0)\geq 0$ when $c\geq 0$ on $\sQ$. Then there are a connected open subset, $\sV\subset\RR^{d+1}$, such that $\bar P^0 \in \bar\sV \subset \underline \sQ$ and a function $w\in C^2(\sQ)\cap C^1(\RR^{d+1})$ such that $w(\bar P^0)=0$ and $Lw < 0$ (a.e.) on $\sV$, so $v := u+w$ obeys
$$
Lv < 0 \quad\hbox{(a.e.) on } \sV,
$$
and $v$ on $\bar \sV$ attains a maximum $v(P^0) > u(\bar P^0)$ at an interior point $P^0\in \sV$.
\end{thm}

\begin{rmk}[Relaxing the regularity condition on $b$]
\label{rmk:Viscosity_maximum_parabolic_relaxing_b_regularity}
While the proofs of the strong maximum principles (Theorems \ref{thm:Strong_max_principle_C2_parabolic_c_geq_zero}, \ref{thm:Strong_max_principle_C2_parabolic_c_arb_sign}, \ref{thm:Strong_max_principle_C2_parabolic_c_geq_zero_refined}, and \ref{thm:Strong_max_principle_C2_parabolic_c_arb_sign_refined} for $L$-subharmonic functions in $C^2(\sQ)$ and Theorem \ref{thm:Strong_max_principle_W2d+1_parabolic_c_geq_zero} for $L$-subharmonic functions in $W^{2,d+1}_{\loc}(\sQ)$) requires that $b$ obeys both \eqref{eq:b_tangential_locally_lipschitz_parabolic} and \eqref{eq:b_tangential_component_zero_degenerate_boundary_parabolic} or that $b$ obeys \eqref{eq:b_C2_parabolic}, these constraints can be removed in later versions of our weak maximum principles (Theorems \ref{thm:Weak_maximum_principle_C2_parabolic_domain_relaxed} and \ref{thm:Weak_maximum_principle_W2d+1_parabolic_domain_relaxed}).
\end{rmk}

\begin{rmk}[Structure of the perturbation]
\label{rmk:Viscosity_maximum_parabolic_w_structure}
In the special case that $\sQ = (0,T)\times\sO$ and $\sO\subset\RR^{d-1}\times\RR_+$, so $\mydirac_0\!\sQ = (0,T)\times\partial_0\sO$, the function $w$ in Theorem \ref{thm:Viscosity_maximum_parabolic} is a quadratic polynomial.
\end{rmk}

We recall from \cite[Definition 1.1]{Feehan_parabolicmaximumprinciple} that $u$ belongs to $C^2_s(\underline\sQ)$ if
\begin{subequations}
\label{eq:C2s_parabolic}
\begin{gather}
\label{eq:C2s_parabolic_C2interior_C1uptoboundary}
u \in C^2(\sQ)\cap \sC^1(\underline\sQ),
\\
\label{eq:C2s_parabolic_traD2u_continuous_upto_boundary}
\tr(aD^2u) \in C(\underline\sQ), \quad\hbox{and}
\\
\label{eq:C2s_parabolic_traD2u_zero_boundary}
\tr(aD^2u) = 0 \quad\hbox{on } \mydirac_0\!\sQ,
\end{gather}
\end{subequations}
Again, the vanishing condition \eqref{eq:C2s_parabolic_traD2u_zero_boundary} is implied by the conditions \eqref{eq:C2s_parabolic_C2interior_C1uptoboundary} and \eqref{eq:C2s_parabolic_traD2u_continuous_upto_boundary} alone when $a(t,x)$ has suitable degeneracy and continuity properties near $\mydirac_0\!\sQ$; see Lemma \ref{lem:Boundary_properties_C2s_functions}.

Theorem \ref{thm:Viscosity_maximum_parabolic} --- when coupled with our version of the Hopf boundary-point Lemma \ref{lem:Degenerate_hopf_lemma_parabolic} for a boundary-degenerate parabolic operator $L$ --- provides a powerful tool for the development of strong maximum principles for $L$-subharmonic functions in $C^2(\sQ)\cap \sC^1(\underline\sQ)$ or $W^{2+1,d}_{\loc}(\sQ)\cap \sC^1(\underline\sQ)$; as in the elliptic case, the proof of the strong maximum principles \cite[Theorems 5.13, 5.15, 5.16, and 5.17]{Feehan_parabolicmaximumprinciple} for $L$-subharmonic functions in $C^2_s(\underline\sQ)$ indicates that the Hopf boundary-point lemma alone is insufficient to obtain the enhanced versions of the strong maximum principle given in this article.

While the proofs of Theorems \ref{thm:Viscosity_maximum_elliptic} and \ref{thm:Viscosity_maximum_parabolic} do not appeal to any particular result from the theory of viscosity solutions, they evoke its philosophy and methods; compare \cite[Example 2.1 and Theorem 3.2]{Crandall_Ishii_Lions_1992}.

\subsection{Applications of Theorems \ref{thm:Viscosity_maximum_elliptic} and \ref{thm:Viscosity_maximum_parabolic} to strong and weak maximum principles and comparison principles for obstacle problems}
\label{subsec:Applications}
The purpose of Theorem \ref{thm:Viscosity_maximum_elliptic} is to transform the problem of analyzing the behavior of an $A$-subharmonic function in $C^2(\sO)$ or $W^{2,d}_{\loc}(\sO)$, near the degenerate boundary portion, $\partial_0\sO$, to the far simpler one of analyzing its behavior in the interior, $\sO$ --- at least near a point on the degenerate boundary portion, $\partial_0\sO$, where the $A$-subharmonic function has a strict local maximum. The proof of Theorem \ref{thm:Viscosity_maximum_elliptic} in turn relies in an essential way on our Hopf boundary-point lemma for the operator, $A$, at such a point \cite[Lemma 4.1]{Feehan_maximumprinciple}. We shall state here just a sample of the strong and weak maximum principles which quickly follow from Theorem \ref{thm:Viscosity_maximum_elliptic}, but refer the reader to sections \ref{subsec:Strong_maximum_principle_elliptic}, \ref{subsec:Weak_maximum_principle_elliptic}, and \ref{subsec:Relaxing_coefficients_elliptic} for many other variants.

We first state an analogue of the classical strong maximum principle for $A$-subharmonic functions in $W^{2,d}_{\loc}(\sO)$ due to Aleksandrov \cite[Theorem 9.6]{GilbargTrudinger}, Bony \cite[Theorem 2]{Bony_1967},  Lions \cite{Lions_1983}, and others. Recall from the Sobolev Embedding \cite[Theorem 5.6, Part I (C)]{Adams_1975} that $W^{2,d}_{\loc}(\sO) \hookrightarrow C(\sO)$, but $W^{2,d}_{\loc}(\sO) \not\hookrightarrow C^1(\sO)$. Note however that Theorem \ref{thm:Strong_maximum_principle_W2d_elliptic} makes a stronger assertion than its classical counterparts since, even if $u$ is only known to attain its global maximum in the degenerate boundary portion, $\partial_0\sO$, we can still conclude that $u$ is constant on $\sO$.

\begin{thm}[Strong maximum principle for $A$-subharmonic functions in $W^{2,d}_{\loc}$]
\label{thm:Strong_maximum_principle_W2d_elliptic}
Let $\sO\subset\RR^d$ be a domain.\footnote{We reserve the term `domain' in $\RR^d$ for a \emph{connected}, open subset.}
Assume the hypotheses of Theorem \ref{thm:Viscosity_maximum_elliptic} for the coefficients of $A$ in \eqref{eq:Generator}, whose coefficients we suppose to be measurable. Suppose that $u\in W^{2,d}_{\loc}(\sO)\cap C^1(\underline\sO)$
and\,\footnote{Our hypothesis that $u\in C^1(\underline\sO)$ could be relaxed to the slightly more technical hypothesis that $Du$ be defined and continuous along $\partial_0\sO$.}
$Au\leq 0$ a.e. on $\sO$.
%COMMENT We rechecked the proof of the Hopf lemma to see that we do not need c defined on \partial_0\sO
If $c=0$ (respectively, $c\geq 0$) a.e. on $\sO$ and $u$ attains a global maximum (respectively, non-negative global maximum) in $\underline\sO$, then $u$ is constant on $\sO$.
\end{thm}

Given Theorem \ref{thm:Strong_maximum_principle_W2d_elliptic} and its variant for $A$-subharmonic functions in $C^2(\sO)$ (Theorem \ref{thm:Strong_maximum_principle_C2_elliptic}), one obtains weak maximum principles for $A$-subharmonic functions in $C^2(\sO)$ or $W^{2,d}_{\loc}(\sO)$ for a boundary-degenerate elliptic operator, $A$, and Dirichlet condition only prescribed along $\partial_1\sO$ --- see Theorems \ref{thm:Weak_maximum_principle_C2_elliptic_domain} and \ref{thm:Weak_maximum_principle_W2d_elliptic_domain} and Corollaries \ref{cor:Weak_maximum_principle_C2_elliptic_opensubset} and \ref{cor:Weak_maximum_principle_W2d_elliptic_opensubset}, respectively.

The interior local strict ellipticity and regularity requirements on $a, b$ which were assumed to prove the strong maximum principles can be relaxed in the case of the weak maximum principles by employing a priori estimates developed in \cite{Feehan_maximumprinciple, Feehan_parabolicmaximumprinciple} and ideas from \cite[\S 9.1]{GilbargTrudinger}. For example, we have

\begin{thm}[Weak maximum principle on domains for $A$-subharmonic functions in $C^2(\sO)$ and relaxed conditions on $a, b$]
\label{thm:Weak_maximum_principle_C2_elliptic_domain_relaxed}
Let $\sO\subset\RR^d$ be a bounded domain and $A$ be as in \eqref{eq:Generator}. Assume that the coefficients of $A$ obey \eqref{eq:b_perp_nonnegative_boundary_elliptic}, \eqref{eq:c_lower*_positive_elliptic}, \eqref{eq:c_locally_bounded_above_near_boundary_elliptic}, \eqref{eq:a_locally_lipschitz_elliptic}, and
\begin{equation}
\label{eq:b_continuous_on_domain_plus_degenerate_boundary}
b \in C(\underline\sO;\RR^d),
\end{equation}
and \emph{one} of the following two conditions:
\begin{gather}
\label{eq:a_lipschitz_on_domain_plus_degenerate_boundary}
a \in C^{0,1}(\underline\sO) \quad\hbox{(locally Lipschitz on $\underline\sO$), or}
\\
\label{eq:a_locally_strictly_elliptic_on_domain}
\tag{\ref*{eq:a_lipschitz_on_domain_plus_degenerate_boundary}$'$}
a \quad\hbox{obeys \eqref{eq:a_locally_strictly_elliptic} (locally strictly elliptic on $\sO$)}.
\end{gather}
Suppose $u \in C^2(\sO)\cap C^1(\underline\sO)$ and $\sup_\sO u < \infty$. If $Au\leq 0$ on $\sO$ and $u^* \leq 0$ on $\partial_1\sO$, then $u\leq 0$ on $\sO$.
\end{thm}

A virtually identical statement is available for $A$-subharmonic functions in $W^{2,d}_{\loc}(\sO)$ (the only difference is that the coefficients of $A$ are required to be measurable); see Theorem \ref{thm:Weak_maximum_principle_W2d_elliptic_domain_relaxed}.

With the aid of our weak maximum principle for $A$-subharmonic functions in $W^{2,d}_{\loc}(\sO)$ and Dirichlet condition prescribed only along $\partial_1\sO$, we obtain comparison principles for a solution (bounded above) and supersolution (bounded below) in $W^{2,d}_{\loc}(\sO)\cap C^1(\underline\sO)$ to an obstacle problem \eqref{eq:Elliptic_obstacle_problem} defined by $A$, again with Dirichlet condition prescribed only along $\partial_1\sO$, and hence uniqueness of solutions in $W^{2,d}_{\loc}(\sO)\cap C^1(\underline\sO)\cap L^\infty(\sO)$. For a general statement, see Theorem \ref{thm:Comparison_principle_elliptic_obstacle_problem}.

The analogous remarks apply, of course, to Theorem \ref{thm:Viscosity_maximum_parabolic} and its application to the development of strong and weak maximum principles for $A$-subharmonic functions in $C^2(\sQ)$ or $W^{2,d+1}_{\loc}(\sQ)$ for a boundary-degenerate parabolic operator, $L$, and Dirichlet condition prescribed only along $\mydirac_1\!\sQ$. We refer the reader to sections \ref{subsec:Strong_maximum_principle_parabolic_C2}, \ref{subsec:Strong_maximum_principle_parabolic_W2d+1}, \ref{subsec:Weak_maximum_principle_parabolic}, and \ref{subsec:Relaxing_coefficients_parabolic} for statements of strong and weak maximum principles for $L$-subharmonic functions in $C^2(\sQ)$ or $W^{2,d+1}_{\loc}(\sQ)$ for a boundary-degenerate parabolic operator, $L$, with Dirichlet condition only prescribed along $\mydirac_1\!\sQ$.

\begin{rmk}[Weak maximum principles for the case of unbounded open subsets or unbounded subharmonic functions]
\label{rmk:Unbounded}
For simplicity, we have stated all of the weak maximum principles derived in this article for the case of bounded open subsets $\sO \subset \RR^d$ or $\sQ \subset \RR^{d+1}$ and $A$ or $L$-subharmonic functions $u$ which are bounded above. For (generalized) subharmonic functions which are bounded above, all of these weak maximum principles extend to the case of unbounded open subsets, provided the coefficients of $A$ or $L$ obey suitable growth conditions; see \cite{Feehan_maximumprinciple, Feehan_parabolicmaximumprinciple} for details. Moreover, these results may be further extended to the case of a (generalized) subharmonic function with controlled growth using the simple device described in \cite[Theorem 2.20]{Feehan_maximumprinciple}; see also \cite[\S 5.D]{Crandall_Ishii_Lions_1992}.
\end{rmk}

\subsection{Theorems \ref{thm:Viscosity_maximum_elliptic} and \ref{thm:Viscosity_maximum_parabolic} and the Perron method}
\label{subsec:Perron}
It is interesting to compare Theorems \ref{thm:Viscosity_maximum_elliptic} and \ref{thm:Viscosity_maximum_parabolic} with the `harmonic lift' and `bump lemma' constructions encountered in Perron's method, keeping in mind, however, that the constructions are quite different.

Perron's method for constructing a solution --- whether in a $C^{2,\alpha}$ or viscosity sense --- to the Dirichlet problem for an elliptic (respectively, parabolic) equation with Dirichlet boundary condition, as the pointwise supremum over all `continuous subsolutions' to the elliptic (respectively, parabolic) equation requires a `harmonic lift' or `bump' construction.

In the Perron proof of existence of a $C^{2,\alpha}$ to the Dirichlet problem, the harmonic lift construction is described in \cite[Equation (2.33) and p. 103]{GilbargTrudinger}. Given a continuous subsolution, $u\in C(\bar\sO)$, to the Dirichlet problem for $w \in C^{2,\alpha}(\sO)\cap C(\bar\sO)$,
$$
\begin{cases}
Aw = f &\hbox{on } \sO,
\\
w = g &\hbox{on } \partial\sO,
\end{cases}
$$
defined by a strictly elliptic operator, $A$, with coefficients and $f$ belonging to $C^\alpha(\sO)$, and $g\in C(\partial\sO)$, and a ball $B\Subset\sO$, we can always find a larger continuous subsolution, $\hat u\in C(\bar\sO)$ obeying $\hat u \geq u$ on $\sO$ and defined by
$$
\hat u
:=
\begin{cases}
\bar u &\hbox{on } B,
\\
u &\hbox{on } \sO\less B,
\end{cases}
$$
where $\bar u \in C^{2,\alpha}(B)\cap C(\bar B)$ is the unique solution to the Dirichlet problem $A\bar u = f$ on $B$ and $\bar u = u$ on $\partial B$.

In Ishii's proof of existence of a viscosity solution to the Dirichlet problem for a fully nonlinear equation by Perron's method \cite{Ishii_1987, Ishii_1989} and \cite[Theorems 4.1 and 7.9]{Crandall_Ishii_Lions_1992},
$$
\begin{cases}
F(\cdot,w,Dw,D^2w) =0 &\hbox{on } \sO,
\\
w = g &\hbox{on } \partial\sO,
\end{cases}
$$
one requires a `bump' construction \cite[Lemma 9.1]{Crandall_1997}: if $u$ is a viscosity subsolution to $F=0$ on $\sO$ and its lower semicontinuous envelope, $u_*$, fails to be a supersolution to $F=0$ on $\sO$ at some point $\hat x \in \sO$, then for any small $\eps>0$, there is a subsolution, $u_\eps$, of $F=0$ on $\sO$ obeying
$$
\begin{cases}
u_\eps \geq u &\hbox{on } \sO \quad\hbox{and}\quad \sup_\sO (u_\eps-u) > 0,
\\
u_\eps =  u &\hbox{on } \sO\less B_\eps(\hat x).
\end{cases}
$$
Here, the construction of $u_\eps$ is elementary and does not require one to solve the Dirichlet problem on the ball $B_\eps(\hat x)$. In either the $C^{2,\alpha}$ or viscosity constructions, however, it is important that $\hat x \in \sO$ be an \emph{interior} point.

\subsection{Variational approaches to maximum principles}
\label{subsec:Variational_approach}
Recall that the strong maximum principle for weakly $A$-subharmonic functions in $W^{1,2}(\sO)$ appears as \cite[Theorem 8.19]{GilbargTrudinger} (as a consequence of the weak Harnack inequality \cite[Theorem 8.18]{GilbargTrudinger}) and for weakly $L$-subharmonic
functions\footnote{See Remark \ref{rmk:Classical_strong_max_principle_W2d+1_parabolic_c_geq_zero} for the definition of $W^{1,2}_b(\sQ)$.}
in $W^{1,2}_b(\sQ)$ as \cite[Theorem 6.25]{Lieberman} (also as a consequence of the weak Harnack inequality \cite[Theorem 6.18]{Lieberman}); the weak maximum principle for weakly $A$-subharmonic functions appears as \cite[Theorem 8.1]{GilbargTrudinger} and for weakly $L$-subharmonic functions as \cite[Corollary 6.26]{Lieberman}.

Thus, when $a \in C^{0,1}(\underline\sO)$, so the operator $A$ in \eqref{eq:Generator} may be written in divergence form, an alternative and quite different approach to the proof of the maximum principles in this article for $A$-subharmonic functions in $W^{2,d}(\sO)$ would be to derive them as consequences of their analogues for weakly $A$-subharmonic functions. This has the advantage of further weakening the regularity requirement on an $A$-subharmonic function, $u$, up to the degenerate boundary portion, $\partial_0\sO$, at the cost of strengthening the locally Lipschitz condition \eqref{eq:a_locally_lipschitz_elliptic} for $a$ along $\partial_0\sO$. However, the proof of even the weak maximum principle for a weakly $A$-subharmonic function in $W^{1,2}(\sO)$ is neither as direct nor as intuitive as that presented here, since it relies on a certain weighted Sobolev inequality (see \cite[Theorem 8.8]{Feehan_maximumprinciple}). Moreover, an analogue of the classical proof of the strong maximum principle (see \cite[Theorem 8.19]{GilbargTrudinger}) would require an analogue of the classical weak Harnack inequality (see \cite[Theorem 8.18]{GilbargTrudinger}) valid for the boundary-degenerate elliptic operator, $A$, written in divergence form (see \cite[Theorem 1.23]{Feehan_Pop_regularityweaksoln} for such a result in the case of the Heston operator \cite{Heston1993}). The proof of our Harnack inequality for the boundary-degenerate elliptic Heston operator, developed with C. Pop, is lengthy and difficult, relying on a Moser iteration argument, the abstract John-Nirenberg inequality due to Bombieri and Giusti \cite[Theorem 4]{Bombieri_Giusti_1972}, Poincar\'e and Sobolev inequalities for weighted Sobolev spaces \cite[\S 2]{Feehan_Pop_regularityweaksoln}, and delicate choices of test functions. Of course, similar remarks apply to the parabolic case.

\subsection{Outline and mathematical highlights}
\label{subsec:Guide}
For the convenience of the reader, we provide an outline of the article. Section \ref{sec:maximum_principles_elliptic} contains our proof of Theorem \ref{thm:Viscosity_maximum_elliptic} and hence strong and weak maximum principles for $A$-subharmonic functions in $C^2(\sO)$ and $W^{2,d}_{\loc}(\sO)$. In \S \ref{sec:maximum_principles_parabolic}, we prove Theorem \ref{thm:Viscosity_maximum_parabolic} and hence strong and weak maximum principles for $L$-subharmonic functions in $C^2(\sQ)$ and $W^{2,d+1}_{\loc}(\sQ)$. Appendix \ref{sec:Properties_confluent_hypergeometric_functions} contains a detailed description of the asymptotic behavior of confluent hypergeometric functions near $x=0$, in support of our example in \S \ref{subsec:Motivation}, while Appendix \ref{sec:Straightening_boundary} provides a useful variant of the familiar boundary-straightening lemma.

\subsection{Notation and conventions}
\label{subsec:Notation}
We let $\NN:=\left\{0,1,2,3,\ldots\right\}$ denote the set of non-negative integers. If $X$ is a subset of a topological space, we let $\bar X$ denote its closure and let $\partial X := \bar X\less X$ denote its topological boundary. For $r>0$ and $x^0\in\RR^d$, we let $B_r(x^0) := \{x\in\RR^d: |x-x^0|<r\}$ denote the open ball with center $x^0$ and radius $r$. We denote $\RR_+=(0,\infty)$ and $B_r^+(x^0) := B_r(x^0) \cap (\RR^{d-1}\times\RR_+)$ when $x^0\in \RR^{d-1}\times\{0\}\subset\RR^d$. When $x^0$ is the origin in $\RR^d$, we often denote $B_r(x^0)$ and $B_r^+(x^0)$ simply by $B_r$ and $B_r^+$ for brevity. When we wish to emphasize the dimension of a ball, we write $B^d$ for an open ball in $\RR^d$.

If $V\subset U\subset \RR^d$ are open subsets, we write $V\Subset U$ when $V$ is bounded with closure $\bar V \subset U$. By $\supp\zeta$, for any $\zeta\in C(\RR^d)$, we mean the \emph{closure} in $\RR^d$ of the set of points where $\zeta\neq 0$. We denote $x\vee y = \max\{x,y\}$ and $x\wedge y = \min\{x,y\}$, for any $x,y\in\RR$. We occasionally shall write coordinates on $\RR^d$ as $x=(x',x_d)\in\RR^{d-1}\times\RR$.

For an open subset of a topological space, $U\subset X$, we let $u^*:\bar U\to[-\infty,\infty]$ (respectively, $u_*:\bar U\to[-\infty,\infty]$) denote the upper (respectively, lower) semicontinuous envelope of a function $u:U\to[-\infty,\infty]$; when $u$ is continuous on $U$, then $u_* = u = u^*$ on $U$.

In the definition and naming of function spaces, we follow Adams \cite{Adams_1975} and alert the reader to occasional differences in definitions between \cite{Adams_1975} and standard references such as Gilbarg and Trudinger \cite{GilbargTrudinger}, Krylov \cite{Krylov_LecturesHolder}, or Lieberman \cite{Lieberman}.

For an open subset $\sO\subset\RR^d$, we let $C(\sO)$ denote the vector space of continuous functions on $\sO$ and let $C(\bar \sO)$ denote the Banach space of functions in $C(\sO)$ which are bounded and uniformly continuous on $\sO$, and thus have unique bounded, continuous extensions to $\bar \sO$, with norm $\|u\|_{C(\bar \sO)} := \sup_{\sO}|u|$ \cite[\S 1.26]{Adams_1975}. We let $C^1(\underline\sO)$ denote the vector subspace of functions $u \in C^1(\sO)$ such that $u\in C^1(\bar\sO')$ for every precompact open subset $\sO'\Subset \bar \sO$. For $1\leq p\leq \infty$, we say that $u \in W^{2,p}(\sO)$ if $u$ is a measurable function on $\sO$ and $u$ and its weak derivatives, $u_{x_i}$ and $u_{x_ix_j}$ for $1\leq i,j\leq d$, belong to $L^p(\sO)$ and similarly define $W^{2,p}_{\loc}(\sO)$ \cite[\S 3.1]{Adams_1975}. For the corresponding definitions of parabolic function spaces, we refer to \S \ref{subsec:Parabolic_sobolev_embedding}.

\subsection{Acknowledgments} This article was written while the author held a visiting faculty appointment in the Department of Mathematics at Columbia University, on sabbatical from Rutgers University. I am very grateful to Ioannis Karatzas and the Department of Mathematics at Columbia University, especially Panagiota Daskalopoulos and Duong Phong, for their generous support during the year.

\section{Maximum principles for the elliptic operator}
\label{sec:maximum_principles_elliptic}
Before proceeding to the proof of Theorem \ref{thm:Viscosity_maximum_elliptic}, we begin in \S \ref{subsec:Simplifying_coefficients_elliptic} by simplifying the problem with the aid of a change of variables on $\RR^d$ to bringing the coefficients of $A$ into a standard form near $\partial_0\sO$ (see Lemma \ref{lem:Simplifying_coefficients_elliptic}). In \S \ref{subsec:Hopf_boundary_point_lemma_elliptic}, we recall our version of the Hopf boundary point lemma for a boundary-degenerate elliptic linear second-order differential operator (Lemma \ref{lem:Degenerate_hopf_lemma_elliptic}). Section \ref{subsec:Perturb_degen_boundary_local_maximum_elliptic} contains the proof of Theorem \ref{thm:Viscosity_maximum_elliptic}. In \S \ref{subsec:Strong_maximum_principle_elliptic}, we apply Theorem \ref{thm:Viscosity_maximum_elliptic} to prove strong maximum principles for $A$-subharmonic functions in $C^2(\sO)\cap C^1(\underline\sO)$ and $W^{2,d}_{\loc}(\sO)\cap C^1(\underline\sO)$ (Theorems \ref{thm:Strong_maximum_principle_C2_elliptic} and \ref{thm:Strong_maximum_principle_W2d_elliptic}, respectively). In \S \ref{subsec:Weak_maximum_principle_elliptic}, we apply the preceding strong maximum principles to deduce initial versions of the weak maximum principles for $A$-subharmonic functions in $C^2(\sO)\cap C^1(\underline\sO)$ and $W^{2,d}_{\loc}(\sO)\cap C^1(\underline\sO)$ (Theorems \ref{thm:Weak_maximum_principle_C2_elliptic_domain} and \ref{thm:Weak_maximum_principle_W2d_elliptic_domain} and Corollaries \ref{cor:Weak_maximum_principle_C2_elliptic_opensubset} and \ref{cor:Weak_maximum_principle_W2d_elliptic_opensubset}, respectively). In \S \ref{subsec:Relaxing_coefficients_elliptic} we recall our definition of the `weak maximum principle property' (Definition \ref{defn:Weak_max_principle_property}) for $A$-subharmonic functions in $C^2(\sO)$ or $W^{2,d}_{\loc}(\sO)$ and the corresponding weak maximum principle estimates (Proposition \ref{prop:Elliptic_weak_max_principle_apriori_estimates}). Given an operator, $A$, with the weak maximum principle property for $A$-subharmonic functions in $W^{2,d}_{\loc}(\sO)$, we immediately a comparison principle, and thus uniqueness, for a solution and supersolution in $W^{2,d}_{\loc}(\sO)$ to the unilateral obstacle problem (Theorem \ref{thm:Comparison_principle_elliptic_obstacle_problem}). We conclude by using our earlier versions of the weak maximum principles --- with rather strong hypotheses on the coefficients $a,b$ of $A$ --- and our weak maximum principle estimates to derive versions of the weak maximum principles, with more relaxed hypotheses on those coefficients, for $A$-subharmonic functions in $C^2(\sO)\cap C^1(\underline\sO)$ and $W^{2,d}_{\loc}(\sO)\cap C^1(\underline\sO)$ (Theorems \ref{thm:Weak_maximum_principle_C2_elliptic_domain_relaxed} and \ref{thm:Weak_maximum_principle_W2d_elliptic_domain_relaxed}, respectively).

\subsection{Simplifying the coefficients}
\label{subsec:Simplifying_coefficients_elliptic}
Before proceeding to the proof of Theorem \ref{thm:Viscosity_maximum_elliptic}, we shall need the following technical lemma. A simpler change-of-variable result is used, with an operator $A$ with constant coefficients for $Du$, in the proofs of \cite[Proposition A.1]{Feehan_Pop_mimickingdegen_pde} and in \cite[\S 3.2]{Feehan_Pop_elliptichestonschauder}.

\begin{lem}[Simplifying the coefficients]
\label{lem:Simplifying_coefficients_elliptic}
Let $\sO\subset\RR^d$ be an open subset, with $C^{1,\alpha}$ boundary portion $\partial_0\sO$ for some $\alpha\in(0,1)$, and $A$ be as in \eqref{eq:Generator}. Assume that $a$ obeys \eqref{eq:a_continuous_degen_boundary_elliptic} and $b$ obeys \eqref{eq:b_perp_positive_boundary_elliptic} and \eqref{eq:b_C2_elliptic}. Suppose that $\bar x^0 \in \partial_0\sO$. Then there are a constant $\delta>0$ and a $C^1$ diffeomorphism, $\Phi:\RR^d\to\RR^d$, which restricts to a $C^2$ diffeomorphism from $\RR^d\less (B_{2\delta}(\bar x^0)\cap\partial_0\sO)$ onto its image such that the following holds. If $u \in C^2(\RR^d)$
and\,\footnote{The function $v$ may only be $C^1$ on $\RR^d$ but is $C^2$ on $\RR^d\less (B_{2\delta}(\bar x^0)\cap\partial_0\sO)$.}
$v := u\circ\Phi^{-1}$ and $\widetilde\sO:=\Phi(\sO)$, and the operator
\begin{equation}
\label{eq:Tilde_generator}
\tilde Av := -\tr(\tilde aD^2v) - \langle \tilde b, Dv\rangle + \tilde cv \quad\hbox{on }\widetilde\sO,
\end{equation}
and its coefficients, $\tilde a:\underline{\widetilde\sO}\to\sS^+(d)$ and $\tilde b:\underline{\widetilde\sO}\to\RR^d$ and $\tilde c:\widetilde\sO\to\RR$, are defined by $\tilde Av := (Au)\circ\Phi^{-1}$ on $\widetilde\sO$, then
\begin{subequations}
\begin{align}
\label{eq:a_tilde_boundary_zero}
\tilde a = 0 &\quad\hbox{on } \partial_0\widetilde\sO,
\\
\label{eq:b_tilde_perp_positive_boundary}
\tilde b^\perp > 0 &\quad\hbox{on } \partial_0\widetilde\sO,
\\
\label{eq:b_tilde_tangential_zero}
\tilde b^\parallel = 0 &\quad\hbox{on } B_\delta(\bar x^0)\cap\partial_0\widetilde\sO,
\\
\label{eq:c_tilde_preserved}
\tilde c = c\circ\Phi^{-1} &\quad\hbox{on } \widetilde\sO.
\end{align}
\end{subequations}
\end{lem}

\begin{rmk}[Smoothness of the diffeomorphism]
\label{rmk:Simplifying_coefficients_elliptic}
If the condition \eqref{eq:b_C2_elliptic} is replaced by $b\in C^k(\underline\sO)$, for an integer $k\geq 2$, then the diffeomorphism, $\Phi$, in Lemma \ref{lem:Simplifying_coefficients_elliptic} will be $C^k$ on $\RR^d\less (B_{2\delta}(\bar x^0)\cap\partial_0\sO)$.
\end{rmk}

\begin{rmk}[Preservation of the eigenvalues of the matrix $a(x)$]
If $\Phi:\RR^d\to\RR^d$ is an arbitrary $C^2$ diffeomorphism and we write $y=\Phi(x)$ and $v(y)=u(x)$, then
$$
u_{x_i} = v_{y_k}\frac{\partial y_k}{\partial x_i}
\quad\hbox{and}\quad
u_{x_ix_j} = v_{y_ky_l}\frac{\partial y_k}{\partial x_i}\frac{\partial y_l}{\partial x_j}
+ v_{y_k}\frac{\partial^2 y_k}{\partial x_ix_j},
$$
and hence the relation $\tilde Av(y) = Au(x)$ gives
$$
\tilde a^{kl} = a^{ij}\frac{\partial y_k}{\partial x_i}\frac{\partial y_l}{\partial x_j},
$$
and so the eigenvalues of $\tilde a(y)$ coincide with those of $a(x)$, for each $x\in\underline\sO$ such that the Jacobian matrix of the change of variables is orthogonal.
\end{rmk}

\begin{proof}[Proof of Lemma \ref{lem:Simplifying_coefficients_elliptic}]
By first applying a translation and then a rotation of $\RR^d$, respectively, we may assume without loss of generality that $\bar x^0 = 0 \in \RR^d$ (origin) and $\vec n(0)=e_d$ (inward-pointing normal vector on $\partial_0\sO$). By virtue of Lemma \ref{lem:Straightening_boundary} and writing $x=(x',x_d)$, we can find a $C^1$ diffeomorphism of $\RR^d$ of the form $\RR^{d-1}\times\RR\ni(x',x_d)\mapsto (z'(x',x_d),x_d)\in\RR^{d-1}\times\RR$ such that, for $\delta>0$ small enough, $\bar B_{2\delta}(\bar x^0)\cap\partial\sO \subset \partial_0\sO$ and
$$
B_{2\delta}(\bar x^0)\cap\partial\sO = B_{2\delta}\cap(\RR^{d-1}\times\bar\RR_+) \equiv B_{2\delta}^+,
$$
and which restricts to a $C^\infty$ diffeomorphism from $\RR^d\less(B_{2\delta}(\bar x^0)\cap\partial\sO)$ onto its image.

In addition, for small enough $\delta$, the hypothesis \eqref{eq:b_C2_elliptic} on $b$ implies (after relabeling the coefficients of $A$ due to the coordinate change) that $b^d>0$ on $\bar B_{2\delta}^+$ since $b^d(\bar x^0)>0$ by the hypothesis \eqref{eq:b_perp_positive_boundary_elliptic}. We relabel the coordinates on $\RR^d$ as $x=(x',x_d)\in\RR^{d-1}\times\RR$ and define
$$
\xi^i(x) := -\frac{b^i(x',0)}{b^d(x',0)}, \quad\hbox{for } 1\leq i\leq d-1, \quad\hbox{and}\quad \xi^d(x) := 0, \quad\forall\, x \in \underline B_{2\delta}^+.
$$
The functions $b^i$, for $1\leq i\leq d$, are $C^2$ on $B_{2\delta}^+$ by hypothesis \eqref{eq:b_C2_elliptic} and may be extended from $B_{2\delta}^+$ to give $C^1$ functions $b^i$ on $B_{2\delta}$ (see \cite[Lemma 6.37]{GilbargTrudinger}) such that $b^d>0$ on $\bar B_{2\delta}$ and so the $C^2$ functions $\xi^i$ on $B_{2\delta}^+$ also extend to give $C^1$ functions $\xi^i$ on $B_{2\delta}$, for $1\leq i\leq d$. We define
$$
y := \begin{cases} x + \xi(x) x_d, & x \in B_\delta, \\ x, & x \in \RR^d\less B_{2\delta}, \end{cases}
$$
and interpolate between these two definitions on the annulus, $B_{2\delta} \less B_\delta$, with a $C^\infty$ cutoff function, $\psi:\RR\to [0,1]$, where $\psi(s)=1$ for $s\leq 1$ and $\psi(s)=0$ for $s\geq 2$, so
$$
y = x + \psi\left(\frac{|x|}{\delta}\right)\xi(x) x_d, \quad\hbox{for } x\in B_{2\delta} \less B_\delta.
$$
Note finally that $y=x$ on $\RR^{d-1}\times\{0\}$.

By direct calculations, using $v(y) = u(x)$ for all $x\in\sO$, we obtain (noting that we only sum over indices where explicitly indicated),
\begin{align*}
u_{x_i} &= v_{y_i} + x_d\sum_{k\neq d}\xi^k_{x_i}v_{y_k}, \quad \forall\, i \neq d,
\\
u_{x_d} &= v_{y_d} + \sum_{k\neq d}\xi^k v_{y_k}.
\end{align*}
Observe that
\begin{align*}
u_{x_ix_j} &= v_{y_iy_j} + x_d\sum_{k\neq d}\xi^k_{x_j}v_{y_iy_k} + x_d\sum_{k\neq d}\xi^k_{x_ix_j}v_{y_k}
+ x_d\sum_{k\neq d}\xi^k_{x_i}\left(v_{y_jy_k}  + x_d\sum_{l=1}^{d-1}\xi^l_{x_j}v_{y_ky_l} \right)
\\
&=v_{y_iy_j} + x_d\sum_{k\neq d}\left(\xi^k_{x_j}v_{y_iy_k} + \xi^k_{x_i}v_{y_jy_k}\right)
+ x_d^2\sum_{k,l\neq d}\xi^k_{x_i}\xi^l_{x_j}v_{y_ky_l}
+ x_d\sum_{k\neq d} \xi^k_{x_ix_j}v_{y_k},
\quad \forall\, i,j \neq d.
\end{align*}
Next, we have
\begin{align*}
u_{x_ix_d} &= v_{y_iy_d} + \sum_{k\neq d}\xi^k v_{y_iy_k}
+ \sum_{k\neq d}\xi^k_{x_i}v_{y_k}
+ x_d\sum_{k\neq d}\xi^k_{x_i}\left(v_{y_ky_d} + \sum_{l\neq d}\xi^l v_{y_ky_l} \right)
\\
&= v_{y_iy_d} + \sum_{k\neq d}\xi^k v_{y_iy_k}
+ x_d\sum_{k\neq d}\xi^k_{x_i}v_{y_ky_d} +  x_d\sum_{k,l\neq d}\xi^k_{x_i}\xi^l v_{y_ky_l}
+ \sum_{k\neq d}\xi^k_{x_i}v_{y_k},
\quad \forall\, i \neq d.
\end{align*}
Finally, we see that
\begin{align*}
u_{x_dx_d} &= v_{y_d y_d} + \sum_{k\neq d}\xi^k v_{y_ky_d}
+  \sum_{k\neq d}\xi^k\left(v_{y_ky_d} + \sum_{l\neq d}\xi^l v_{y_ky_l}\right)
\\
&= v_{y_d y_d} + 2\sum_{k\neq d}\xi^k v_{y_ky_d} + \sum_{k,l\neq d}\xi^k\xi^l v_{y_ky_l}.
\end{align*}
Substituting into
$$
\tilde Av(y) = Au(x), \quad x\in \sO,
$$
and keeping in mind that $(\tilde a^{ij})$ will be symmetric, we find that,
\begin{equation}
\begin{aligned}
\label{eq:tilde a_coefficients}
\tilde a^{ij} &= a^{ij} + x_d\sum_{n\neq d}a^{in}\xi^j_{x_n} + x_d\sum_{m\neq d}a^{mj}\xi^i_{x_m}
+ x_d^2\sum_{m,n\neq d}a^{mn}\xi^i_{x_m}\xi^j_{x_n}
\\
&\quad + \frac{1}{2}\left(a^{id}\xi^j + a^{jd}\xi^i\right) +  x_d\sum_{m\neq d} a^{md}\xi^i_{x_m}\xi^j + a^{dd}\xi^i\xi^j, \quad\forall\, i,j \neq d,
\\
\tilde a^{id} &= a^{id} + \xi^i a^{dd} + x_d\sum_{m\neq d}a^{md}\xi^i_{x_m}, \quad\forall\, i \neq d,
\\
\tilde a^{dd} &= a^{dd}.
\end{aligned}
\end{equation}
Clearly, by \eqref{eq:Degeneracy_locus_elliptic} and \eqref{eq:a_continuous_degen_boundary_elliptic}, the matrix $\tilde a = (\tilde a^{ij})$ obeys
$$
\tilde a = 0 \quad\hbox{on } \partial_0 B_{2\delta}^+,
$$
that is, \eqref{eq:a_tilde_boundary_zero} holds. Furthermore,
\begin{equation}
\begin{aligned}
\label{eq:tilde b_coefficients}
\tilde b^j &= b^j + x_db^j\sum_{k\neq d}\xi^k_{x_j} + \xi^jb^d + x_d\sum_{m,n\neq d} a^{mn}\xi^j_{x_mx_n} + \sum_{m\neq d}a^{md}\xi^j_{x_m},
\quad\forall\, j \neq d,
\\
\tilde b^d &= b^d.
\end{aligned}
\end{equation}
Because $\vec n = e_d$, we see that $\tilde b^\parallel = (\tilde b^1,\ldots,\tilde b^{d-1},0)$ and $\tilde b^\perp = (0,\ldots,0,\tilde b^d)$ on $B_{2\delta}^+$ by \eqref{eq:b_splitting_elliptic}. Thus, \eqref{eq:b_perp_positive_boundary_elliptic} yields \eqref{eq:b_tilde_perp_positive_boundary}, while $\tilde c = c\circ\Phi^{-1}$ on $\widetilde\sO$ and so \eqref{eq:c_tilde_preserved} holds. In particular, due to our choice of $\xi^i$, we see that \eqref{eq:Degeneracy_locus_elliptic} and \eqref{eq:a_continuous_degen_boundary_elliptic} yield
$$
\tilde b^i = 0 \quad\hbox{on } \partial_0 B_{2\delta}^+, \quad\forall\, i \neq d,
$$
so \eqref{eq:b_tilde_tangential_zero} holds. This completes the proof of the lemma.
\end{proof}

\subsection{Hopf boundary point lemma for a boundary-degenerate elliptic linear second-order differential operator}
\label{subsec:Hopf_boundary_point_lemma_elliptic}
We shall also need the following refinement of the classical Hopf boundary point lemma \cite[Lemma 3.4]{GilbargTrudinger}.

\begin{lem}[Hopf boundary point lemma for a boundary-degenerate elliptic linear second-order differential operator]
\label{lem:Degenerate_hopf_lemma_elliptic}
\cite[Lemma 4.1]{Feehan_maximumprinciple}
%TODO Many small changes; update per wmp paper
%TODO 7.11.2013 Recheck against updated WMP paper
Let $\sO\subset\RR^d$ be an open subset, $B\subset\sO$ an open ball such that $\bar B\subset\underline\sO$, and $x^0\in\partial B$. Require that the coefficients of the operator $A$ in \eqref{eq:Generator} obey \eqref{eq:a_locally_strictly_elliptic}, \eqref{eq:c_nonnegative_elliptic}, \eqref{eq:b_perp_positive_boundary_elliptic}, \eqref{eq:c_locally_bounded_above_near_boundary_elliptic}, and that
\begin{gather}
\label{eq:ab_locally_bounded_near_boundary_elliptic}
%TODO Looks like we just need b essentially locally bounded
a \hbox{ and } b \quad\hbox{are (essentially) locally bounded on } \underline\sO,
\\
%TODO Looks like we need b continuous on \overline{\partial_0\sO}, as \lambda_* being lsc does the rest
\label{eq:ab_continuous_coefficient_boundary_elliptic}
a \hbox{ and } b \quad \hbox{are continuous on } \partial_0\sO.
\end{gather}
Suppose that $u\in C^2(\sO)$ (respectively, $u\in W^{2,d}_{\loc}(\sO)$) obeys $Au\leq 0$ (a.e.) on $\sO$ and satisfies the conditions,
\begin{enumerate}
\renewcommand{\theenumi}{\roman{enumi}}
\item $u$ is continuous at $x^0$;
\item\label{item:xzero_strict_local_max_elliptic} $u(x^0) > u(x)$, for all $x\in B$;
\item $D_{\vec n} u(x^0)$ exists,
\end{enumerate}
where $D_{\vec n} u(x^0)$ is the derivative of $u$ at $x^0$ in the direction of the \emph{inward}-pointing unit normal vector, $\vec n(x^0)$, at $x^0\in\partial B$. Then the following hold:
\begin{enumerate}
\item\label{item:Hopf_c_zero_elliptic} If $c=0$ on $\sO$, then $D_{\vec n} u(x^0)$ obeys the strict inequality,
\begin{equation}
\label{eq:Positive_inward_normal_derivative_elliptic}
D_{\vec n} u(x^0) < 0.
\end{equation}
\item\label{item:Hopf_c_geq_zero_elliptic} If $c\geq 0$ on $\sO$ and $u(x^0)\geq 0$, then \eqref{eq:Positive_inward_normal_derivative_elliptic} holds.
\item\label{item:Hopf_c_no_sign_elliptic} If $u(x^0)=0$, then \eqref{eq:Positive_inward_normal_derivative_elliptic} holds irrespective of the sign of $c$.
\end{enumerate}
\end{lem}

\subsection{Perturbing a degenerate-boundary local maximum to an interior local maximum}
\label{subsec:Perturb_degen_boundary_local_maximum_elliptic}
We now give the

\begin{comment}
%COMMENT Include in journal version
\begin{proof}[Proof of Theorem \ref{thm:Viscosity_maximum_elliptic}]
The argument is merely a simpler version of the proof of Theorem \ref{thm:Viscosity_maximum_parabolic} (parabolic case), with Lemmas \ref{lem:Simplifying_coefficients_elliptic} and Lemma \ref{lem:Degenerate_hopf_lemma_elliptic} replacing the roles of Lemmas \ref{lem:Simplifying_coefficients_parabolic} and Lemma \ref{lem:Degenerate_hopf_lemma_parabolic}, so we only include a detailed proof of Theorem \ref{thm:Viscosity_maximum_parabolic}. However, the full proof of Theorem \ref{thm:Viscosity_maximum_elliptic} is included in an otherwise identical version of this article available at \cite{Feehan_perturbationlocalmaxima}.
\end{proof}
\end{comment}

%COMMENT Omit in journal version
\begin{proof}[Proof of Theorem \ref{thm:Viscosity_maximum_elliptic}]
For ease of language, we shall suppose that $u\in C^2(\sO)$; there is no difference in the argument when $u\in W^{2,d}_{\loc}(\sO)$.

Just as in the proof of Lemma \ref{lem:Simplifying_coefficients_elliptic}, by first applying a $C^1$ diffeomorphism of $\RR^d$, which is $C^2$ on the complement of an open neighborhood in $\partial_0\sO$ of the point $\bar x^0\in\partial_0\sO$, to straighten the boundary, we may assume without loss of generality that $\bar x^0 = 0 \in \RR^d$ (origin) and $\vec n(\bar x^0) = e_d$ and, for $\rho_0>0$ small enough, that $\bar B_{\rho_0}(\bar x^0)\cap\partial\sO \subset \partial_0\sO$ and
$$
B_{\rho_0}(\bar x^0)\cap\partial\sO = B_{\rho_0}^{d-1}\times \{0\} \subset \RR^{d-1}\times\bar\RR_+,
$$
where $B_{\rho_0}^{d-1}$ is the open ball in $\RR^{d-1}$ with center at the origin and radius ${\rho_0}$. For a small enough positive constant $\ell$ (depending on the open subset $\sO$ and implicitly on $\bar x^0$), we may assume that $\bar B_{\rho_0}^{d-1}\times [0,\ell]\subset\underline\sO$, with $u$ attaining a non-negative maximum on $\bar B_{\rho_0}^{d-1}\times [0,\ell]$ at the origin, and which is a strict maximum in the sense that
$$
u(x) < u(0), \quad\forall x \in B_{\rho_0}^{d-1}\times (0,\ell).
$$
We begin with the construction of the quadratic polynomial.

\begin{step}[Construction of the quadratic polynomial, $w$]
\label{step:Construction_w_elliptic}
Since $Du(\bar x^0)$ exists by hypothesis, the function $u(x)$ has a first-order Taylor polynomial approximation,
$$
u(x) = u(0) + \langle Du(0),x\rangle + |x|o(x), \quad\forall\, x \in \bar B_{\rho_0}^{d-1}\times [0,\ell],
$$
where $o$ is continuous on $\bar B_{\rho_0}^{d-1}\times [0,\ell]$ with $o(x)=0$ when $x=0$, and
$$
u_{x_i}(0) = 0, \quad 1\leq i\leq d-1, \quad\hbox{and}\quad u_{x_d}(0) \leq 0,
$$
because $u$ has a local maximum at $x=0$. By hypothesis, this local maximum is strict and thus Lemma \ref{lem:Degenerate_hopf_lemma_elliptic} implies that
$$
u_{x_d}(0) := p < 0.
$$
Consequently, setting $u(0):=r$, the Taylor formula for $u$ becomes
\begin{equation}
\label{eq:u_taylor_polynomial_elliptic}
u(x) = r + px_d + |x|o(x), \quad\forall\, x \in \bar B_{\rho_0}^{d-1}\times [0,\ell].
\end{equation}
We now define a quadratic polynomial, for constants $0<\eta\leq 1$ and $Q>0$, to be chosen later,
\begin{equation}
\label{eq:Quadratic_polynomial_perturbation_elliptic}
w(x) := (\eta-p)x_d - \frac{Q}{2}|x|^2, \quad\forall\, x \in \RR^d,
\end{equation}
and observe that
$$
Aw(x) = \tr a(x)\,Q - b^d(x)(\eta - p) + \langle b(x), x\rangle Q + c(x)\left((\eta-p)x_d - \frac{Q}{2}|x|^2\right),
$$
and thus, for all $x\in \sO$,
\begin{equation}
\label{eq:Aw}
Aw(x) = \left(\tr a(x) + \langle b(x), x\rangle\right)Q - b^d(x)(\eta-p) + c(x)\left((\eta-p)x_d - \frac{Q}{2}|x|^2\right).
\end{equation}
From \eqref{eq:b_splitting_elliptic} we have $b(x) = b^\parallel(x) + b^\perp(x)$ for $x\in\bar B_{\rho_0}^{d-1}\times [0,\ell]$. We have $b^\perp(x) = (0,\ldots,0,b^d(x))$, since $\vec n(x)=e_d$, and $b^\parallel(x) = (b^1(x),\ldots,b^{d-1}(x),0)$. For convenience, write $x=(x',x_d) \in \RR^{d-1}\times\RR$. By hypothesis \eqref{eq:b_tangential_component_zero_degenerate_boundary_elliptic} or hypothesis \eqref{eq:b_C2_elliptic} and  Lemma \ref{lem:Simplifying_coefficients_elliptic} if \eqref{eq:b_tangential_component_zero_degenerate_boundary_elliptic} does not initially hold, we can assume
$$
b^\parallel(x',0) = 0, \quad\forall\,x' \in B_{\rho_0}^{d-1}.
$$
The definition \eqref{eq:Degeneracy_locus_elliptic} of $\partial_0\sO$ and the fact that $a\in\ C(\underline\sO;\sS^+(d))$ by \eqref{eq:a_continuous_degen_boundary_elliptic} implies
$$
a(x',0) = 0, \quad\forall\,x' \in B_{\rho_0}^{d-1}.
$$
Therefore, the locally Lipschitz hypothesis \eqref{eq:a_locally_lipschitz_elliptic} for $a$ and hypotheses \eqref{eq:b_tangential_locally_lipschitz_elliptic} or \eqref{eq:b_C2_elliptic} for $b$ imply that there is a positive constant, $K$, such that
\begin{align}
\label{eq:a_Lipschitz_endpoint_nbhd}
\tr a(x) &\leq Kx_d,
\\
\label{eq:bprime_endpoint_nbhd_bound}
|b^\parallel(x)| &\leq Kx_d, \quad\forall\, x\in B_{\rho_0}^{d-1}\times(0,\ell).
\end{align}
Also, by hypothesis \eqref{eq:b_perp_positive_boundary_elliptic}, we have
$$
b^d(0) := b_0>0,
$$
and so, for small enough $\ell$ and $\rho_0$, we see that
\begin{equation}
\label{eq:bd_continuity_endpoint_nbhd_bound}
\frac{b_0}{2} \leq b^d(x) \leq 2b_0, \quad\forall\, x\in B_{\rho_0}^{d-1}\times(0,\ell).
\end{equation}
Moreover, $c$ is locally bounded above on $\underline\sO$ by hypothesis \eqref{eq:c_locally_bounded_above_near_boundary_elliptic}, so there is a positive constant, $\Lambda_0$, such that
$$
c(x) \leq \Lambda_0, \quad \forall\, x\in B_{\rho_0}^{d-1}\times(0,\ell).
$$
We may further assume that $\ell=(b_0,\Lambda_0)$ is chosen small enough that
$$
\Lambda_0\ell \leq \frac{b_0}{4},
$$
and so we obtain
\begin{equation}
\label{eq:cx_endpoint_nbhd_bound}
0\leq c(x)x_d \leq \frac{b_0}{4}, \quad \forall\, x\in B_{\rho_0}^{d-1}\times(0,\ell).
\end{equation}
Thus, writing $\langle b(x), x\rangle = \langle b^\parallel(x), x'\rangle + b^d(x)x_d$ in
\eqref{eq:Aw} and using $c\geq 0$ on $\sO$ (by hypothesis \eqref{eq:c_nonnegative_elliptic}) yields
$$
Aw(x) \leq \langle b^\parallel(x), x'\rangle Q + \left(\tr a(x) + b^d(x)x_d\right)Q - b^d(x)(\eta-p) + c(x)(\eta-p)x_d,
$$
and by applying \eqref{eq:a_Lipschitz_endpoint_nbhd}, \eqref{eq:bprime_endpoint_nbhd_bound}, \eqref{eq:bd_continuity_endpoint_nbhd_bound}, and \eqref{eq:cx_endpoint_nbhd_bound}, we discover that
\begin{equation}
\label{eq:Aw_upper_bound}
Aw(x) \leq KQ |x'|x_d + (K+2b_0)Qx_d - \frac{b_0}{4}(\eta-p),
\quad \forall\, x\in B_{\rho_0}^{d-1}\times(0,\ell).
\end{equation}
Now let
$$
m := \frac{8}{b_0}(K+2b_0) = \frac{8K}{b_0} + 16,
$$
and observe that
\begin{equation}
\label{eq:xd_bound_elliptic}
0 \leq (K + 2b_0)Q x_d <  \frac{b_0}{8}(\eta-p) \iff 0 \leq x_d < \frac{\eta-p}{mQ}.
\end{equation}
But then
\begin{equation}
\label{eq:xprime_bound_elliptic}
KQ|x'|x_d < \frac{b_0}{8}(\eta-p) \Leftarrow |x'| < \frac{mb_0}{8K} \quad\hbox{and}\quad 0 \leq x_d < \frac{\eta-p}{mQ}.
\end{equation}
Note that $mb_0/(8K) = 1 + 2b_0/K$ by definition of $m$ and so we may suppose without loss of generality that $\rho_0=\rho_0(b_0,K)>0$ is chosen small enough that
\begin{equation}
\label{eq:rho_less_xprime_bound_elliptic}
\rho_0 \leq 1 + \frac{2b_0}{K}.
\end{equation}
By choosing $Q=Q(\ell,m,p)$ large enough, we may also assume (recall that $p<0$ depends on $u$ but is fixed and $\eta\in(0,1]$) without loss of generality that
\begin{equation}
\label{eq:xd_bound_less_ell_elliptic}
\hat x_d(\eta,Q) := \frac{\eta-p}{mQ} \leq \ell.
\end{equation}
In the sequel, we shall exploit the fact the inequality in \eqref{eq:xd_bound_less_ell_elliptic} is preserved as $\eta\downarrow 0$ or $Q\uparrow\infty$. By virtue of \eqref{eq:Aw_upper_bound}, \eqref{eq:xprime_bound_elliptic}, \eqref{eq:xd_bound_elliptic}, \eqref{eq:rho_less_xprime_bound_elliptic}, and \eqref{eq:xd_bound_less_ell_elliptic}, we obtain
\begin{align*}
Aw &< \frac{b_0}{8}(\eta-p) + \frac{b_0}{8}(\eta-p) - \frac{b_0}{4}(\eta-p)
\\
&= 0 \quad\hbox{on } B_{\rho_0}^{d-1}\times (0,\hat x_d),
\end{align*}
where it necessarily also holds that $Av = Au+Aw < 0$, with $v := u + w$, since $Au\leq 0$ on $\sO$ by hypothesis. This completes Step \ref{step:Construction_w_elliptic} and establishes one of the conclusions of Theorem \ref{thm:Viscosity_maximum_elliptic}, where the function $w$ appearing in that conclusion is equal to the preceding quadratic polynomial composed with the coordinate changes required in our construction.
\end{step}

We next show that $v$ attains a maximum value $v(x^0)>u(0)$ at some point $x^0 \in \bar B_{\rho_0}^{d-1}\times[0,\ell]$.

\begin{step}[Maximum of $v$ on the cylinder]
\label{step:Interior_maximum_elliptic}
By \eqref{eq:u_taylor_polynomial_elliptic} and writing $x=(x',x_d)$, the function $v(x)$ has the Taylor polynomial approximation,
\begin{equation}
\label{eq:v_taylor_polynomial_elliptic}
v(x) = r + \left(\eta - \frac{Q}{2}x_d\right)x_d + |x|o(x) - \frac{Q}{2}|x'|^2,
\quad \forall\, x\in \bar B_{\rho_0}^{d-1}\times[0,\ell].
\end{equation}
Consequently, writing $o(0,x_d) = o(x_d)$,
$$
v(0,x_d) = r + \left(\eta - \frac{Q}{2}x_d\right)x_d + x_do(x_d), \quad \forall\, x_d\in [0,\ell],
$$
and thus, setting $x_d = \xi \hat x_d \in (0, \hat x_d)$ for $0<\xi<1$ and $\hat x_d\in (0,\ell]$ defined in \eqref{eq:xd_bound_less_ell_elliptic}, we see that
$$
v(0,\xi\hat x_d) = r + \left(\eta - \frac{Q}{2}\xi\hat x_d + o(\xi\hat x_d)\right)\xi\hat x_d, \quad \forall\, \xi \in (0,1).
$$
For any $\eta, Q$ (and thus $\hat x_d(\eta,Q)\in (0,\ell]$), we may choose $\xi=\xi(\eta,Q) \in (0,1)$ such that
$$
\frac{Q}{2}\xi\hat x_d + |o(\xi\hat x_d)| \leq \frac{\eta}{2},
$$
and so
$$
v(0,\xi\hat x_d) \geq r + \frac{\eta}{2}\xi\hat x_d > r.
$$
Therefore, because $v(0,\xi\hat x_d)>r$ for small enough $\xi=\xi(\eta,Q)\in (0, 1)$, the function $v$ must attain a maximum $v(x^0)>r$ at some point $x^0 \in \bar B_{\rho_0}^{d-1}\times[0,\hat x_d]$.
\end{step}

It remains to show that $v$ on $\bar B_{\rho_0}^{d-1}\times[0,\ell]$ attains its maximum, $v(x^0)>u(0)$, at a point $x^0$ in the interior of a cylinder with curved side,
$$
V := \{(x',x_d)\in B_{\rho_0}^{d-1}\times (0,\hat x_d): |x'|< \rho(x_d)\}
\subset
B_{\rho_0}^{d-1}\times (0,\ell),
$$
defined by a suitably-chosen positive-valued function,
\begin{equation}
\label{eq:Defn_rho_elliptic}
\rho:[0,\infty)\to(0,\infty),
\end{equation}
to be determined in the sequel.

\begin{step}[Upper bound for $v$ on the boundary]
\label{step:Upper_bound_boundary}
To show that $x^0$ must belong to the interior of $V$, we prove that $v\leq r$ on $\partial V$ for suitable $\eta$ and $Q$, where
$$
\partial V = \left(\bar V\cap\{x_d=0\}\right) \cup \left(\bar V\cap\{x_d=\hat x_d\}\right) \cup \left(\bar V\cap\{0<x_d<\hat x_d\}\right),
$$
consisting of the bottom, top, and curved side of the cylinder, $V$, respectively.

\setcounter{case}{0}
\begin{case}[Upper bound for $v$ on the bottom of the cylinder]
\label{case:Cylinder_bottom_elliptic}
First, consider the contribution from the \emph{bottom} of the cylinder,
$$
\bar V\cap\{x_d=0\} \subset B_{\rho_0}^{d-1}\times\{0\}.
$$
We have $u(x',0)\leq r$ for all $x' \in \bar B_{\rho_0}^{d-1}$ (since $u(x)$ has maximum value $u(0)=r$ on $\bar B_{\rho_0}^{d-1}\times[0,\ell]$) and thus
$$
u(x',0) \leq r \quad\hbox{for all } |x'|\leq \rho_0,
$$
and so the definition \eqref{eq:Quadratic_polynomial_perturbation_elliptic} of $w$ and the fact that $v=u+w$ yields
$$
v(x',0) = r - \frac{Q}{2}|x'|^2 < r \quad\hbox{for }|x'|\leq \rho_0,
$$
as desired for Case \ref{case:Cylinder_bottom_elliptic}.
\end{case}

\begin{case}[Upper bound for $v$ on the top of the cylinder]
\label{case:Cylinder_top_elliptic}
Second, consider the contribution from the \emph{top} of the cylinder,
$$
\bar V\cap\{x_d=\hat x_d\} \subset B_{\rho_0}^{d-1}\times\{\hat x_d\}.
$$
Gathering like terms in the Taylor formula \eqref{eq:v_taylor_polynomial_elliptic}, setting $x_d = \hat x_d(\eta,Q) \equiv (\eta-p)/(mQ)$ via \eqref{eq:xd_bound_less_ell_elliptic}, and simplifying yields
$$
v(x',\hat x_d) = r + \left(\eta + o(x',\hat x_d) - \frac{(\eta-p)}{2m}\right)\hat x_d  + |x'|o(x',\hat x_d) - \frac{Q}{2}|x'|^2,
\quad \forall\, x' \in \bar B_{\rho_0}^{d-1}.
$$
Now choose $\eta = \eta(m,p) \in (0,1]$ small enough that (recalling that $p<0$ by assumption)
\begin{equation}
\label{eq:eta_bound_elliptic}
\eta < -\frac{p}{8m},
\end{equation}
so (recalling that $m>1$)
$$
\frac{(2m-1)}{2m}\eta < -\frac{p}{8m}.
$$
Now make $\hat x_d\in (0,\ell]$ small enough, by choosing $Q=Q(m,p,u)$ large enough and applying the definition \eqref{eq:xd_bound_less_ell_elliptic} of $\hat x_d$, and choose $\rho_0(m,p,u)>0$ small enough that
\begin{equation}
\label{eq:o_xprime_hatxd_bound}
o(x',x_d) < -\frac{p}{8m}\wedge \frac{\zeta}{2}, \quad |x'| \leq \rho_0, \quad 0\leq x_d\leq \hat x_d,
\end{equation}
which we can do since $o(x',x_d)$ is continuous on $\bar B_{\rho_0}^{d-1}\times[0,\ell]$. These choices of $\eta$ and $Q$ preserve the upper bound in \eqref{eq:xd_bound_less_ell_elliptic}. By \eqref{eq:eta_bound_elliptic} and \eqref{eq:o_xprime_hatxd_bound} and discarding the term $-Q|x'|^2/2$, we discover that
$$
v(x',\hat x_d) \leq r + \frac{p}{4m}\hat x_d - \frac{p}{8m}|x'|,
\quad \forall\, x' \in \bar B_{\rho_0}^{d-1}.
$$
We now restrict our function $\rho$ in \eqref{eq:Defn_rho_elliptic} by requiring that
\begin{equation}
\label{eq:rho_value_cylinder_top_elliptic}
\rho(\hat x_d) = \hat x_d,
\end{equation}
so now $\rho$ depends on $\eta\in(0,1]$ and $Q>0$. Consequently,
$$
v(x',\hat x_d) \leq r + \frac{p}{8m}\hat x_d < r, \quad \forall\, x' \in \bar B_{\hat x_d}^{d-1},
$$
as desired for Case \ref{case:Cylinder_top_elliptic}.
\end{case}

\begin{case}[Upper bound for $v$ on the curved side of the cylinder]
\label{case:Cylinder_side_elliptic}
It remains to consider the contribution from the curved side of the cylinder in
$$
\{(x',x_d)\in B_{\rho_0}^{d-1}\times (0,\hat x_d): |x'|< \rho(x_d)\}
\subset
\bar V,
$$
for a suitably-chosen positive-valued function $\rho$ in \eqref{eq:Defn_rho_elliptic}.

We proceed by modifying the argument for Case \ref{case:Cylinder_top_elliptic}. Again gathering like terms in the Taylor formula \eqref{eq:v_taylor_polynomial_elliptic} yields
$$
v(x',x_d) = r + \left(\eta + o(x',x_d) - \frac{Q}{2}x_d\right)x_d
+ |x'|o(x',x_d) - \frac{Q}{2}|x'|^2,  \quad \forall\, x' \in \bar B_{\rho_0}^{d-1}.
$$
Applying \eqref{eq:eta_bound_elliptic} (which had required a small enough $\eta>0$) and \eqref{eq:o_xprime_hatxd_bound} (which had required a large enough $Q$) yields
\begin{align*}
v(x',x_d) &\leq r + \left(-\frac{p}{8m} - \frac{Q}{2}x_d\right)x_d  - \frac{p}{8m}|x'| - \frac{Q}{2}|x'|^2,
\\
&\qquad \forall\, (x',x_d) \in \bar B_{\rho_0}^{d-1}\times (0,\hat x_d).
\end{align*}
Therefore, when
$$
-\frac{p}{4mQ} < x_d < \hat x_d,
$$
the $x_d$-coefficient in the first-order Taylor formula for $v(x',x_d)$ becomes non-positive; recall that $\hat x_d = (\eta-p)/(mQ)$ (and $p<0$) by \eqref{eq:xd_bound_less_ell_elliptic}, so $\hat x_d$ obeys
$$
-\frac{p}{4mQ} < \hat x_d = \frac{\eta}{mQ} - \frac{p}{mQ}.
$$
Now we restrict $x'$ by requiring that $|x'|=\rho(x_d)>0$, for $x_d \in (0,\hat x_d)$, and choose $\rho(x_d)$ such that
$$
\left(-\frac{p}{8m} - \frac{Q}{2}x_d\right)x_d  - \frac{p}{8m}\rho(x_d) - \frac{Q}{2}\rho^2(x_d) \leq 0, \quad 0<  x_d < \hat x_d,
$$
that is,
$$
\frac{Q}{2}\rho^2(x_d) + \frac{p}{8m}\rho(x_d) + \left(\frac{p}{8m} + \frac{Q}{2}x_d\right)x_d \geq 0, \quad 0<  x_d < \hat x_d,
$$
and so $\rho(x_d)$ should obey
\begin{equation}
\label{eq:rho_value_cylinder_side_elliptic}
\rho(x_d)
\geq
\begin{cases}
0 &\hbox{if } x_d = 0,
\\
\displaystyle
\frac{1}{Q}\left(-\frac{p}{8m} - \sqrt{ \frac{p^2}{64m^2} - 2x_dQ\left(\frac{p}{8m} + \frac{Q}{2}x_d\right)}\right),
&\hbox{if }0 < x_d < \displaystyle-\frac{p}{4mQ},
\\
0 &\hbox{if }\displaystyle-\frac{p}{4mQ} \leq x_d < \hat x_d.
\end{cases}
\end{equation}
Note that the function on the right-hand side is bounded by $-p/(8mQ)$ and thus, for large enough $Q=Q(m,p,\rho_0)$, we may assume that $8mQ\leq\rho_0$. We now fix a continuous function, $\rho$ in \eqref{eq:Defn_rho_elliptic}, obeying \eqref{eq:rho_value_cylinder_top_elliptic} and \eqref{eq:rho_value_cylinder_side_elliptic} together with $\rho(0)=\rho_0>0$ and $\rho(x_d)\leq\rho_0$ for $0\leq x_d\leq\hat x_d$, so the previous inequality for $v(x',x_d)$ becomes
%COMMENT In the parabolic version of the below, the inequality is strict
$$
v(x',x_d) \leq r, \quad\hbox{for } |x'| = \rho(x_d), \quad 0<  x_d < \hat x_d,
$$
as desired for Case \ref{case:Cylinder_side_elliptic}.
\end{case}
Combining the conclusions of Cases \ref{case:Cylinder_bottom_elliptic} through \ref{case:Cylinder_side_elliptic} implies that $v\leq r$ on the boundary, $\partial V$, as we had sought.
\end{step}

This completes the proof of Theorem \ref{thm:Viscosity_maximum_elliptic}, with open neighborhood, $\sU\subset\sO$, given by the preimage of $V \subset \RR^{d-1}\times\RR_+$ with respect to the initial coordinate transformation.
\end{proof}

%COMMENT Omit in journal version
\begin{rmk}[Motivation for the argument in Step \ref{step:Interior_maximum_elliptic}]
Since $D_{\vec n}v(O)=\eta>0$ is continuous with respect to $x' \in B_{2\rho_0}^{d-1}$, because this is true of $D_{\vec n}u$ by hypothesis, then we may assume without loss of generality (for small enough $\rho_0$) that
$$
v_{x_d}(x',0) > 0, \quad \forall\, x' \in \bar B_{\rho_0}^{d-1},
$$
and so the function $v$ cannot have a local maximum, relative to $B_{\rho_0}^{d-1}\times(0,\ell)$, on the boundary portion $\bar B_{\rho_0}^{d-1}\times\{0\}$.
\end{rmk}

\subsection{Strong maximum principles for $A$-subharmonic functions in $C^2(\sO)$ or $W^{2,d}_{\loc}(\sO)$}
\label{subsec:Strong_maximum_principle_elliptic}
We are now ready to prove versions of the strong maximum principle for $A$-subharmonic functions in $C^2(\sO)$ or $W^{2,d}_{\loc}(\sO)$. We first prove the analogue of \cite[Theorem 3.5]{GilbargTrudinger}.

\begin{thm}[Strong maximum principle for $A$-subharmonic functions in $C^2(\sO)$]
\label{thm:Strong_maximum_principle_C2_elliptic}
Let $\sO\subset\RR^d$ be a domain.\footnote{Recall that by a `domain' in $\RR^d$, we always mean a \emph{connected}, open subset.}
Assume the hypotheses of Theorem \ref{thm:Viscosity_maximum_elliptic} for the coefficients of $A$ in \eqref{eq:Generator}. Suppose that $u \in C^2(\sO)\cap C^1(\underline\sO)$ and $Au\leq 0$ on $\sO$.
%COMMENT We rechecked the proof of the Hopf lemma to see that we do not need c defined on \partial_0\sO
If $c=0$ (respectively, $c\geq 0$) on $\sO$ and $u$ attains a global maximum (respectively, non-negative global maximum) in $\underline\sO$, then $u$ is constant on $\sO$.
\end{thm}

\begin{proof}
If $u$ attains its global maximum (respectively, non-negative global maximum) $M$ at a strictly interior point $x^0\in\sO$, the conclusion follows from \cite[Theorem 3.5]{GilbargTrudinger}. Thus it suffices to consider the case of a point $\bar x^0\in \partial_0\sO$ with $u(\bar x^0)=M$ and $u<M$ on $\sO$. By Theorem \ref{thm:Viscosity_maximum_elliptic}, there is a function, $w\in C^2(\sO)\cap C^1(\RR^d)$, and an open subset, $\sU\subset\RR^d$, such that $\bar x^0\in \bar\sU\subset\underline\sO$ and $v=u+w$ is an $A$-subharmonic function on $\sU$, so $Au\leq 0$ on $\sU$, and $v$ attains a maximum $v(x^0)>u(\bar x^0)$ at an interior point $x^0\in \sU$. We can now apply \cite[Theorem 3.5]{GilbargTrudinger} to conclude that $v$ is constant on $\sU$, a contradiction since that would imply $v(x^0)=u(\bar x^0)$. This concludes the proof.
\end{proof}

\begin{rmk}[On the hypotheses for the classical strong maximum principle for $A$-subharmonic functions in $C^2(\sO)$]
\label{rmK:Strong_maximum_principle_C2_elliptic_classical_hypotheses}
We recall from \cite[p. 35]{GilbargTrudinger} that the classical strong maximum principle, \cite[Theorem 3.5]{GilbargTrudinger}, holds when the coefficient, $a:\sO\to\sS^+(d)$, of $A$ is only locally uniformly elliptic in the sense of \cite[p. 31]{GilbargTrudinger}, so the function $\Lambda/\lambda$ is locally bounded on $\sO$, and the functions $|b|/\lambda$ and $c/\lambda$ are locally bounded on $\sO$, where $\lambda:\sO\to[0,\infty)$ and $\Lambda:\sO\to[0,\infty)$ denote the minimum and maximum eigenvalue functions defined by $a:\sO\to\sS^+(d)$, respectively.
\end{rmk}

We next give the

\begin{proof}[Proof of Theorem \ref{thm:Strong_maximum_principle_W2d_elliptic}]
The proof is identical to that of Theorem \ref{thm:Strong_maximum_principle_C2_elliptic}, except that the role of classical strong maximum principle for $A$-subharmonic functions in $C^2(\sO)$, \cite[Theorem 3.5]{GilbargTrudinger}, is replaced by classical strong maximum principle for $A$-subharmonic functions in $W^{2,d}_{\loc}(\sO)$, \cite[Theorem 9.6]{GilbargTrudinger}.
\end{proof}

\begin{rmk}[Relaxing the hypothesis of a global maximum to a local maximum]
\label{rmk:Aronszajn_elliptic}
In Theorems \ref{thm:Strong_maximum_principle_C2_elliptic} or \ref{thm:Strong_maximum_principle_W2d_elliptic}, if $u$ only attains a \emph{local} maximum at a point $x^0\in\underline\sO$ and $c=0$ on $\sO$, so $u\leq u(x^0)$ on an open ball $\sO\cap B_\eps(x^0)$ for some $\eps>0$, we may still apply \cite[Theorem 3.5 or 9.6]{GilbargTrudinger} to conclude that $u$ is constant on $\sO\cap B_\eps(x^0)$ and therefore, provided $A$ has the unique continuation property (see, for example, \cite[Theorem 1]{Aronszajn} or \cite{Koch_Tataru_2001}), we can again conclude that $u$ is constant on $\sO$ by a standard argument (see, for just one of many examples, the proof of \cite[Theorem 1.3]{Kim_Lee_2011}).
\end{rmk}

\subsection{Weak maximum principles for $A$-subharmonic functions in $C^2(\sO)$ or $W^{2,d}_{\loc}(\sO)$}
\label{subsec:Weak_maximum_principle_elliptic}
Recall that we reserve the term \emph{domain} for a connected, open subset of $\RR^d$. The proof of Theorem \ref{thm:Weak_maximum_principle_C2_elliptic_domain} is of course standard \cite[Theorem 3.1 and Corollary 3.2]{GilbargTrudinger} when $\partial_0\sO$ is empty; when $\partial_0\sO$ is non-empty, Theorem \ref{thm:Viscosity_maximum_elliptic} provides the key new technical ingredient.

\begin{thm}[Weak maximum principle on domains for $A$-subharmonic functions in $C^2(\sO)$]
\label{thm:Weak_maximum_principle_C2_elliptic_domain}
Let $\sO\subset\RR^d$ be a bounded domain. Assume the hypotheses of Theorem \ref{thm:Viscosity_maximum_elliptic} for the coefficients of $A$ in \eqref{eq:Generator} and that
\begin{gather}
\label{eq:Nonempty_nondegenerate_boundary_elliptic_domain}
\partial_1\sO \hbox{ is non-empty, \emph{or}}
\\
\label{eq:Empty_nondegenerate_boundary_elliptic_domain}
\tag{\ref*{eq:Nonempty_nondegenerate_boundary_elliptic_domain}$'$}
\hbox{$c>0$ at some point in } \sO.
\end{gather}
Suppose $u \in C^2(\sO)\cap C^1(\underline\sO)$ and $\sup_\sO u<\infty$.
%COMMENT We need to assume u is bounded above on \sO since that is not guaranteed as we do not assume continuity up to the boundary and unless we assume boundedness, u might not attain even a local maximum on \sO
If $Au\leq 0$ on $\sO$ and $u^* \leq 0$ on $\partial_1\sO$, then $u\leq 0$ on $\sO$.
\end{thm}
%COMMENT The result is valid even when $\partial\sO\less\partial_0\sO$ is an isolated set of points
%TODO So why not rephrase in this way?

\begin{proof}
Since $u^*$ is upper semicontinuous on $\bar\sO$ and $\sup_\sO u<\infty$, then $u^*$ attains a finite maximum at some point $x^0\in\bar\sO$; if $u^*(x^0)<0$, we are done, so we may suppose that $u^*(x^0)\geq 0$. If $x^0\in\underline\sO$, then $u$ must be constant on $\bar\sO$, by Theorem \ref{thm:Strong_maximum_principle_C2_elliptic}. If $\partial_1\sO$ is non-empty, we must have that $u^*(x^0)\leq 0$, since $u^*\leq 0$ on $\partial_1\sO$ by hypothesis. If $\partial_1\sO$ is empty and $c>0$ at some point in $\sO$, then, because $Au = cu\leq 0$ on $\sO$, we obtain $u^*(x^0)=u(x^0)\leq 0$.
\end{proof}

By replacing the role of Theorem \ref{thm:Strong_maximum_principle_C2_elliptic} by that of Theorem \ref{thm:Strong_maximum_principle_W2d_elliptic} in the proof of Theorem \ref{thm:Weak_maximum_principle_C2_elliptic_domain}, we obtain an analogue of the weak maximum principle implied by the Aleksandrov-Bakelman-Pucci estimate \cite[Theorem 9.1]{GilbargTrudinger} when $f=0$ on $\sO$.

\begin{thm}[Weak maximum principle on domains for $A$-subharmonic functions in $W^{2,d}_{\loc}(\sO)$]
\label{thm:Weak_maximum_principle_W2d_elliptic_domain}
Let $\sO\subset\RR^d$ be a bounded domain. Assume the hypotheses of Theorem \ref{thm:Viscosity_maximum_elliptic} for the coefficients of $A$ in \eqref{eq:Generator}, which we require to be measurable and that
\eqref{eq:Nonempty_nondegenerate_boundary_elliptic_domain} holds or that
\begin{equation}
\label{eq:Empty_nondegenerate_boundary_elliptic_domain_measurable}
\tag{\ref*{eq:Nonempty_nondegenerate_boundary_elliptic_domain}$''$}
\hbox{$c>0$ on a set of positive measure in } \sO.
\end{equation}
Suppose $u\in W^{2,d}_{\loc}(\sO)\cap C^1(\underline\sO)$
and\,\footnote{Our hypothesis that $u\in C^1(\underline\sO)$ could be relaxed to the slightly more technical hypothesis that $Du$ be defined and continuous along $\partial_0\sO$.}
$\sup_\sO u < \infty$. If $Au\leq 0$ a.e. on $\sO$ and $u^* \leq 0$ on $\partial_1\sO$, then $u\leq 0$ on $\sO$.
\end{thm}

We can also deduce versions of the weak maximum principle when $\sO$ is not necessarily connected, albeit with the addition of a slightly technical condition on the boundary conditions for isolated components.\footnote{Suppose $\sO$ is expressed as a union of its connected components, $\sO=\cup_{n\in\NN}\sC_n$, where each $\sC_n\subset\sO$ is an open subset (since $\sO\subset\RR^d$ is open and thus necessarily locally connected) and $\sC_n\cap\bar\sC_m = \emptyset$ for all $m,n$; we call a component, $\sC_i$, \emph{isolated} if $\bar\sC_i\cap\bar\sC_n=\emptyset$ for all $n\neq i$.}

\begin{cor}[Weak maximum principle on open subsets for $A$-subharmonic functions in $C^2(\sO)$]
\label{cor:Weak_maximum_principle_C2_elliptic_opensubset}
Let $\sO\subset\RR^d$ be a bounded, open subset. Assume the hypotheses of Theorem \ref{thm:Viscosity_maximum_elliptic} for the coefficients of $A$ in \eqref{eq:Generator}. If $\sC$ is an isolated component of $\sO$ and $\sO'$ is the complement in $\sO$ of the union of isolated components, require that $\sO'$ and each isolated component, $\sC$, obey \eqref{eq:Nonempty_nondegenerate_boundary_elliptic_domain} or \eqref{eq:Empty_nondegenerate_boundary_elliptic_domain}. Suppose $u \in C^2(\sO)\cap C^1(\underline\sO)$ and $\sup_\sO u<\infty$. If $Au\leq 0$ on $\sO$ and $u^* \leq 0$ on $\partial_1\sO$, then $u\leq 0$ on $\sO$.
\end{cor}

%COMMENT A simpler proof would be nice!
\begin{proof}
We write $\sO = \cup_{n\in\NN}\sC_n$ as a union of connected open components. Note that $\partial\sC_n\subset\partial\sO$ for each $n\in\NN$. If $\sC_i$ is an \emph{isolated} component, so $\bar\sC_i\cap\bar\sC_n$ is empty for all $n\neq i$, we find that $u\leq 0$ on $\sC_i$ by Theorem \ref{thm:Weak_maximum_principle_C2_elliptic_domain}. Hence, it remains to consider the case where no connected component of $\sO$ is isolated. If $\partial_1\sO$ is non-empty, let $\sC_0$ denote a component such that $\bar\sC_0\cap\partial_1\sO$ is non-empty; if $\partial_1\sO$ is empty, let $\sC_0$ denote a component containing a point where $c>0$. Theorem \ref{thm:Weak_maximum_principle_C2_elliptic_domain} implies that $u\leq 0$ on $\sC_0$ in either case. If $\sC_1$ is a component such that $\bar\sC_1\cap\bar\sC_0$ is non-empty, then we must also have $u\leq 0$ on $\sC_1$. By induction on $n\in\NN$ and the fact that no component $\sC_n$ of $\sO$ is isolated, by our reduction, we discover that $u\leq \sC_n$ for all $n\in\NN$.
\end{proof}

The next result follows from Theorem \ref{thm:Weak_maximum_principle_W2d_elliptic_domain} in exactly the same way that Corollary \ref{cor:Weak_maximum_principle_C2_elliptic_opensubset} follows from Theorem \ref{thm:Weak_maximum_principle_C2_elliptic_domain}.

\begin{cor}[Weak maximum principle on open subsets for $A$-subharmonic functions in $W^{2,d}_{\loc}(\sO)$]
\label{cor:Weak_maximum_principle_W2d_elliptic_opensubset}
Let $\sO\subset\RR^d$ be a bounded, open subset. Assume the hypotheses of Theorem \ref{thm:Viscosity_maximum_elliptic} for the coefficients of $A$ in \eqref{eq:Generator}, which we require to be measurable. If $\sC$ is an isolated component of $\sO$ and $\sO'$ is the complement in $\sO$ of the union of isolated components, require that $\sO'$ and each isolated component, $\sC$, obey \eqref{eq:Nonempty_nondegenerate_boundary_elliptic_domain} or \eqref{eq:Empty_nondegenerate_boundary_elliptic_domain_measurable}. Suppose $u\in W^{2,d}_{\loc}(\sO)\cap C^1(\underline\sO)$ and\,\footnote{Our hypothesis that $u\in C^1(\underline\sO)$ could be relaxed to the slightly more technical hypothesis that $Du$ be defined and continuous along $\partial_0\sO$.}
$\sup_\sO u < \infty$. If $Au\leq 0$ a.e. on $\sO$ and $u^* \leq 0$ on $\partial_1\sO$, then $u\leq 0$ on $\sO$.
\end{cor}

\subsection{Weak maximum principles for $A$-subharmonic functions in $C^2(\sO)$ or $W^{2,d}_{\loc}(\sO)$ and relaxed hypotheses on the coefficients}
\label{subsec:Relaxing_coefficients_elliptic}
Before we can relax the hypotheses of Theorems \ref{thm:Weak_maximum_principle_C2_elliptic_domain} and \ref{thm:Weak_maximum_principle_W2d_elliptic_domain} and Corollaries \ref{cor:Weak_maximum_principle_C2_elliptic_opensubset} and \ref{cor:Weak_maximum_principle_W2d_elliptic_opensubset} on the coefficients $a$ and $b$ of $A$ in \eqref{eq:Generator}, we shall need a priori maximum principle estimates. To state these properties in some generality, it is convenient to make use of \cite[Definition 2.8]{Feehan_maximumprinciple}; compare \cite[p. 292]{Trudinger_1977}. Given a real vector space, $V$, recall that a \emph{convex cone}, $\fK\subset V$, is a subset such that if $u, v \in \fK$ and $\alpha,\beta \in \bar\RR_+$, then $\alpha u + \beta v \in \fK$.

\begin{defn}[Weak maximum principle property for $A$-subharmonic functions in $C^2(\sO)$ or $W^{2,d}_{\loc}(\sO)$]
\label{defn:Weak_max_principle_property}
Let $\sO\subset\RR^d$ be an open subset, let $\Sigma \subseteqq \partial\sO$ be an open subset, and let $\fK\subset C^2(\sO)$ (respectively, $W^{2,d}_{\loc}(\sO)$) be a convex cone. We say that an operator $A$ in \eqref{eq:Generator} obeys the \emph{weak maximum principle property on $\sO\cup\Sigma$ for $\fK$} if whenever $u\in \fK$ obeys
$$
Au \leq 0 \quad \hbox{(a.e.) on } \sO \quad\hbox{and}\quad u^* \leq 0 \quad\hbox{on } \partial\sO\less\Sigma,
$$
then
$$
u \leq 0 \quad\hbox{on } \sO.
$$
\end{defn}

\begin{rmk}[Examples for $A$-subharmonic functions in $C^2(\sO)$ or $W^{2,d}_{\loc}(\sO)$ of the weak maximum principle property]
\label{rmk:Examples_weak maximum principle property_C2_W2d_elliptic}
One can find examples of subsets $\sO$ and $\Sigma\subseteqq\partial\sO$, operators $A$, and cones $\fK$ yielding the weak maximum principle property in the following settings.
\begin{enumerate}
\item In \cite[Theorem 3.1 \& Corollary 3.2]{GilbargTrudinger} (respectively, \cite[Theorem 9.1]{GilbargTrudinger}), where $\sO$ is bounded, one takes $\Sigma = \emptyset$ and $\fK = C^2(\sO)\cap C(\bar\sO)$ (respectively, $W^{2,d}_{\loc}(\sO)\cap C(\bar\sO)$).
\item In \cite[Theorem 2.9.2]{Krylov_LecturesHolder}, where $\sO$ may be unbounded, one takes $\Sigma = \emptyset$ and $\fK$ to be the set of $u\in C^2(\sO)$ such that $\sup_\sO u < \infty$.
\item In Corollary \ref{cor:Weak_maximum_principle_C2_elliptic_opensubset} (respectively, Corollary \ref{cor:Weak_maximum_principle_W2d_elliptic_opensubset}), where $\sO$ is bounded, one takes $\Sigma = \partial_0\sO$ and $\fK = C^2(\sO)\cap C_1(\underline\sO)$ (respectively, $u\in W^{2,d}_{\loc}(\sO)\cap C_1(\underline\sO)$) and $\sup_\sO u < \infty$.
\item In \cite[Theorem 5.1]{Feehan_maximumprinciple}, where $\sO$ is bounded, one takes $\Sigma = \partial_0\sO$ and $\fK$ to be the set of $u\in C^2(\sO)\cap C^1(\underline\sO)$ such that $\lim_{\sO\ni x\to x^0}\tr(aD^2u)(x)=0$ for all $x^0\in\partial_0\sO$ and $\sup_\sO u < \infty$.
\item In \cite[Theorem 5.3]{Feehan_maximumprinciple}, where $\sO$ may be unbounded, one takes $\Sigma = \partial_0\sO$ and $\fK$ to be the set of $u\in C^2(\sO)\cap C^1(\underline\sO)$ such that $\lim_{\sO\ni x\to x^0}\tr(aD^2u)(x)=0$ for all $x^0\in\partial_0\sO$ and $\sup_\sO u < \infty$.
\end{enumerate}
\end{rmk}

\begin{rmk}[Weak maximum principle property for viscosity subsolutions]
\label{rmk:Example_weak maximum principle property_viscosity_elliptic}
Suppose that the coefficients of $A$ in \eqref{eq:Generator} obey the hypotheses of \cite[Theorem 3.3 and Example 3.6]{Crandall_Ishii_Lions_1992}, so $c$ is continuous on $\sO$ and $c\geq c_0$ on $\sO$ for some positive constant, $c_0$; the vector field $b$ is continuous on $\sO$ and obeys $\langle b(x)-b(y), x-y\rangle \geq -b_0|x-y|^2$ for some positive constant $b_0$; and $a = \sigma^*\sigma$ where $\sigma:\sO\to\RR^{d\times d}$ is uniformly Lipschitz continuous. Then \cite[Theorem 3.3 and Example 3.6]{Crandall_Ishii_Lions_1992}, when $\sO$ is bounded, imply that $A$ has the weak maximum principal property when $\Sigma = \emptyset$ and $\fK$ is the set of upper semicontinuous functions on $\bar\sO$.
\end{rmk}

We have the following analogue of \cite[Proposition 2.19]{Feehan_maximumprinciple}.

\begin{prop}[Weak maximum principle estimates for $A$-subharmonic functions in $C^2(\sO)$ or $W^{2,d}_{\loc}(\sO)$]
\label{prop:Elliptic_weak_max_principle_apriori_estimates}
Let $\sO\subset\RR^d$ be an open subset, open subset $\Sigma\subseteqq\partial\sO$, and $A$ in \eqref{eq:Generator} have the weak maximum principle property on $\sO\cup\Sigma$ in the sense of Definition \ref{defn:Weak_max_principle_property}, for a convex cone $\fK\subset C^2(\sO)$ (respectively, $W^{2,d}_{\loc}(\sO)$) containing the constant function $1$, and $c$ obey \eqref{eq:c_nonnegative_elliptic}. Suppose that $u\in \fK$.
\begin{enumerate}
\item\label{item:Subsolution_Au_leq_zero} If $Au\leq 0$ (a.e.) on $\sO$, then
$$
u\leq 0 \vee \sup_{\partial\sO\less\Sigma}u^* \quad\hbox{on } \sO.
$$
\item\label{item:Subsolution_Au_arb_sign} If $c\geq c_0$ on $\sO$, for a positive constant $c_0$, then
$$
u\leq 0 \vee \frac{1}{c_0}\esssup_\sO Au \vee \sup_{\partial\sO\less\Sigma}u^* \quad\hbox{on } \sO.
$$
\end{enumerate}
When $\partial\sO\less\Sigma$ has empty interior, the terms $\sup_{\partial\sO\less\Sigma}u^*$ are omitted; when $u\in C^2(\sO)$, then $\esssup_\sO Au$ may be replaced by $\sup_\sO Au$.
\end{prop}

We provide an application of Proposition \ref{prop:Elliptic_weak_max_principle_apriori_estimates} to the question of uniqueness for solutions to an obstacle problem, via the following analogue of \cite[Theorem 1.3.4]{Friedman_1982}.

\begin{thm}[Comparison principle and uniqueness for $W^{2,d}_{\loc}$ solutions to the obstacle problem]
\label{thm:Comparison_principle_elliptic_obstacle_problem}
Let $\sO\subset\RR^d$ be an open subset, let $\Sigma \subseteqq \partial\sO$ be an open subset, and let $\fK\subset W^{2,d}_{\loc}(\sO)$ be a convex cone. For every open subset $\sU\subset\sO$, let $A$ in \eqref{eq:Generator} have the weak maximum principle property on $\sU\cup\Sigma$ in the sense of Definition \ref{defn:Weak_max_principle_property}
\footnote{Note that the weak maximum principle property hypothesis on $A$ here is stronger than that in Proposition \ref{prop:Elliptic_weak_max_principle_apriori_estimates}.}.
Let $f\in L^d_{\loc}(\sO)$ and $\psi\in L^d_{\loc}(\sO)$.
%COMMENT We do not need to assume that psi is bounded above or continuous
Suppose $u\in \fK$ (respectively, $v\in -\fK$) is a solution (respectively, supersolution) to the obstacle problem,
$$
\min\{Au-f, \ u-\psi\} = 0 \ (\geq  0) \quad\hbox{a.e. on } \sO.
$$
If $v_*\geq u^*$ on $\partial\sO\less\Sigma$, then $v \geq u$ on $\sO$; if $u, v$ are solutions and $v_* = u^*$ on $\partial\sO\less\Sigma$, then $v = u$ on $\sO$.
\end{thm}

\begin{proof}
The proof is identical to that of \cite[Theorem 1.3.4]{Friedman_1982}, except that we appeal to our version of the weak maximum principle (Definition \ref{defn:Weak_max_principle_property}) rather than \cite[Theorem 9.1]{GilbargTrudinger} and so the role of $\partial\sO$ is replaced by $\partial\sO\less\Sigma$.
\end{proof}

We now use Proposition \ref{prop:Elliptic_weak_max_principle_apriori_estimates} to extend our previous versions of the weak maximum principle. We begin with a simple extension of the classical maximum weak maximum principle for $A$-subharmonic functions in $C^2(\sO)$ \cite[Theorem 3.1 and Corollary 3.2]{GilbargTrudinger} using elliptic regularization (see, for example, the proof of \cite[Theorem 6.5]{Crandall_Ishii_Lions_1992}; other versions of the extension are noted in a remark immediately following the statement of \cite[Theorem 3.1]{GilbargTrudinger} and in \cite[p. 33, top of page]{GilbargTrudinger}.

\begin{thm}[Classical weak maximum principle for $A$-subharmonic functions in $C^2(\sO)$ and nonnegative characteristic form]
\label{thm:Classical_weak_maximum_principle_C2_elliptic_relaxed}
Let $\sO\subset\RR^d$ be a bounded, open subset and $A$ in \eqref{eq:Generator} with $a:\sO\to\sS(d)$ nonnegative, and $b$ locally bounded on $\sO$, and $c$ obeys \eqref{eq:c_lower*_positive_elliptic}, that is, $c_*>0$ on $\sO$. Suppose $u\in C^2(\sO)$ and $\sup_\sO u < \infty$. If $Au \leq 0$ on $\sO$ and $u_*\leq 0$ on $\partial\sO$, then $u\leq 0$ on $\sO$.
\end{thm}

\begin{proof}
For any $\eps>0$, let $A_\eps u := Au + \eps\Delta u$, where $\Delta u := -\sum_{i=1}^du_{x_ix_i}$. Let $\{\sO_l\}_{l\in\NN}$ be an exhaustion of $\sO$ by open subsets such that $\sO_l\subseteqq \sO_{l+1}$ and $\sO = \cup_{l\in\NN}\sO_l$ and $\sO_l\Subset \sO$ for all $l\geq 0$. Since $c_*>0$ on $\sO$ by hypothesis, there is a positive constant, $c_l$, such that $c\geq c_l$ on $\sO_l$, for each $l\in\NN$. By \cite[Corollary 3.2]{GilbargTrudinger} and Proposition \ref{prop:Elliptic_weak_max_principle_apriori_estimates} \eqref{item:Subsolution_Au_arb_sign}, we have
$$
u \leq 0\vee \sup_{\partial \sO_l}u \vee\frac{1}{c_l}\sup_{\sO_l} A_\eps u \leq 0\vee \sup_{\partial \sO_l}u \vee\frac{\eps}{c_l}\sup_{\sO_l} \Delta u\quad\hbox{on }\sO_l,
$$
for all $\eps>0$ and $l\in \NN$. Letting $\eps\downarrow 0$ and noting that $u\in C^2(\bar\sO_l)$, we obtain
$$
u \leq 0\vee \sup_{\partial \sO_l}u, \quad\forall\, l \in \NN.
$$
But the preceding inequality holds for every $l\in\NN$ and as $\sO = \cup_{l\in\NN}\sO_l$ and taking the limit as $l\to\infty$, we obtain
$$
u \leq 0 \vee \lim_{l\to\infty}\sup_{\partial\sO_l}u \quad\hbox{on } \sO.
$$
Since
$$
\lim_{l\to\infty}\sup_{\partial\sO_l}u = \sup_{\partial\sO}u^*,
$$
and $u*\leq 0$ on $\partial\sO$, we obtain
$$
u \leq 0 \quad\hbox{on } \sO,
$$
as desired.
\end{proof}

We can now give the

\begin{proof}[Proof of Theorem \ref{thm:Weak_maximum_principle_C2_elliptic_domain_relaxed}]
We proceed by modifying the proof of Theorem \ref{thm:Classical_weak_maximum_principle_C2_elliptic_relaxed}. We consider two cases.

\setcounter{case}{0}
\begin{case}[$a \in C^{0,1}(\underline\sO)$]
\label{case:a_lipschitz}
Let $\{\sO_l\}_{l\in\NN}$ be an exhaustion of $\sO$ by subdomains such that $\sO_l\subseteqq \sO_{l+1}$ and $\sO = \cup_{l\in\NN}\sO_l$ and $\sO_l\Subset\sO$ and $\partial\sO_l$ is $C^1$, for all $l\geq 0$. Choose $l_0\in\NN$ large enough that $\diam \sO <2/l_0$ and, for each $l\in\NN$, define $\Sigma_l\subset\partial\sO_l$ to be the interior of the set of points $x\in\partial\sO_l$ such that $\dist(x,\partial_0\sO) < 1/(l+l_0)$. Define $\Gamma_l := \partial\sO_l\less\bar\Sigma_l$. Therefore, $\Sigma_l\to\partial_0\sO$ and $\Gamma_l\to \partial_1\sO$ as $l\to\infty$. For each $l\in\NN$, choose $\zeta_l\in C^\infty(\RR^d)$ such that $0<\zeta_l\leq 1$ on $\sO_l$ and $\zeta = 0$ on $\Sigma_l$.

Because $u\in C^1(\underline\sO)$ by hypothesis and $b \in C(\underline\sO;\RR^d)$ by hypothesis \eqref{eq:b_continuous_on_domain_plus_degenerate_boundary}, we have
$$
Du, \ b \in C(\bar\sO_l;\RR^d).
$$
Choose a sequence of vector fields obeying \eqref{eq:b_C2_elliptic}, namely $\{b_k\}_{k\in\NN} \subset C^2(\underline\sO;\RR^d)$, such that $b_k$ converges to $b$ in $C(\underline\sO;\RR^d)$ (that is, uniformly on compact subsets of $\underline\sO$) as $k\to\infty$, and each $b_k^\perp$ obeys \eqref{eq:b_perp_positive_boundary_elliptic} with $\partial_0\sO$ replaced by $\Sigma_l$ for large enough $l$, say $l \geq l_0$ for a suitable $l_0\in\NN$, so $b_k^\perp>0$ on $\Sigma_l$. We can ensure that the preceding condition on $b_k^\perp$ holds since $b \in C(\underline\sO;\RR^d)$ and $b^\perp\geq 0$ on $\partial_0\sO$ by hypothesis \eqref{eq:b_perp_nonnegative_boundary_elliptic} and $\Sigma_l\to\partial_0\sO$ as $l\to\infty$.

We see that, for each $l\in\NN$,
$$
\|Du\|_{C(\bar\sO_l)} < \infty \quad\hbox{and}\quad \|b-b_k\|_{C(\bar\sO_l)} \to 0 \quad\hbox{as } k\to \infty.
$$
For each $\eps>0$ and $k\in\NN$, define
$$
A_{\eps kl}u := Au + \eps\zeta_l\Delta u + \langle b-b_k, Du\rangle \quad\hbox{on }\sO,
$$
and observe that
$$
A_{\eps kl}u = - \tr((a+\eps\zeta_l I_d)D^2u) - \langle b_k, Du\rangle + cu \quad\hbox{on } \sO,
$$
where $I_d\in \sS(d)$ is the identity matrix. Clearly, the coefficients of $A_{\eps kl}$ obey \eqref{eq:a_locally_strictly_elliptic}, \eqref{eq:b_perp_positive_boundary_elliptic},
\eqref{eq:a_locally_lipschitz_elliptic} along $\Sigma_l$ (in place of along $\partial_0\sO$), and \eqref{eq:b_C2_elliptic}, in addition, of course, to \eqref{eq:c_locally_bounded_above_near_boundary_elliptic} and \eqref{eq:Empty_nondegenerate_boundary_elliptic_domain}. The weak maximum principle, in the form of Theorem \ref{thm:Weak_maximum_principle_C2_elliptic_domain}, now applies, with $\sO$ replaced by $\sO_l$ and $A$ replaced by $A_{\eps kl}$.

Because $Au\leq 0$ on $\sO$, we have
$$
A_{\eps kl}u \leq \eps\zeta_l\Delta u + \langle b-b_k, Du\rangle \quad\hbox{on }\sO.
$$
Since $c_*>0$ on $\sO$ by hypothesis, there is a positive constant, $c_l$, such that $c\geq c_l$ on $\sO_l$, for each $l\in\NN$. Therefore, the weak maximum principle estimate, Proposition \ref{prop:Elliptic_weak_max_principle_apriori_estimates} \eqref{item:Subsolution_Au_arb_sign} (with $\sO$ replaced by $\sO_l$ and $A$ replaced by $A_{\eps kl}$), gives
\begin{align*}
u &\leq 0 \vee \frac{1}{c_l}\sup_{\sO_l} A_{\eps kl} u \vee \sup_{\Gamma_l}u
\\
&\leq 0 \vee \frac{1}{c_l}\left(\eps\|\Delta u\|_{C(\bar\sO_l)} + \|Du\|_{C(\bar\sO_l)} \|b-b_k\|_{C(\bar\sO_l)} \right) \vee \sup_{\Gamma_l}u
\quad\hbox{on } \sO_l.
\end{align*}
First taking the limit as $\eps\downarrow 0$ yields,
$$
u \leq 0 \vee \frac{1}{c_l}\|Du\|_{C(\bar\sO)} \|b-b_k\|_{C(\bar\sO)} \vee \sup_{\Gamma_l}u,
\quad\forall\, k, l \in \NN,
$$
and then taking the limit as $k\to \infty$ gives
$$
u \leq 0 \vee \sup_{\Gamma_l}u \quad\hbox{on } \sO_l, \quad\forall\, l \in \NN.
$$
We conclude the proof of this case just as in the proof of Theorem \ref{thm:Classical_weak_maximum_principle_C2_elliptic_relaxed}.
\end{case}

\begin{case}[$a$ obeys \eqref{eq:a_locally_strictly_elliptic}]
We choose $\{\sO_l\}_{l\in\NN}$ and $\{b_k\}_{k\in\NN}$ as in Case \ref{case:a_lipschitz}, except that we now require $\sO_l\Subset\underline\sO$ rather than $\sO_l\Subset\sO$, for all $l\geq 0$, but may omit the requirement that each $\partial\sO_l$ is $C^1$. Let $\{\psi_l\}_{l\in\NN} \subset C^\infty(\RR^d)$ be a sequence of cutoff functions such that
$0\leq\psi_l\leq 1$ on $\RR^d$, and $\psi_l=1$ on $\sO_l$, and $\bar\sO\cap\supp\psi_l\subset \underline\sO$. Because $u\in C^1(\underline\sO)$ by hypothesis and $b \in C(\underline\sO;\RR^d)$ by \eqref{eq:b_continuous_on_domain_plus_degenerate_boundary}, we have
$$
\psi_l Du, \ \psi_l b \in C(\bar\sO;\RR^d).
$$
We see that, for each $l\in\NN$,
$$
\|\psi_l Du\|_{C(\bar\sO)} < \infty \quad\hbox{and}\quad \|\psi_l(b-b_k)\|_{C(\bar\sO)} \to 0 \quad\hbox{as } k\to \infty.
$$
For each $k,l\in\NN$, define
$$
A_{kl}u := Au + \psi_l^2\langle b-b_k, Du\rangle \quad\hbox{on }\sO,
$$
and observe that
$$
A_{kl}u = - \tr(aD^2u) - \langle b_k, Du\rangle + cu \quad\hbox{on } \sO_l,
$$
while
$$
A_{kl}u \leq \psi_l^2\langle b-b_k, Du\rangle \quad\hbox{on }\sO.
$$
The weak maximum principle, in the shape of Theorem \ref{thm:Weak_maximum_principle_C2_elliptic_domain}, again applies, with $\sO$ replaced by $\sO_l$ and $A$ replaced by $A_{kl}$ and the remainder of the argument is very similar to that in Case \ref{case:a_lipschitz}.
\end{case}
Combining the conclusions of the two cases completes the proof.
\end{proof}

We also have the

\begin{thm}[Weak maximum principle on domains for $A$-subharmonic functions in $W^{2,d}_{\loc}(\sO)$ and relaxed conditions on $a,b$]
\label{thm:Weak_maximum_principle_W2d_elliptic_domain_relaxed}
Assume the hypotheses of Theorem \ref{thm:Weak_maximum_principle_C2_elliptic_domain_relaxed} on $\sO$ and $A$, except that the coefficients of $A$ are now required to be measurable. Suppose $u \in W^{2,d}_{\loc}(\sO)\cap C^1(\underline\sO)$ and $\sup_\sO u<\infty$. If $Au\leq 0$ a.e. on $\sO$ and $u^* \leq 0$ on $\partial_1\sO$, then $u\leq 0$ on $\sO$.
\end{thm}

\begin{proof}
The proof is the same as that of Theorem \ref{thm:Weak_maximum_principle_C2_elliptic_domain_relaxed}, except that the role of Theorem \ref{thm:Weak_maximum_principle_C2_elliptic_domain} is replaced by that of Theorem \ref{thm:Weak_maximum_principle_W2d_elliptic_domain}.
\end{proof}

\begin{rmk}[Alternative hypotheses on $b$ and $c$]
\label{rmk:Weak_maximum_principle_C2_elliptic_relaxed}
There is a trade-off between requirements of uniform positive lower bounds on $\langle b,\vec n\rangle$ and $c$ which can sometimes be exploited. For example, suppose $\sO\subset\RR^{d-1}\times(0,\nu)$, for some positive constant $\nu$, and $\partial_0\sO\subset \RR^{d-1}\times\{0\}$, so $\vec n = e_d$. If $\sup_\sO a^{dd} < \infty$, then a requirement that $c\geq c_0$ on $\sO$ for some positive constant $c_0$ (which obviously implies \eqref{eq:c_lower*_positive_elliptic}), can be relaxed to $c\geq 0$ on $\sO$ if $b^d \geq b_0$ on $\sO$, for some positive constant $b_0$. See the proof of \cite[Lemma A.1]{Feehan_Pop_elliptichestonschauder}.
\end{rmk}

One can immediately deduce versions of the weak maximum principle (Theorems \ref{thm:Weak_maximum_principle_C2_elliptic_domain_relaxed} or \ref{thm:Weak_maximum_principle_W2d_elliptic_domain_relaxed}) when $\sO$ is not necessarily connected, by the same argument as used in the proofs of Corollaries \ref{cor:Weak_maximum_principle_C2_elliptic_opensubset} or \ref{cor:Weak_maximum_principle_W2d_elliptic_opensubset}; moreover, the slightly stronger hypothesis \eqref{eq:c_lower*_positive_elliptic} on $c$ yields statements which are more easily applied.

\begin{cor}[Weak maximum principle on open subsets for $A$-subharmonic functions in $C^2(\sO)$ and relaxed conditions on $a, b$]
\label{cor:Weak_maximum_principle_C2_elliptic_opensubset_relaxed}
Let $\sO\subset\RR^d$ be a bounded, open subset and assume the hypotheses of Theorem \ref{thm:Weak_maximum_principle_C2_elliptic_domain_relaxed} for $A$ in \eqref{eq:Generator}. Suppose $u \in C^2(\sO)\cap C^1(\underline\sO)$ and $\sup_\sO u < \infty$. If $Au\leq 0$ on $\sO$ and $u^* \leq 0$ on $\partial_1\sO$, then $u\leq 0$ on $\sO$.
\end{cor}

\begin{cor}[Weak maximum principle on open subsets for $A$-subharmonic functions in $W^{2,d}_{\loc}(\sO)$ and relaxed conditions on $a, b$]
\label{cor:Weak_maximum_principle_W2d_elliptic_opensubset_relaxed}
Let $\sO\subset\RR^d$ be a bounded, open subset and assume the hypotheses of Theorem \ref{thm:Weak_maximum_principle_W2d_elliptic_domain_relaxed} for $A$ in \eqref{eq:Generator}. Suppose $u \in W^{2,d}_{\loc}(\sO)\cap C^1(\underline\sO)$ and $\sup_\sO u < \infty$. If $Au\leq 0$ a.e. on $\sO$ and $u^* \leq 0$ on $\partial_1\sO$, then $u\leq 0$ on $\sO$.
\end{cor}

\section{Maximum principles for the parabolic operator}
\label{sec:maximum_principles_parabolic}
Section \ref{subsec:Parabolic_sobolev_embedding} reviews definitions of the parabolic H\"older and parabolic Sobolev spaces, together with the parabolic Sobolev embedding
theorem\footnote{A precise statement appears very difficult to find in the literature, at least for researchers who do not read Russian.}
which we shall need in this section. Section \ref{subsec:Simplifying_coefficients_parabolic} provides a change of variables on $\RR^{d+1}$ to bringing the coefficients of $L$ into a standard form near $\mydirac_0\!\sQ$ (see Lemma \ref{lem:Simplifying_coefficients_parabolic}). In \S \ref{subsec:Hopf_boundary_point_lemma_parabolic}, we recall our version of the Hopf boundary point lemma for a boundary-degenerate parabolic linear second-order differential operator (Lemma \ref{lem:Degenerate_hopf_lemma_parabolic}). Section \ref{subsec:Perturb_degen_boundary_local_maximum_parabolic} contains the proof of Theorem \ref{thm:Viscosity_maximum_parabolic}. In \S \ref{subsec:Strong_maximum_principle_parabolic_C2}, we apply Theorem \ref{thm:Viscosity_maximum_parabolic} to prove strong maximum principles for $L$-subharmonic functions in $C^2(\sQ)\cap \sC^1(\underline\sQ)$ (Theorems \ref{thm:Strong_max_principle_C2_parabolic_c_geq_zero}, \ref{thm:Strong_max_principle_C2_parabolic_c_arb_sign},
\ref{thm:Strong_max_principle_C2_parabolic_c_geq_zero_refined}, and \ref{thm:Strong_max_principle_C2_parabolic_c_arb_sign_refined}), while \S \ref{subsec:Strong_maximum_principle_parabolic_W2d+1} contains the corresponding application to the strong maximum principle for $L$-subharmonic functions in $W^{2,d+1}_{\loc}(\sQ)\cap \sC^1(\underline\sQ)$ (Theorem \ref{thm:Strong_max_principle_W2d+1_parabolic_c_geq_zero}). In \S \ref{subsec:Weak_maximum_principle_parabolic}, we apply the preceding strong maximum principles to deduce initial versions of the weak maximum principles for $L$-subharmonic functions in $C^2(\sQ)\cap \sC^1(\underline\sQ)$ and $W^{2,d+1}_{\loc}(\sQ)\cap \sC^1(\underline\sQ)$ (Theorems \ref{thm:Weak_maximum_principle_C2_parabolic_domain} and \ref{thm:Weak_maximum_principle_W2d+1_parabolic_domain}, respectively). In \S \ref{subsec:Relaxing_coefficients_parabolic}, we recall our definition of the `weak maximum principle property' (Definition \ref{defn:Weak_max_principle_property_parabolic}) for $L$-subharmonic functions in $C^2(\sQ)$ or $W^{2,d+1}_{\loc}(\sQ)$ and the corresponding weak maximum principle estimates (Proposition \ref{prop:Parabolic_weak_max_principle_apriori_estimates}). Given an operator, $L$, with the weak maximum principle property for $L$-subharmonic functions in $W^{2,d+1}_{\loc}(\sQ)$, we deduce a comparison principle, and thus uniqueness, for a solution and supersolution in $W^{2,d+1}_{\loc}(\sQ)$ to the unilateral obstacle problem (Theorem \ref{thm:Comparison_principle_parabolic_obstacle_problem}). We conclude by using our earlier versions of the weak maximum principles --- with strong hypotheses on the coefficients $a,b$ of $L$ --- and our weak maximum principle estimates to derive versions of the weak maximum principles, with more relaxed hypotheses on those coefficients, for $L$-subharmonic functions in $C^2(\sQ)\cap \sC^1(\underline\sQ)$ and $W^{2,d+1}_{\loc}(\sQ)\cap \sC^1(\underline\sQ)$ (Theorems \ref{thm:Weak_maximum_principle_C2_parabolic_domain_relaxed} and \ref{thm:Weak_maximum_principle_W2d+1_parabolic_domain_relaxed}, respectively).

\subsection{Parabolic Sobolev embeddings}
\label{subsec:Parabolic_sobolev_embedding}
See \cite[Section 2.1]{Feehan_parabolicmaximumprinciple} for related background. For an open subset $\sQ\subset\RR^{d+1}$ and $p\geq 1$, we say that (following Lieberman \cite[p. 155]{Lieberman})
\begin{equation}
\label{eq:Defn_W2p}
u \in W^{2,p}(\sQ)
\end{equation}
if $u$ is a measurable function on $\sQ$ and $u$ and its weak derivatives, $u_t$ and $u_{x_i}$ and $u_{x_ix_j}$ for $1\leq i,j\leq d$, belong to $L^p(\sQ)$ and similarly define $W^{2,p}_{\loc}(\sQ)$. Here, $W^{2,p}(\sQ)$ is a \emph{parabolic Sobolev space} \cite[\S 2.2]{Krylov_LecturesSobolev}, \cite[\S 1.1]{LadyzenskajaSolonnikovUralceva}, because we only assume $u_t \in L^p(\sQ)$ and do not, in addition, assume that $u_{tt} \in L^p(\sQ)$ and $u_{tx_i} \in L^p(\sQ)$ for $1\leq i\leq d$.

We will need the Sobolev embedding theorem for parabolic Sobolev spaces, so we first recall the definitions of the parabolic H\"older spaces that we require. We denote the parabolic distance between points $(t,x),(\tau,y) \in \RR^{d+1}$ by
$$
|t-\tau|^{1/2} + |x-y|, \quad\forall\, (t,x),(\tau,y) \in \RR^{d+1}.
$$
We have the

\begin{defn}[Parabolic $C^\alpha$ seminorm, norm, and Banach space]
\label{defn:C_alpha}
\cite[p. 117]{Krylov_LecturesHolder}.
Given $\alpha \in (0,1)$ and an open subset $\sQ\subset\RR^{d+1}$, we say that $u\in C^\alpha(\bar \sQ)$ if $u\in C(\bar \sQ)$ and
$$
\|u\|_{C^\alpha(\bar \sQ)} < \infty,
$$
where
\begin{equation}
\label{eq:C_alpha_norm}
\|u\|_{C^\alpha(\bar \sQ)} := [u]_{C^\alpha(\bar \sQ)} + \|u\|_{C(\bar \sQ)},
\end{equation}
and
\begin{equation}
\label{eq:C_alpha_seminorm}
[u]_{C^\alpha(\bar \sQ)} := \sup_{\begin{subarray}{c} (t,x), (\tau,y)\in \sQ \\ (t,x) \neq (\tau,y)\end{subarray}}\frac{|u(t,x)-u(\tau,y)|}{(|t-\tau|^{1/2} + |x-y|)^\alpha}.
\end{equation}
We say that $u\in C^\alpha(\sQ)$ if $u\in C^\alpha(\bar \sQ')$ for all precompact open subsets $\sQ'\Subset \sQ$.
\end{defn}

We shall need parabolic variants of the definitions of $C^1$ and $C^2$ functions; note, in particular, the distinction between $C^1(\sQ)$ and $\sC^1(\sQ)$.

\begin{defn}[Parabolic $C^1$ and $C^2$ functions]
\label{defn:C1_C2_function_parabolic}
Let $\sQ\subset\RR^{d+1}$ be an open subset. We say that
\begin{enumerate}
  \item $u\in C^1(\sQ)$ (respectively, $C^1(\bar \sQ)$) if $u, u_{x_i} \in C(\sQ)$ for $1\leq i\leq d$ (respectively, $C(\bar \sQ)$;
  \item $u\in \sC^1(\sQ)$ (respectively, $\sC^1(\bar \sQ)$) if $u, u_t, u_{x_i} \in C(\sQ)$ for $1\leq i\leq d$ (respectively, $C(\bar \sQ)$;
  \item $u\in C^2(\sQ)$ (respectively, $C^2(\bar \sQ)$) if $u, u_t, u_{x_i}, u_{x_ix_j} \in C(\sQ)$ for $1\leq i,j\leq d$ (respectively, $C(\bar \sQ)$
\end{enumerate}
\end{defn}

\begin{defn}[Parabolic $C^{1,\alpha}$ and $C^{2,\alpha}$ norms and Banach spaces]
\label{defn:C_1_2_alpha}
\cite[\S 1.1]{LadyzenskajaSolonnikovUralceva}, \cite[p. 117]{Krylov_LecturesHolder}. We say that $u\in C^{1,\alpha}(\bar \sQ)$ if $u\in C^1(\bar \sQ)$ and
$$
\|u\|_{C^{1,\alpha}(\bar \sQ)} := \|u\|_{C^\alpha(\bar \sQ)} + \|Du\|_{C^\alpha(\bar \sQ)} < \infty,
$$
and say that $u\in C^{2,\alpha}(\bar \sQ)$ if $u\in C^2(\bar \sQ)$ and
$$
\|u\|_{C^{2,\alpha}(\bar \sQ)} := \|u\|_{C^\alpha(\bar \sQ)} + \|u_t\|_{C^\alpha(\bar \sQ)} + \|Du\|_{C^\alpha(\bar \sQ)} + \|D^2u\|_{C^\alpha(\bar \sQ)}  < \infty.
$$
We say that $u\in C^{1,\alpha}(\sQ)$ if $u\in C^{1,\alpha}(\bar \sQ')$ for all precompact open subsets $\sQ'\Subset \sQ$ and similarly define $C^{2,\alpha}(\sQ)$.
\end{defn}

We can now state the parabolic Sobolev embedding theorem.\footnote{While the embedding $W^{2,\infty}(\sO_T) \hookrightarrow C^{1,1}(\bar \sO_T)$ should also hold, this is not asserted by \cite[Lemma 2.3.3]{LadyzenskajaSolonnikovUralceva}.}

\begin{thm}[Continuous embedding for parabolic Sobolev spaces]
\cite[Lemma 2.3.3]{LadyzenskajaSolonnikovUralceva}
\label{thm:Parabolic_sobolev_embedding}
Let $\sO \subset \RR^d$ be an open subset which obeys an interior cone condition, with cone $K$, and $0<T<\infty$. If $\alpha\in [0,1)$, then there are continuous embeddings
\begin{align*}
W^{2,p}(\sO_T) \hookrightarrow C^\alpha(\bar \sO_T), \quad\hbox{for } p > \frac{d}{2} + 1 \quad\hbox{and}\quad 0 \leq \alpha \leq 2 - (d+2)/p,
\\
W^{2,p}(\sO_T) \hookrightarrow C^{1,\alpha}(\bar \sO_T), \quad\hbox{for } p > d+2 \quad\hbox{and}\quad 0 \leq \alpha \leq 1 - (d+2)/p,
\end{align*}
with Sobolev embedding constants depending on $d,K,p,T,\alpha$. The embeddings are compact when $\alpha < 2 - (d+2)/p$ or $\alpha < 1 - (d+2)/p$, respectively.
\end{thm}

The two embeddings in Theorem \ref{thm:Parabolic_sobolev_embedding} follow by taking $l=1$ and (i) $r=s=0$ and (ii) $r=0, s=1$, respectively, in the second displayed equation in \cite[Lemma 2.3.3]{LadyzenskajaSolonnikovUralceva} (but writing $\alpha$ for $\lambda$, and $p$ for $q$, and $d$ for $n$ to be consistent with our notation) and keeping in mind the definitions in \cite[Equations (1.1.4)]{LadyzenskajaSolonnikovUralceva} for the parabolic Sobolev space and \cite[Equations (1.1.10--12)]{LadyzenskajaSolonnikovUralceva} for the parabolic H\"older spaces. Compactness of the embeddings follows from the Arzel\`a-Ascoli Theorem when the inequalities in the upper bounds for $\alpha$ are strict.

\subsection{Simplifying the coefficients}
\label{subsec:Simplifying_coefficients_parabolic}
Before proceeding to the proof of Theorem \ref{thm:Viscosity_maximum_parabolic}, we shall need the following parabolic analogue of Lemma \ref{lem:Simplifying_coefficients_elliptic}. Recall that we write the coordinates on $\RR^d$ as $x=(x',x_d)\in\RR^{d-1}\times\RR$.

\begin{lem}[Simplifying the coefficients]
\label{lem:Simplifying_coefficients_parabolic}
Let $\sQ\subset\RR^{d+1}$ be an open subset, with $C^{1,\alpha}$ boundary portion $\mydirac_0\!\sQ$ for some $\alpha\in(0,1)$, and $L$ be as in \eqref{eq:Generator_parabolic}. Assume that $a$ obeys \eqref{eq:a_continuous_degen_boundary_parabolic} and $b$ obeys \eqref{eq:b_perp_positive_boundary_parabolic} and \eqref{eq:b_C2_parabolic}. Suppose that $\bar P^0 = (\bar t^0,\bar x^0) \in \mydirac_0\!\sQ$. Then there are a constant $\delta>0$ and a $C^1$ diffeomorphism, $\Phi:\RR^{d+1}\to\RR^{d+1}$, of the form $(t,x)\mapsto (t,y(t,x))$ and which restricts to a $C^2$ diffeomorphism from $\RR^{d+1}\less(B_{2\delta}(\bar P^0)\cap\mydirac_0\!\sQ)$ onto its image, such that the following holds. If $u \in C^2(\RR^{d+1})$
and\,\footnote{The function $v$ may only be $C^1$ on $\RR^{d+1}$ but is $C^2$ on the complement of $\left((-2\delta,2\delta)\times B_{2\delta}^d(\bar P^0)\right)\cap\mydirac_0\!\sQ$ in $\RR^{d+1}$.}
$v := u\circ\Phi^{-1}$ and $\widetilde\sQ:=\Phi(\sQ)$, and the operator
\begin{equation}
\label{eq:Tilde_generator_parabolic}
\tilde Lv := -v_t -\tr(\tilde aD^2v) - \langle \tilde b, Dv\rangle + \tilde cv \quad\hbox{on }\widetilde\sQ,
\end{equation}
and its coefficients, $\tilde a:\underline{\widetilde\sQ}\to\sS^+(d)$ and $\tilde b:\underline{\widetilde\sQ}\to\RR^d$ and $\tilde c:\widetilde\sQ\to\RR$, are defined by $\tilde Lv := (Lu)\circ\Phi^{-1}$ on $\widetilde\sQ$, then
\begin{subequations}
\begin{align}
\label{eq:a_tilde_boundary_zero_parabolic}
\tilde a = 0 &\quad\hbox{on } \mydirac_0\!\widetilde\sQ,
\\
\label{eq:b_tilde_perp_positive_boundary_parabolic}
\tilde b^\perp > 0 &\quad\hbox{on } \mydirac_0\!\widetilde\sQ,
\\
\label{eq:b_tilde_tangential_zero_parabolic}
\tilde b^\parallel = 0 &\quad\hbox{on } \left((\bar t^0-\delta,\bar t^0+\delta)\times B_\delta^d(\bar x^0)\right)\cap\mydirac_0\!\widetilde\sQ,
\\
\label{eq:c_tilde_preserved_parabolic}
\tilde c = c\circ\Phi^{-1} &\quad\hbox{on } \widetilde\sQ.
\end{align}
\end{subequations}
\end{lem}

\begin{rmk}[Smoothness of the diffeomorphism]
\label{rmk:Simplifying_coefficients_parabolic}
If the condition \eqref{eq:b_C2_parabolic} is replaced by $b\in C^k(\underline\sQ)$, for an integer $k\geq 2$, then the diffeomorphism, $\Phi$, in Lemma \ref{lem:Simplifying_coefficients_parabolic} will be $C^k$ on the complement of $\left((-2\delta,2\delta)\times B_{2\delta}^d(\bar P^0)\right)\cap\mydirac_0\!\sQ$ in $\RR^{d+1}$.
\end{rmk}

\begin{proof}[Proof of Lemma \ref{lem:Simplifying_coefficients_parabolic}]
We just indicate the changes required in the proof of Lemma \ref{lem:Simplifying_coefficients_elliptic}. Recall from the definition \eqref{eq:Degeneracy_locus_parabolic} of $\mydirac_0\!\sQ$ that $n_0(P)=0$ for any $P \in \mydirac_0\!\sQ$, where $\vec n(P) = n_0(P)e_0+\vec n(P)\in\RR\times\RR^d$ denotes an inward-pointing normal vector at a point $P\in\partial\sQ$. By first applying a translation of $\RR^{d+1}=\RR\times\RR^d$ and then a rotation of the spatial factor, $\RR^d$, respectively, we may assume without loss of generality that $\bar P^0 = O \in \RR^{d+1}$ (origin) and $\vec n(O)=e_d$. By virtue of Lemma \ref{lem:Straightening_boundary}, we can find a $C^1$ diffeomorphism of $\RR^{d+1}=\RR\times\RR^{d-1}\times\RR$ of the form
$$
\RR\times\RR^{d-1}\times\RR\ni(t,x',x_d)\mapsto (t,z'(t,x',x_d),x_d)\in\RR^{d-1}\times\RR
$$
such that, for $\delta>0$ small enough,
$$
\left([-2\delta,2\delta]\times\bar B_{2\delta}^d(\bar P^0)\right)\cap\partial\sQ \subset \mydirac_0\!\sQ,
$$
and
$$
\left((-2\delta,2\delta)\times B_{2\delta}^d(\bar P^0)\right)\cap\mydirac\sQ
=
(-2\delta,2\delta)\times B_{2\delta}^d\cap(\RR\times\RR^{d-1}\times\bar\RR_+)
\equiv
(-2\delta,2\delta)\times B_{2\delta}^{d,+},
$$
and which restricts to a $C^\infty$ diffeomorphism from
$$
\RR^{d+1}\less \left\{\left((-2\delta,2\delta)\times B_{2\delta}^d(\bar P^0)\right)\cap\mydirac\sQ\right\}
$$
onto its image.

In addition, for small enough $\delta$, the hypothesis \eqref{eq:b_C2_parabolic} on $b$ implies (after relabeling the coefficients of $L$ due to the coordinate change) that $b^d>0$ on $[-2\delta,2\delta]\times\bar B_{2\delta}^+$ since $b^d(\bar P^0)>0$ by the hypothesis \eqref{eq:b_perp_positive_boundary_parabolic}. We relabel the coordinates on $\RR^d$ as $x=(x',x_d)\in\RR^{d-1}\times\RR$ and define
$$
\xi^i(t,x) := -\frac{b^i(t,x',0)}{b^d(t,x',0)}, \quad\hbox{for } 1\leq i\leq d-1, \quad\hbox{and}\quad \xi^d(t,x) := 0, \quad\forall\, (t,x) \in (-2\delta,2\delta)\times \underline B_{2\delta}^{d,+}.
$$
The functions $b^i$, for $1\leq i\leq d$, are $C^2$ on $(-2\delta,2\delta)\times B_{2\delta}^{d,+}$ by hypothesis \eqref{eq:b_C2_parabolic} and
may be extended from $(-2\delta,2\delta)\times B_{2\delta}^{d,+}$ to give $C^1$ functions $b^i$ on $(-2\delta,2\delta)\times B_{2\delta}$ (see \cite[Lemma 6.37]{GilbargTrudinger}) such that $b^d>0$ on $[-2\delta,2\delta]\times\bar B_{2\delta}$ and so the $C^2$ functions $\xi^i$ on $(-2\delta,2\delta)\times B_{2\delta}^{d,+}$ also extend to give $C^1$ functions $\xi^i$ on $(-2\delta,2\delta)\times B_{2\delta}$, for $1\leq i\leq d$. We define
$$
y := \begin{cases} x + \xi(t,x) x_d, & (t,x) \in (-\delta,\delta)\times B_\delta^d, \\ x, & (t,x) \in \RR^{d+1}\less \left((-2\delta,2\delta)\times B_{2\delta}^d\right), \end{cases}
$$
and interpolate between these two definitions on the cylindrical annulus,
$$
(-2\delta,2\delta)\times B_{2\delta}^d \less \left((-\delta,\delta)\times B_\delta^d\right),
$$
with a $C^\infty$ cutoff function, $\psi:\RR^{d+1}\to [0,1]$, where $\psi=1$ on $[-\delta,\delta]\times \bar B_\delta^d$ and $\psi=0$ on the complement of $(-2\delta,2\delta)\times B_{2\delta}^d$ in $\RR^{d+1}$, so
$$
y = x + \psi(t,x)\xi(t,x) x_d, \quad\hbox{for } (t,x)\in \RR^{d+1}.
$$
Note finally that $y=x$ on $\RR^{d-1}\times\{0\}$.

It remains to describe the differences in the effects of the combined coordinate changes on the coefficients of $L$ which were not already encountered in the proof of Lemma \ref{lem:Simplifying_coefficients_elliptic}. Setting $v(t,y',y_d) := u(t,x',x_d)$ and writing $y'=(y_1',\ldots,y_{d-1}')$, where $y_i=y_i(t,x',x_d)$ for $1\leq i\leq d$, we have
$$
u_t = v_t + \sum_{i=1}^{d-1}v_{y_i'}\frac{\partial y_i'}{\partial t} + v_{y_d}\frac{\partial y_d}{\partial t},
$$
and so the additional terms $\sum_{i=1}^{d-1}v_{y_i'}\partial y_i'/\partial t$ and $v_{y_d}\partial y_d/\partial t$ will simply contribute to the terms $\sum_{i=1}^{d-1}\tilde b^i v_{y_i'}$ and $\tilde b^d v_{y_d}$ in the expression \eqref{eq:Tilde_generator_parabolic} for $\tilde L$. The remainder of the proof is the same as that of Lemma \ref{lem:Simplifying_coefficients_elliptic}.
\end{proof}

\subsection{Hopf boundary point lemma for a boundary-degenerate parabolic linear second-order differential operator}
\label{subsec:Hopf_boundary_point_lemma_parabolic}
We shall also need the following boundary-degenerate parabolic analogue of our boundary-degenerate elliptic Hopf boundary point result, Lemma \ref{lem:Degenerate_hopf_lemma_elliptic}.

\begin{lem}[Hopf boundary point lemma for a boundary-degenerate parabolic linear second-order differential operator]
\label{lem:Degenerate_hopf_lemma_parabolic}
\cite[Lemma 5.8]{Feehan_parabolicmaximumprinciple}
Let $\sQ\subset\RR^{d+1}$ be an open subset and assume that the coefficients of $L$ in \eqref{eq:Generator_parabolic} obey \eqref{eq:a_locally_strictly_parabolic}, \eqref{eq:c_nonnegative_parabolic}, \eqref{eq:b_perp_positive_boundary_parabolic}\footnote{It is sufficient for the proof of Lemma \ref{lem:Degenerate_hopf_lemma_parabolic} that $\langle b,\vec n\rangle(P^0) > 0$ in the case $P^0 \in \mydirac_0\!\sQ$.}, \eqref{eq:c_locally_bounded_above_domain_plus_degen_boundary_parabolic}, and that
\begin{gather}
\label{eq:ab_locally_bounded_near_boundary_parabolic}
a \hbox{ and } b \quad\hbox{are (essentially) locally bounded on } \underline \sQ,
\\
\label{eq:ab_continuous_coefficient_boundary_parabolic}
a \hbox{ and } b \quad \hbox{are continuous on } \mydirac_0\!\sQ.
\end{gather}
Suppose that $\underline \sQ$ contains the closure $\bar B$ of an open ball,
$$
B := \left\{(t,x)\in\RR^{d+1}: |x-x^*|^2 + |t-t^*|^2 < R^2\right\} \subset \sQ,
$$
and $P^0 := (t^0,x^0)\in\partial B$ with $x^0\neq x^*$. Suppose that $u\in C^2(\sQ)$ (respectively, $u\in W^{2,d+1}_{\loc}(\sQ)$) obeys $Lu\leq 0$ (a.e.) on $\sQ$ and satisfies the conditions,
\begin{enumerate}
\renewcommand{\theenumi}{\roman{enumi}}
\item $u$ is continuous at $P^0$;
\item\label{item:xzero_strict_local_max_parabolic} $u(P^0) > u(P)$, for all $P\in B$;
\item $D_{\vec n} u(P^0)$ exists,
\end{enumerate}
where $D_{\vec n} u(P^0)$ is the derivative of $u$ at $P^0$ in the direction of the \emph{inward}-pointing unit normal vector, $\vec n(P^0)$, at $P^0\in\partial B$. Then the following hold:
\begin{enumerate}
\item\label{item:Hopf_c_zero_parabolic} If $c=0$ on $\sQ$, then $D_n u(P^0)$ obeys the strict inequality,
\begin{equation}
\label{eq:Positive_inward_normal_derivative_parabolic}
D_{\vec n} u(P^0) < 0.
\end{equation}
\item\label{item:Hopf_c_geq_zero_parabolic} If $c\geq 0$ on $\sQ$ and $u(P^0)\geq 0$, then \eqref{eq:Positive_inward_normal_derivative_parabolic} holds.
\item\label{item:Hopf_c_no_sign_parabolic} If $u(P^0)=0$, then \eqref{eq:Positive_inward_normal_derivative_parabolic} holds irrespective of the sign of $c$.
\end{enumerate}
\end{lem}

\subsection{Perturbing a degenerate-boundary local maximum to an interior local maximum}
\label{subsec:Perturb_degen_boundary_local_maximum_parabolic}
We can now give the

\begin{proof}[Proof of Theorem \ref{thm:Viscosity_maximum_parabolic}]
For ease of language, we shall suppose that $u\in C^2(\sQ)$; as with Theorem \ref{thm:Viscosity_maximum_elliptic}, there is no difference in the argument when $u\in W^{2,p}_{\loc}(\sQ)$.

Just as in the proof of Lemma \ref{lem:Simplifying_coefficients_parabolic}, by first applying a $C^1$ diffeomorphism of $\RR^{d+1}$, which is $C^2$ on the complement of an open neighborhood in $\mydirac_0\!\sQ$ of the point $\bar P^0\in\mydirac_0\!\sQ$, to straighten the boundary, we may assume without loss of generality that $\bar P^0 = O \in \RR^{d+1}$ (origin) and $\vec n(\bar P^0) = e_d$ and, for a small enough constants $\tau>0$ and $\rho_0>0$ (depending on the geometry of $\partial \sQ$ and implicitly on $\bar P^0$), that
$$
\left([-\tau,\tau]\times\bar B_{\rho_0}^d\right)\cap\partial \sQ \subset \mydirac_0\!\sQ,
$$
and, in particular,
$$
\left((-\tau,\tau)\times B_{\rho_0}^d\right)\cap\mydirac_0\!\sQ = (-\tau,\tau)\times B_{\rho_0}^{d-1}\times \{0\} \subset \RR^{d-1}\times\bar\RR_+,
$$
where $B_{\rho_0}^d$ is the open ball in $\RR^d$ with center at the origin and radius ${\rho_0}$. For a small enough positive constant $\ell$ (depending on the open subset $\sQ$ and implicitly on $\bar P^0$), we may assume that $[-\tau,\tau]\times\bar B_{\rho_0}^{d-1}\times [0,\ell]\subset\underline \sQ$ with $u$ attaining a non-negative maximum on $[-\tau,\tau]\times\bar B_{\rho_0}^{d-1}\times [0,\ell]$ at the origin, $O$, and which is strict maximum in the sense that
$$
u(P) < u(O), \quad\forall P \in (-\tau,\tau)\times B_{\rho_0}^{d-1}\times (0,\ell).
$$
We begin with the construction of the quadratic polynomial.

\setcounter{step}{0}
\begin{step}[Construction of the quadratic polynomial, $w$]
\label{step:Construction_w_parabolic}
Since $u_t(\bar P^0)$ and $Du(\bar P^0)$ exist by hypothesis, the function $u(t,x)$ has a first-order Taylor polynomial approximation\footnote{We write the error term variously as $(t+|x|)o(t,x)$ or $(t+|x'|+x_d)o(t,x',x_d)$, depending on the context.},
$$
u(t,x) = u(O) + u_t(O)t + \langle Du(O),x\rangle + \left(t+|x|\right)o(t,x), \quad\forall\, (t,x) \in [-\tau,\tau]\times\bar B_{\rho_0}^{d-1}\times [0,\ell],
$$
where $o$ is continuous on $[-\tau,\tau]\times\bar B_{\rho_0}^{d-1}\times [0,\ell]$ with $o(t,x)=0$ when $(t,x)=O$, and
$$
u_t(O) = u_{x_i}(O) = 0, \quad 1\leq i\leq d-1, \quad\hbox{and}\quad u_{x_d}(O) \leq 0,
$$
because $u$ has a local maximum at $(t,x)=O$ relative to $(-\tau,\tau)\times B_{\rho_0}^{d-1}\times [0,\ell)$. By hypothesis, this local maximum is strict and so Lemma \ref{lem:Degenerate_hopf_lemma_parabolic} implies that
$$
u_{x_d}(O) := p < 0.
$$
Consequently, setting $u(O):=r$, the Taylor formula for $u$ becomes
\begin{equation}
\label{eq:u_taylor_polynomial_parabolic}
u(t,x) = r + px_d + \left(t+|x|\right)o(t,x), \quad\forall\, x \in \bar B_{\rho_0}^{d-1}\times [0,\ell].
\end{equation}
We now define a quadratic polynomial, for constants $0<\eta,\zeta\leq 1$ and $Q>0$, to be chosen later,
\begin{equation}
\label{eq:Quadratic_polynomial_perturbation_parabolic}
w(t,x) := -\zeta t + (\eta-p)x_d - \frac{Q}{2}|x|^2, \quad\forall\, (t,x) \in \RR^{d+1},
\end{equation}
and observe that
\begin{align*}
Lw(t,x) &= \zeta + \tr a(t,x)\,Q - b^d(t,x)(\eta - p) + \langle b(t,x), x\rangle Q
\\
&\quad + c(t,x)\left((\eta-p)x_d - \frac{Q}{2}|x|^2\right),
\end{align*}
and thus,
\begin{equation}
\label{eq:Lw}
\begin{aligned}
Lw(t,x) &= \zeta + \left(\tr a(t,x) + \langle b(t,x), x\rangle\right)Q - b^d(t,x)(\eta-p)
\\
&\quad + c(t,x)\left(\eta-p)x_d - \frac{Q}{2}|x|^2\right), \quad\forall (t,x) \in (-\tau,\tau)\times B_{\rho_0}^{d-1}\times (0,\ell).
\end{aligned}
\end{equation}
From \eqref{eq:b_splitting_parabolic} we have $b(P) = b^\parallel(P) + b^\perp(P)$ for $P\in [-\tau,\tau]\times\bar B_{\rho_0}^{d-1}\times [0,\ell]$. We have $b^\perp(P) = (0,\ldots,0,b^d(P))$, since $\vec n(P)=e_d$, and $b^\parallel(P) = (b^1(P),\ldots,b^{d-1}(P),0)$. For convenience, write $x=(x',x_d) \in \RR^{d-1}\times\RR$. By hypothesis \eqref{eq:b_tangential_component_zero_degenerate_boundary_parabolic} or hypothesis \eqref{eq:b_C2_parabolic} and  Lemma \ref{lem:Simplifying_coefficients_parabolic} if \eqref{eq:b_tangential_component_zero_degenerate_boundary_parabolic} does not initially hold, we can assume
$$
b^\parallel(t,x',0) = 0, \quad\forall\,(t,x') \in (-\tau,\tau)\times B_{\rho_0}^{d-1},
$$
The definition \eqref{eq:Degeneracy_locus_parabolic} of $\mydirac_0\!\sQ$ and the fact that $a\in\ C(\underline\sQ;\sS^+(d))$ by \eqref{eq:a_continuous_degen_boundary_parabolic} implies
$$
a(t,x',0) = 0, \quad\forall\,(t,x') \in (-\tau,\tau)\times B_{\rho_0}^{d-1},
$$
Therefore, the locally Lipschitz hypothesis \eqref{eq:a_locally_lipschitz_parabolic} for $a$ and hypotheses \eqref{eq:b_tangential_locally_lipschitz_parabolic} or \eqref{eq:b_C2_parabolic} for $b$ imply that there is a positive constant, $K$, such that
\begin{align}
\label{eq:a_Lipschitz_endpoint_nbhd_parabolic}
\tr a(t,x) &\leq Kx_d,
\\
\label{eq:bprime_endpoint_nbhd_bound_parabolic}
|b^\parallel(t,x)| &\leq Kx_d, \quad\forall\, (t,x) \in (-\tau,\tau)\times B_{\rho_0}^{d-1}\times(0,\ell).
\end{align}
Also, by hypothesis \eqref{eq:b_perp_positive_boundary_parabolic}, we have
$$
b^d(O) := b_0>0,
$$
and so, for small enough $\ell$ and $\rho_0$, we see that
\begin{equation}
\label{eq:bd_continuity_endpoint_nbhd_bound_parabolic}
\frac{b_0}{2} \leq b^d(t,x) \leq 2b_0, \quad\forall\, (t,x)\in (-\tau,\tau)\times B_{\rho_0}^{d-1}\times(0,\ell).
\end{equation}
Moreover, $c$ is locally bounded above on $\underline\sQ$ by hypothesis \eqref{eq:c_locally_bounded_above_domain_plus_degen_boundary_parabolic}, so there is a positive constant, $\Lambda_0$, such that
$$
c(t,x) \leq \Lambda_0, \quad \forall\, (t,x) \in (-\tau,\tau)\times B_{\rho_0}^{d-1}\times(0,\ell).
$$
We may further assume that $\ell=\ell(b_0,\Lambda_0)$ is chosen small enough that
$$
\Lambda_0\ell \leq \frac{b_0}{4},
$$
and so we obtain
\begin{equation}
\label{eq:cx_endpoint_nbhd_bound_parabolic}
0\leq c(t,x)x_d \leq \frac{b_0}{4}, \quad \forall\, (t,x) \in (-\tau,\tau)\times B_{\rho_0}^{d-1}\times(0,\ell).
\end{equation}
Thus, writing $\langle b(t,x), x\rangle = \langle b^\parallel(t,x), x'\rangle + b^d(t,x)x_d$ in
\eqref{eq:Lw} and using $c\geq 0$ on $\sQ$ (by hypothesis \eqref{eq:c_nonnegative_parabolic}) yields
\begin{align*}
Lw(t,x) &\leq \zeta + \langle b^\parallel(t,x), x'\rangle Q + \left(\tr a(t,x) + b^d(t,x)x_d\right)Q - b^d(t,x)(\eta-p)
\\
&\quad + c(t,x)(\eta-p)x_d,
\end{align*}
and by applying \eqref{eq:a_Lipschitz_endpoint_nbhd_parabolic}, \eqref{eq:bprime_endpoint_nbhd_bound_parabolic}, \eqref{eq:bd_continuity_endpoint_nbhd_bound_parabolic}, and \eqref{eq:cx_endpoint_nbhd_bound_parabolic}, we discover that
\begin{equation}
\label{eq:Lw_upper_bound_parabolic}
\begin{aligned}
Lw(t,x) &\leq \zeta + KQ |x'|x_d + (K+2b_0)Qx_d - \frac{b_0}{4}(\eta-p),
\\
&\qquad \forall\, (t,x) \in (-\tau,\tau)\times B_{\rho_0}^{d-1}\times(0,\ell).
\end{aligned}
\end{equation}
Now let
$$
m := \frac{16}{b_0}(K+2b_0) = \frac{16K}{b_0} + 32,
$$
and observe that
\begin{equation}
\label{eq:xd_bound_parabolic}
0 \leq (K + 2b_0)Q x_d <  \frac{b_0}{16}(\eta-p) \iff 0 \leq x_d < \frac{\eta-p}{mQ}.
\end{equation}
But then
\begin{equation}
\label{eq:xprime_bound_parabolic}
KQ|x'|x_d < \frac{b_0}{16}(\eta-p) \Leftarrow |x'| < \frac{mb_0}{16K} \quad\hbox{and}\quad 0 \leq x_d < \frac{\eta-p}{mQ}.
\end{equation}
Note that $mb_0/(16K) = 1 + 2b_0/K$ by definition of $m$ and so we may suppose without loss of generality that $\rho_0=\rho_0(b_0,K)>0$ is chosen small enough that
\begin{equation}
\label{eq:rho_less_xprime_bound_parabolic}
\rho_0 \leq 1 + \frac{2b_0}{K}.
\end{equation}
By choosing $Q=Q(\ell,m,p)$ large enough, we may also assume (recall that $p<0$ depends on $u$ but is fixed and $\eta\in(0,1]$) without loss of generality that
\begin{equation}
\label{eq:xd_bound_less_ell_parabolic}
0<\hat x_d(\eta,Q) := \frac{\eta-p}{mQ} \leq \ell.
\end{equation}
In the sequel, we shall exploit the fact the upper bound in \eqref{eq:xd_bound_less_ell_parabolic} is preserved as $\eta\downarrow 0$ or $Q\uparrow\infty$. By virtue of \eqref{eq:Lw_upper_bound_parabolic}, \eqref{eq:xprime_bound_parabolic}, \eqref{eq:xd_bound_parabolic}, \eqref{eq:rho_less_xprime_bound_parabolic}, and \eqref{eq:xd_bound_less_ell_parabolic}, we obtain
\begin{align*}
Lw &< \zeta + \frac{b_0}{16}(\eta-p) + \frac{b_0}{16}(\eta-p) - \frac{b_0}{4}(\eta-p)
\\
&= \zeta - \frac{b_0}{8}(\eta-p) \quad\hbox{on } (-\tau,\tau)\times B_{\rho_0}^{d-1}\times (0,\hat x_d),
\end{align*}
and thus, provided $\zeta=\zeta(b_0)>0$ is chosen small enough that
\begin{equation}
\label{eq:zeta_bound}
\zeta \leq -\frac{b_0}{8}p,
\end{equation}
so $\zeta < b_0(\eta-p)/8$, we find that
$$
Lw < 0 \quad\hbox{on } (-\tau,\tau)\times B_{\rho_0}^{d-1}\times (0,\hat x_d),
$$
where it necessarily also holds that $Lv = Lu+Lw < 0$, with $v := u + w$, since $Lu\leq 0$ on $\sQ$ by hypothesis. This completes Step \ref{step:Construction_w_parabolic} and establishes one of the conclusions of Theorem \ref{thm:Viscosity_maximum_parabolic}, where the function $w$ appearing in that conclusion is equal to the preceding quadratic polynomial composed with the coordinate changes required in our construction.
\end{step}

We next show that $v$ attains a maximum value $v(P^0)>u(O)$ at some point $P^0 \in [0,\tau]\times \bar B_{\rho_0}^{d-1}\times[0,\ell]$.

\begin{step}[Maximum of $v$ on the parabolic cylinder]
\label{step:Interior_maximum_parabolic}
By \eqref{eq:u_taylor_polynomial_parabolic} and writing $x=(x',x_d)$, the function $v(t,x',x_d)$ has the Taylor polynomial approximation,
\begin{equation}
\label{eq:v_taylor_polynomial_parabolic}
\begin{aligned}
v(t,x',x_d) &= r - \zeta t + \left(\eta - \frac{Q}{2}x_d\right)x_d + (t+|x'|+x_d)o(t,x',x_d) - \frac{Q}{2}|x'|^2,
\\
&\qquad \forall\, (t,x',x_d)\in [0,\tau]\times \bar B_{\rho_0}^{d-1}\times[0,\ell].
\end{aligned}
\end{equation}
Consequently, writing $o(0,0,x_d) = o(x_d)$,
$$
v(0,0,x_d) = r + \left(\eta - \frac{Q}{2}x_d\right)x_d + x_do(x_d), \quad \forall\, x_d\in [0,\ell],
$$
and thus, setting $x_d = \xi \hat x_d \in (0, \hat x_d)$ for $0<\xi<1$ and $\hat x_d\in (0,\ell]$ defined in \eqref{eq:xd_bound_less_ell_parabolic}, we see that
$$
v(0,0,\xi\hat x_d) = r + \left(\eta - \frac{Q}{2}\xi\hat x_d + o(\xi\hat x_d)\right)\xi\hat x_d, \quad \forall\, \xi \in (0,1).
$$
For any $\eta, Q$ (and thus $\hat x_d(\eta,Q)\in (0,\ell]$), we may choose $\xi=\xi(\eta,Q) \in (0,1)$ such that
$$
\frac{Q}{2}\xi\hat x_d + |o(\xi\hat x_d)| \leq \frac{\eta}{2},
$$
and so
$$
v(0,0,\xi\hat x_d) \geq r + \frac{\eta}{2}\xi\hat x_d > r.
$$
Therefore, because $v(0,0,\xi\hat x_d)>r$ for small enough $\xi=\xi(\eta,Q)\in (0, 1)$, the function $v$ must attain a maximum $v(t^0,x^0)>r$ at some point $P^0 = (t^0,x^0)\in [0,\tau]\times \bar B_{\rho_0}^{d-1}\times[0,\hat x_d]$.
\end{step}

It remains to show that $v$ on $[0,\tau]\times \bar B_{\rho_0}^{d-1}\times[0,\ell]$ attains its maximum, $v(P^0)>u(O)$, at a point $P^0$ in the parabolic interior, $[0,\tau)\times V$, of the parabolic cylinder $[0,\tau]\times \bar V$ with spatial cylinder $V$ having a curved side,
$$
V := \{(x',x_d)\in B_{\rho_0}^{d-1}\times (0,\hat x_d): |x'|< \rho(x_d)\}
\subset
B_{\rho_0}^{d-1}\times (0,\ell),
$$
defined by a suitably-chosen positive-valued function,
\begin{equation}
\label{eq:Defn_rho_parabolic}
\rho:[0,\infty)\to(0,\infty),
\end{equation}
to be determined later in the proof of Theorem \ref{thm:Viscosity_maximum_parabolic}.

\begin{step}[Upper bound for $v$ on the parabolic boundary]
\label{step:Upper_bound_parabolic_boundary}
To show that $(t^0,x^0)$ must belong to the \emph{parabolic interior}, that is $[0,\tau)\times V$, we prove that $v\leq r$ on
$$
\mydirac((0,\tau)\times V) = \left(\{\tau\}\times \bar V\right) \cup \left((0,\tau)\times \partial V\right),
$$
for suitable $\zeta$, $\eta$, and $Q$.

\setcounter{case}{0}
\begin{case}[Upper bound for $v$ on the temporal boundary]
\label{case:Cylinder_temporalboundary_parabolic}
We begin by examining $v$ on the temporal boundary, $\{\tau\}\times \bar V$, and prove the stronger result, that $v\leq r$ on
$$
\{\tau\}\times \bar B_{\hat\rho}^{d-1}\times[0,\hat x_d] \supset \{\tau\}\times \bar V.
$$
Observe that \eqref{eq:v_taylor_polynomial_parabolic} yields
\begin{align*}
v(\tau,x) &= r - \zeta \tau + \eta x_d + (\tau+|x|)o(\tau,x) - \frac{Q}{2}|x|^2,
\\
&\leq r - \zeta \tau + \eta x_d + (\tau+|x|)o(\tau,x), \quad \forall\, x \in \bar B_{\rho_0}^{d-1}\times[0,\hat x_d].
\end{align*}
Recall from \eqref{eq:xd_bound_less_ell_parabolic} that $\hat x_d(\eta,Q) = (\eta-p)/(mQ)$. Choose $\tau=\tau(u,\zeta)>0$ and $\rho_0=\rho_0(u,\zeta)$ small enough and $\hat x_d>0$ small enough (via large enough $Q=Q(m,p,u,\tau,\zeta)$ in \eqref{eq:xd_bound_less_ell_parabolic}) to ensure that
\begin{equation}
\label{eq:o_tau_x_bound}
o(\tau,x) \leq \frac{\zeta}{4}, \quad \forall\, x \in \bar B_{\rho_0}^{d-1}\times[0,\hat x_d].
\end{equation}
We now decrease $\hat x_d>$ (by increasing $Q=Q(m,p,u,\tau,\zeta)$ further) and decrease $\rho_0=\rho_0(u,\tau)>0$ enough, respectively, to ensure that
$$
|x'| \leq \frac{\tau}{2} \quad\hbox{and}\quad x_d \leq \frac{\zeta}{4}\tau \wedge \frac{\tau}{2}, \quad \forall\, x \in \bar B_{\rho_0}^{d-1}\times[0,\hat x_d],
$$
and thus, using $|x| \leq |x'| + x_d$ and recalling that $\eta \in (0,1]$,
$$
|x| \leq \tau  \quad\hbox{and}\quad \eta x_d \leq \frac{\zeta}{4}\tau, \quad \forall\, x \in \bar B_{\rho_0}^{d-1}\times[0,\hat x_d].
$$
Therefore, we obtain
\begin{align*}
v(\tau,x) &\leq r - \zeta \tau + \frac{\zeta}{4}\tau + \frac{\zeta}{4}\tau + \frac{\zeta}{4}\tau
\\
&= r - \frac{\zeta}{4}\tau
\\
&< r, \quad\forall\, x \in \bar B_{\rho_0}^{d-1}\times[0,\hat x_d],
\end{align*}
as desired for Case \ref{case:Cylinder_temporalboundary_parabolic}.
\end{case}

Next we examine $v$ on $(0,\tau)\times \partial V$ and focus on the three spatial boundary portions of $V$,
$$
\partial V = \left(\bar V\cap\{x_d=0\}\right) \cup \left(\bar V\cap\{x_d=\hat x_d\}\right) \cup \left(\bar V\cap\{0<x_d<\hat x_d\}\right),
$$
consisting of the bottom, top, and curved side of the spatial cylinder, $V$, respectively.

\begin{case}[Upper bound for $v$ on the bottom of the spatial cylinder]
\label{case:Cylinder_spatialbottom_parabolic}
First, consider the contribution from the \emph{bottom} of the spatial cylinder,
$$
\bar V\cap\{x_d=0\} \subset B_{\rho_0}^{d-1}\times\{0\}.
$$
We have $u(t,x',0)\leq r$ for all $(t,x') \in (0,\tau)\times\bar B_{\rho_0}^{d-1}$ (since $u(t,x)$ has maximum value $u(O)=r$ on $[0,\tau]\times \bar B_{\rho_0}^{d-1}\times[0,\ell]$) and thus
$$
u(t,x',0) \leq r \quad\hbox{for all } t \in (0,\tau) \hbox{ and } |x'|\leq \rho_0,
$$
and so the definition \eqref{eq:Quadratic_polynomial_perturbation_parabolic} of $w$ and the fact that $v=u+w$ yields
$$
v(t,x',0) = r - \zeta t - \frac{Q}{2}|x'|^2 < r \quad\hbox{for } 0<t<\tau \hbox{ and } |x'|\leq \rho_0,
$$
as desired for Case \ref{case:Cylinder_spatialbottom_parabolic}.
\end{case}

\begin{case}[Upper bound for $v$ on the top of the spatial cylinder]
\label{case:Cylinder_spatialtop_parabolic}
Second, consider the contribution from the \emph{top} of the cylinder,
$$
\bar V\cap\{x_d=\hat x_d\} \subset B_{\rho_0}^{d-1}\times\{\hat x_d\}.
$$
Gathering like terms in the Taylor formula \eqref{eq:v_taylor_polynomial_parabolic}, setting $x_d = \hat x_d(\eta,Q) \equiv (\eta-p)/(mQ)$ via \eqref{eq:xd_bound_less_ell_parabolic}, and simplifying yields
\begin{align*}
v(t,x',\hat x_d) &= r - (\zeta - o(t,x',\hat x_d))t + \left(\eta + o(t,x',\hat x_d) - \frac{(\eta-p)}{2m}\right)\hat x_d  + |x'|o(t,x',\hat x_d) - \frac{Q}{2}|x'|^2,
\\
&\qquad \forall\, (t,x') \in [0,\tau]\times \bar B_{\rho_0}^{d-1}.
\end{align*}
Now choose $\eta = \eta(m,p) \in (0,1]$ small enough that (recalling that $p<0$ by assumption)
\begin{equation}
\label{eq:eta_bound_parabolic}
\eta < -\frac{p}{8m},
\end{equation}
so (recalling that $m>1$)
$$
\frac{(2m-1)}{2m}\eta < -\frac{p}{8m}.
$$
Now make $\hat x_d\in (0,\ell]$ small enough, by choosing $Q=Q(m,p,u,\tau,\zeta)$ large enough and applying the definition \eqref{eq:xd_bound_less_ell_parabolic} of $\hat x_d$, and choose $\tau=\tau(m,p,u,\zeta)>0$ and $\rho_0(m,p,u,\zeta)>0$ small enough that
\begin{equation}
\label{eq:o_t_xprime_hatxd_bound}
o(t,x',x_d) < -\frac{p}{8m}\wedge \frac{\zeta}{2}, \quad 0\leq t\leq \tau, \quad |x'| \leq \rho_0, \quad 0\leq x_d\leq \hat x_d,
\end{equation}
which we can do since $o(t,x',x_d)$ is continuous on $[0,\tau]\times\bar B_{\rho_0}^{d-1}\times[0,\ell]$. These choices of $\eta$ and $Q$ preserve the upper bound in \eqref{eq:xd_bound_less_ell_parabolic}. By \eqref{eq:eta_bound_parabolic} and \eqref{eq:o_t_xprime_hatxd_bound} and discarding the term $-Q|x'|^2/2$, we discover that
$$
v(t,x',\hat x_d) \leq r - \frac{\zeta}{2}t + \frac{p}{4m}\hat x_d  -\frac{p}{8m}|x'|,
\quad \forall\, (t,x') \in [0,\tau]\times \bar B_{\rho_0}^{d-1}.
$$
We now restrict our function $\rho$ in \eqref{eq:Defn_rho_parabolic} by requiring that
\begin{equation}
\label{eq:rho_value_spatial_cylinder_top_parabolic}
\rho(\hat x_d) = \hat x_d,
\end{equation}
so now $\rho$ depends on $\eta\in(0,1]$ and $Q>0$. Consequently,
$$
v(t,x',\hat x_d) \leq r - \frac{\zeta}{2}t + \frac{p}{8m}\hat x_d < r, \quad \forall\, (t,x') \in (0,\tau)\times \bar B_{\hat x_d}^{d-1},
$$
as desired for Case \ref{case:Cylinder_spatialtop_parabolic}.
\end{case}

\begin{case}[Upper bound for $v$ on the curved side of the spatial cylinder]
\label{case:Cylinder_spatialside_parabolic}
It remains to consider the contribution from the curved side of the spatial cylinder in
$$
(0,\tau)\times\{(x',x_d)\in B_{\rho_0}^{d-1}\times (0,\hat x_d): |x'|< \rho(x_d)\}
\subset
(0,\tau)\times\bar V,
$$
for a suitably-chosen positive-valued function $\rho$ in \eqref{eq:Defn_rho_parabolic}.

We proceed by modifying the argument for Case \ref{case:Cylinder_spatialtop_parabolic}. Again gathering like terms in the Taylor formula \eqref{eq:v_taylor_polynomial_parabolic} yields
\begin{align*}
v(t,x',x_d) &= r - (\zeta - o(t,x',x_d))t + \left(\eta + o(t,x',x_d) - \frac{Q}{2}x_d\right)x_d
\\
&\quad + |x'|o(t,x',x_d) - \frac{Q}{2}|x'|^2,  \quad \forall\, (t,x') \in [0,\tau]\times \bar B_{\rho_0}^{d-1}.
\end{align*}
Applying \eqref{eq:eta_bound_parabolic} (which had required a small enough $\eta>0$) and \eqref{eq:o_t_xprime_hatxd_bound} (which had required a large enough $Q$) yields
\begin{align*}
v(t,x',x_d) &\leq r - \frac{\zeta}{2}t + \left(-\frac{p}{8m} - \frac{Q}{2}x_d\right)x_d  - \frac{p}{8m}|x'| - \frac{Q}{2}|x'|^2,
\\
&\qquad \forall\, (t,x',x_d) \in [0,\tau]\times \bar B_{\rho_0}^{d-1}\times (0,\hat x_d).
\end{align*}
Therefore, when
$$
-\frac{p}{4mQ} < x_d < \hat x_d,
$$
the $x_d$-coefficient in the first-order Taylor formula for $v(t,x',x_d)$ becomes non-positive; recall that $\hat x_d = (\eta-p)/(mQ)$ (and $p<0$) by \eqref{eq:xd_bound_less_ell_parabolic}, so $\hat x_d$ obeys
$$
-\frac{p}{4mQ} < \hat x_d = \frac{\eta}{mQ} - \frac{p}{mQ}.
$$
Now we restrict $x'$ by requiring that $|x'|=\rho(x_d)>0$, for $x_d \in (0,\hat x_d)$, and choose $\rho(x_d)$ such that
$$
\left(-\frac{p}{8m} - \frac{Q}{2}x_d\right)x_d  - \frac{p}{8m}\rho(x_d) - \frac{Q}{2}\rho^2(x_d) \leq 0, \quad 0<  x_d < \hat x_d,
$$
that is,
$$
\frac{Q}{2}\rho^2(x_d) + \frac{p}{8m}\rho(x_d) + \left(\frac{p}{8m} + \frac{Q}{2}x_d\right)x_d \geq 0, \quad 0<  x_d < \hat x_d,
$$
and so $\rho(x_d)$ should obey
\begin{equation}
\label{eq:rho_value_spatial_cylinder_side_parabolic}
\rho(x_d)
\geq
\begin{cases}
0 &\hbox{if } x_d = 0,
\\
\displaystyle
\frac{1}{Q}\left(-\frac{p}{8m} - \sqrt{ \frac{p^2}{64m^2} - 2x_dQ\left(\frac{p}{8m} + \frac{Q}{2}x_d\right)}\right),
&\hbox{if }0 < x_d < \displaystyle-\frac{p}{4mQ},
\\
0 &\hbox{if }\displaystyle-\frac{p}{4mQ} \leq x_d < \hat x_d.
\end{cases}
\end{equation}
Note that the function on the right-hand side is bounded by $-p/(8mQ)$ and thus, for large enough $Q=Q(m,p,\rho_0)$, we may assume that $8mQ\leq\rho_0$. We now fix a continuous function, $\rho$ in \eqref{eq:Defn_rho_parabolic}, obeying \eqref{eq:rho_value_spatial_cylinder_top_parabolic} and \eqref{eq:rho_value_spatial_cylinder_side_parabolic} together with $\rho(0)=\rho_0>0$ and $\rho(x_d)\leq\rho_0$ for $0\leq x_d\leq\hat x_d$, so the previous inequality for $v(t,x',x_d)$ becomes
$$
v(t,x',x_d) \leq r - \frac{\zeta}{2}t < r, \quad\hbox{for } 0 < t < \tau, \quad |x'| = \rho(x_d), \quad 0<  x_d < \hat x_d,
$$
as desired for Case \ref{case:Cylinder_spatialside_parabolic}.
\end{case}
Combining the conclusions of Cases \ref{case:Cylinder_temporalboundary_parabolic} through \ref{case:Cylinder_spatialside_parabolic} implies that $v\leq r$ on the parabolic boundary, $\mydirac((0,\tau)\times V)$, as we had sought.
\end{step}

This completes the proof of Theorem \ref{thm:Viscosity_maximum_parabolic}, with open subset, $\sV\subset\RR_+\times\RR^{d-1}\times\RR_+$, given by the preimage of $(0,\tau)\times V$ with respect to the initial coordinate transformation.
\end{proof}

\begin{rmk}[Motivation for the argument in Step \ref{step:Interior_maximum_parabolic}]
Since $D_{\vec n}v(O)=\eta>0$ and $D_{\vec n}v(t,x',0)$ is continuous with respect to $(t,x') \in (-\tau,\tau)\times B_{2\rho}^{d-1}$, because this is true of $D_{\vec n}u$ by hypothesis, then we may assume without loss of generality (for small enough $\tau$ and $\rho$) that
$$
v_{x_d}(t,x',0) > 0, \quad \forall\, (t,x') \in [0,\tau]\times \bar B_{\rho_0}^{d-1},
$$
and so the function $v$ cannot have a local maximum, relative to $(0,\tau)\times B_{\rho_0}^{d-1}\times(0,\ell)$, on the boundary portion $[0,\tau]\times \bar B_{\rho_0}^{d-1}\times\{0\}$.

Since $v_t(O)=-\zeta<0$ and $v_t$ is continuous with respect to $(t,x) \in (-\tau,\tau)\times B_{2\rho}^{d-1}\times[0,\ell]$, because this is true of $u_t$ by hypothesis, then we may assume without loss of generality (for small enough $\tau$, $\rho$, and $\ell$) that
$$
v_t(\tau,x)  < 0, \quad \forall\, x \in \bar B_{\rho_0}^{d-1}\times[0,\ell],
$$
and so the function $v$ cannot have a local maximum, relative to $(0,\tau)\times B_{\rho_0}^{d-1}\times(0,\ell)$, on the boundary portion $\{\rho\}\times \bar B_{\rho_0}^{d-1}\times[0,\ell]$.
\end{rmk}

\subsection{Strong maximum principles for $L$-subharmonic functions in $C^2(\sQ)$}
\label{subsec:Strong_maximum_principle_parabolic_C2}
Before we can state and prove the strong maximum principle for the parabolic operator, we adapt the following \emph{notational conventions} of \cite[p. 34]{FriedmanPDE}.

\begin{defn}[Connected subsets of $\sQ\subset\RR^{d+1}$]
\label{defn:Connected_subsets_parabolic}
For any point $P^0=(t^0,x^0) \in \underline \sQ$, we denote by $S(P^0)$ the set of all points $P \in \underline \sQ$ which can be connected to $P^0$ by a simple continuous curve in $\underline \sQ$ along which the time coordinate is non-\emph{increasing}\footnote{In \cite[p. 34]{FriedmanPDE}, the time coordinate is required to be non-\emph{decreasing}, consistent with Friedman's convention of considering an initial value problem rather than the convention of considering a terminal value problem in this article and \cite{Bensoussan_Lions}.}
from $P$ to $P^0$. By $C(P^0)$ we denote the connected component of $\underline \sQ\cap\{t=t^0\}$ which contains $P^0$.
\end{defn}

Clearly, $C(P^0) \subset S(P^0)$. Since $C(P^0)$ is a connected component of $\underline \sQ\cap\{t=t^0\}$, it is necessarily a closed subset of $\underline \sQ\cap\{t=t^0\}$ and, if the number of connected components is finite, then it is also an open subset.

\begin{exmp}[$S(P^0)$ and $C(P^0)$ when $\sQ$ is a parabolic cylinder]
\label{exmp:Connected_subsets_parabolic_cylinder}
Suppose, as in Example \ref{exmp:Boundary_cylinder_parabolic}, that $\sQ = (0,T)\times\sO = \sO_T$ for some spatial domain $\sO\subseteqq\RR^d$ and $T>0$. For any $P^0 = (t^0,x^0) \in \underline \sQ$, Definition \ref{defn:Connected_subsets_parabolic} yields
$$
S(P^0) = (0,t_0]\times\underline\sO \quad\hbox{and}\quad C(P^0) = \{t^0\}\times\underline\sO,
$$
and so
$$
S(P^0) = \left((0,t_0)\times\underline\sO\right) \cup \left(\{t^0\}\times\underline\sO\right) = \underline\sO_{t^0} \cup C(P^0).
$$
This concludes our example.
\end{exmp}

We are now ready to prove the strong maximum principles for $L$-subharmonic functions in $C^2(\sQ)$. We begin with the following analogue of \cite[Theorem 2.1]{FriedmanPDE}, \cite[Theorem 2.7]{Lieberman}.

\begin{thm}[Strong maximum principle for $L$-subharmonic functions in $C^2(\sQ)$ when $c\geq 0$]
\label{thm:Strong_max_principle_C2_parabolic_c_geq_zero}
Let $\sQ \subset \RR^{d+1}$ be an open subset. Assume the hypotheses of Theorem \ref{thm:Viscosity_maximum_parabolic} for the coefficients of $L$ \eqref{eq:Generator_parabolic}, including \eqref{eq:c_nonnegative_parabolic}, that is, $c\geq 0$ on $\sQ$. Suppose that $u \in C^2(\sQ)\cap \sC^1(\underline\sQ)$ and $Lu \leq 0$ on $\sQ$. If $u$ attains a positive global maximum at a point $P^0 \in \underline \sQ$, then $u = u(P^0)$ on $S(P^0)$.
\end{thm}

\begin{proof}
The proof is identical to that of Theorem \ref{thm:Strong_maximum_principle_C2_elliptic}, except that the role of the Hopf boundary-point Lemma \ref{lem:Degenerate_hopf_lemma_elliptic} for the elliptic operator, $A$, is replaced by the Hopf boundary-point Lemma \ref{lem:Degenerate_hopf_lemma_parabolic} for the parabolic operator, $L$.
\end{proof}

We can also relax the requirement that $c\geq 0$ on $\sQ$ and give the following analogue of \cite[Theorem 2.3]{FriedmanPDE}.

\begin{thm}[Strong maximum principle for $L$-subharmonic functions in $C^2(\sQ)$ when $c$ has arbitrary sign]
\label{thm:Strong_max_principle_C2_parabolic_c_arb_sign}
Let $\sQ \subset \RR^{d+1}$ be an open subset. Assume the hypotheses of Theorem \ref{thm:Viscosity_maximum_parabolic} for the coefficients of $L$ \eqref{eq:Generator_parabolic}, though \eqref{eq:c_nonnegative_parabolic} may be omitted. Require also that
\begin{equation}
\label{eq:Continuous_c_coefficient_boundary}
c \quad \hbox{is continuous on } \mydirac_0\!\sQ.
\end{equation}
Suppose that $u \in C^2(\sQ)\cap \sC^1(\underline\sQ)$ and $Lu \leq 0$ on $\sQ$. If $u\leq 0$ on $\underline\sQ$ and $u(P^0)=0$ for some $P^0 \in \underline \sQ$, then $u = 0$ on $C(P^0)$.
\end{thm}

\begin{proof}
The proof is identical to that of \cite[Theorem 5.15]{Feehan_parabolicmaximumprinciple} except that the role of \cite[Theorem 5.13]{Feehan_parabolicmaximumprinciple} is replaced by that of Theorem \ref{thm:Strong_max_principle_C2_parabolic_c_geq_zero}.
\end{proof}

The following refinement of Theorem \ref{thm:Strong_max_principle_C2_parabolic_c_geq_zero}, analogous to \cite[Theorem 2.4]{FriedmanPDE}, makes a stronger assertion since the hypotheses only assume that $u(P^0)$ is the maximum of $u$ on $S(P^0)\subset \underline \sQ$ rather than $\underline \sQ$.

\begin{thm}[Refined strong maximum principle when $c\geq 0$]
\label{thm:Strong_max_principle_C2_parabolic_c_geq_zero_refined}
Let $\sQ\subset\RR^{d+1}$ be an open subset. Assume the hypotheses of Theorem \ref{thm:Viscosity_maximum_parabolic} for the coefficients of $L$ \eqref{eq:Generator_parabolic}, though \eqref{eq:c_nonnegative_parabolic} may be omitted. If $u \in C^2(\sQ)\cap \sC^1(\underline\sQ)$ obeys $Lu\leq 0$ on $S(P^0)$, and $c\geq 0$ on $S(P^0)$, and $u$ has a positive global maximum which is attained at the point $P^0$, then $u = u(P^0)$ on $S(P^0)$.
\end{thm}

\begin{proof}
The proof is the same as that of Theorem \ref{thm:Strong_max_principle_C2_parabolic_c_geq_zero}, since its proof only made use of the fact that $u(P^0)$ is the maximum of $u$ on $S(P^0)\subset \underline \sQ$ (and not necessarily the maximum on $\underline \sQ$).
\end{proof}

We have the following analogue of \cite[Theorem 2.5]{FriedmanPDE}; note the stronger conclusion ($u = 0$ on $S(P^0)$) relative to Theorem \ref{thm:Strong_max_principle_C2_parabolic_c_arb_sign} ($u = 0$ on $C(P^0)$, where we recall that $C(P^0)\subset S(P^0)$).

\begin{thm}[Refined strong maximum principle when $c$ has arbitrary sign]
\label{thm:Strong_max_principle_C2_parabolic_c_arb_sign_refined}
Let $\sQ\subset\RR^{d+1}$ be an open subset. Assume the hypotheses of Theorem \ref{thm:Viscosity_maximum_parabolic} for the coefficients of $L$ \eqref{eq:Generator_parabolic}, though \eqref{eq:c_nonnegative_parabolic} may be omitted. If $u \in C^2(\sQ)\cap \sC^1(\underline\sQ)$ obeys $Lu\leq 0$ on $\sQ$, and $u\leq 0$ on $S(P^0)$, and $u(P^0)=0$, then $u = 0$ on $S(P^0)$.
\end{thm}

\begin{proof}
The proof is identical to that of \cite[Theorem 2.5]{FriedmanPDE}, except that the roles of \cite[Theorem 2.3]{FriedmanPDE} and its method of proof and the proof of \cite[Theorem 2.1]{FriedmanPDE} are replaced by those of Theorems \ref{thm:Strong_max_principle_C2_parabolic_c_arb_sign} and \ref{thm:Strong_max_principle_C2_parabolic_c_geq_zero}, respectively.
\end{proof}

A version of Remark \ref{rmk:Aronszajn_elliptic} also applies for the parabolic operator.

\begin{rmk}[Relaxing the hypothesis of a global maximum to a local maximum]
\label{rmk:Aronszajn_parabolic}
In Theorems \ref{thm:Strong_max_principle_C2_parabolic_c_geq_zero} and \ref{thm:Strong_max_principle_C2_parabolic_c_geq_zero_refined} (when $c=0$ on $\sQ$) or Theorems \ref{thm:Strong_max_principle_C2_parabolic_c_arb_sign} and \ref{thm:Strong_max_principle_C2_parabolic_c_arb_sign_refined} (when $c$ has arbitrary sign on $\sQ$), if $u$ only attains a \emph{local} maximum at a point $P^0\in\underline\sQ$, so $u\leq u(P^0)$ on an open ball $S(P^0)\cap B_\eps(P^0)$ for some $\eps>0$, we may still conclude that $u$ is constant on $S(P^0)\cap B_\eps(P^0)$ and therefore, provided $L$ has the unique continuation property (see, for example, \cite{Chen_1998, Koch_Tataru_2009}), we find again that $u$ is constant on $S(P^0)$ (or $C(P^0)$ in the setting of Theorem \ref{thm:Strong_max_principle_C2_parabolic_c_arb_sign}).
\end{rmk}

\subsection{Strong maximum principles for $L$-subharmonic functions in $W^{2,d+1}_{\loc}(\sQ)$}
\label{subsec:Strong_maximum_principle_parabolic_W2d+1}
Statements of a classical strong maximum principle for a $L$-subharmonic functions in $W^{2,d+1}_{\loc}(\sQ)$ (that is, no open subset $\Sigma \subset\mydirac\!\sQ$ of the parabolic boundary is treated as part of the `interior' of $\sQ$, as we do in our present article) are more difficult to find in the literature, so we shall instead recast versions of the strong maximum principle for weak subsolutions \cite[Theorem 6.25]{Lieberman} of linear second-order parabolic equations or viscosity subsolutions \cite[Theorem 2.1 and Corollary 2.4]{DaLio_2004} (see also \cite[Corollary 8]{Gripenberg_2007}) of fully nonlinear second-order parabolic equations so they can be more readily applied in this article. Given $1\leq p\leq\infty$, we let $W^{1,p}(\sQ)$ denote the parabolic Sobolev space of measurable functions, $u$ on $\sQ$, such that $u$ and its weak derivatives with respect to the spatial coordinates, $u_{x_i}$ for $1\leq i\leq d$, belong to $L^p(\sQ)$ \cite[p. 155]{Lieberman}.

\begin{thm}[Classical strong maximum principle for $L$-subharmonic functions in $W^{2,d+1}_{\loc}(\sQ)$]
\label{thm:Classical_strong_max_principle_W2d+1_parabolic_c_geq_zero}
\cite[Corollary 2.4]{DaLio_2004}, \cite[Theorem 6.25]{Lieberman}
%COMMENT: Theorem 7.1 in Lieberman (ABP estimate) implies the weak maximum principle (Corollary 7.4 in Lieberman) for strong solutions
Let $\sQ\subset \RR^{d+1}$ be an open subset. Assume that the coefficients of $L$ in \eqref{eq:Generator_parabolic} obey one of the following two conditions,
\begin{gather}
\label{eq:a_Lipschitz_b_c_essentially_bounded_parabolic}
a^{ij} \in W^{1,\infty}(\sQ), \quad b^i, \ c \in L^\infty(\sQ), \quad 1\leq i,j \leq d, \quad\hbox{or }
\\
\label{eq:a_b_c_continuous_parabolic}
\tag{\ref*{eq:a_Lipschitz_b_c_essentially_bounded_parabolic}$'$}
a^{ij}, b^i, c \in C(\bar\sQ),
\end{gather}
and $c\geq 0$ (a.e.) on $\sQ$, and there is a positive constant, $\lambda_0$, such that
\begin{equation}
\label{eq:Strictly_parabolic}
\langle a\xi,\xi\rangle \geq \lambda_0 \quad\hbox{(a.e.) on }\sQ, \quad\forall\,\xi \in \RR^d.
\end{equation}
Suppose that $u\in W^{2,d+1}_{\loc}(\sQ)$ and $Lu\leq 0$ a.e. on $\sQ$. If $u$ attains a nonnegative global maximum at a point $P^0\in\sQ$, then $u$ is constant on $S(P^0)$.
\end{thm}

Theorem \ref{thm:Classical_strong_max_principle_W2d+1_parabolic_c_geq_zero} holds, naturally, in more generality than stated here (even in the classical case).

\begin{rmk}[Classical strong maximum principle for weakly $L$-subharmonic functions]
\label{rmk:Classical_strong_max_principle_W2d+1_parabolic_c_geq_zero}
Under hypothesis \eqref{eq:a_Lipschitz_b_c_essentially_bounded_parabolic} in Theorem \ref{thm:Classical_strong_max_principle_W2d+1_parabolic_c_geq_zero}, the conclusion follows from \cite[Theorem 6.25]{Lieberman}. Indeed, it suffices to take $u \in W^{1,2}_b(\sQ)$, that is, $u \in L^2(\sQ)$ and $Du \in L^2(\sQ)$ and $u(t,\cdot) \in L^2(\sQ(t))$ for all $t \in I(\sQ)$ with $\sup_{t\in I(\sQ)}\|u(t,\cdot)\| < \infty$; here $\sQ(t)=\{x\in\RR^d: (t,x)\in\sQ\}$ and $I(\sQ) = \{t\in\RR: \sQ(t)\neq\emptyset\}$ \cite[p. 5 and p. 101]{Lieberman}. The additional condition, $a^{ij} \in W^{1,\infty}(\sQ)$, is imposed so we may integrate by parts and consider $u$ to be a weak subsolution. Since $u \in C(\sQ)$ when $u\in W^{2,d+1}_{\loc}(\sQ)$ by Theorem \ref{thm:Parabolic_sobolev_embedding}, the hypothesis in \cite[Theorem 6.25]{Lieberman} that there is a cylinder $Q_R(P^0) \Subset \sQ$ (as in \eqref{eq:Parabolic_cylinder}) such that $\sup_{Q_R(P^0)} u = \sup_\sQ u$ can be replaced by the simpler assertion that there is a point $P^0 \in \sQ$ such that $u(P^0) = \sup_\sQ u$.
\end{rmk}

\begin{rmk}[Classical strong maximum principle for viscosity subsolutions]
\label{rmk:Classical_strong_max_principle_viscosity_parabolic_c_geq_zero}
Under hypothesis \eqref{eq:a_b_c_continuous_parabolic} in Theorem \ref{thm:Classical_strong_max_principle_W2d+1_parabolic_c_geq_zero}, the conclusion follows from \cite[Corollary 2.4]{DaLio_2004} when $\sQ=(0,T)\times\sO$ for some $T>0$ and open subset $\sO\subset\RR^d$; the extension to the case of arbitrary open subsets $\sQ\subset\RR^{d+1}$ is standard (see, for example, \cite{Gripenberg_2007}). Indeed, since $L$ in Theorem \ref{thm:Classical_strong_max_principle_W2d+1_parabolic_c_geq_zero} is also assumed to be linear and uniformly parabolic (in the sense of \cite[p. 204]{Lieberman}), then by \cite[p. 398, top of page, and p. 402]{DaLio_2004}, the conditions in \cite[Equations (A0)--(A4)]{DaLio_2004} are obeyed. We now apply the parabolic version of the elliptic Bony-Lions maximum principle \cite[Theorem 1]{Bony_1967}, \cite[Corollary 2]{Lions_1983} due to Yazhe \cite[Theorem 4.2]{Yazhe_1985} and the parabolic analogue of Lions' proof that a $W^{2,d}_{\loc}(\sO)$ subsolution is a viscosity subsolution \cite[Corollary 3]{Lions_1983} to conclude that $u \in W^{2,d+1}_{\loc}(\sQ)$ obeying $Lu\leq 0$ a.e. on $\sQ$ is necessarily a viscosity subsolution \cite[Definition 2.2]{Crandall_Ishii_Lions_1992}, so \cite[Corollary 2.4]{DaLio_2004} applies to $u$.
\end{rmk}

We have the following parabolic analogue of Theorem \ref{thm:Strong_max_principle_C2_parabolic_c_geq_zero} for $L$-subharmonic functions in $W^{2,d+1}_{\loc}(\sQ)$.

\begin{thm}[Strong maximum principle for $L$-subharmonic functions in $W^{2,d+1}_{\loc}(\sQ)$]
\label{thm:Strong_max_principle_W2d+1_parabolic_c_geq_zero}
Let $\sQ \subset \RR^{d+1}$ be an open subset. Assume the hypotheses of Theorem \ref{thm:Viscosity_maximum_parabolic} for the coefficients of $L$ \eqref{eq:Generator_parabolic}, which we require to be measurable, including \eqref{eq:c_nonnegative_parabolic}, that is, $c\geq 0$ a.e. on $\sQ$. In addition, require that the coefficients of $L$ obey one of the following two conditions,
\begin{gather}
\label{eq:a_locally_lipschitz_b_c_essentially_bounded_parabolic}
a^{ij} \in W^{1,\infty}_{\loc}(\sQ), \quad b^i, \ c \in L^\infty_{\loc}(\sQ), \quad 1\leq i,j \leq d, \quad\hbox{or }
\\
\label{eq:a_b_c_locally_continuous_parabolic}
\tag{\ref*{eq:a_locally_lipschitz_b_c_essentially_bounded_parabolic}$'$}
a^{ij}, b^i, c \in C(\sQ),
\end{gather}
Suppose that $u\in W^{2,d+1}_{\loc}(\sQ)\cap \sC^1(\underline\sQ)$
and\,\footnote{Our hypothesis that $u \in \sC^1(\underline\sQ)$ could be relaxed to the slightly more technical hypothesis that $u_t$ and $Du$ be defined and continuous along $\mydirac_0\!\sQ$.}
$Lu\leq 0$ a.e. on $\sQ$. If $c=0$ (respectively, $c\geq 0$) a.e. on $\sQ$ and $u$ attains a global maximum (respectively, non-negative global maximum) at a point $P^0 \in \underline\sQ$, then $u$ is constant on $S(P^0)$.
\end{thm}

\begin{proof}
The proof is identical to that of Theorem \ref{thm:Strong_max_principle_C2_parabolic_c_geq_zero}, except that the role of classical strong maximum principle for $L$-subharmonic functions in $C^2(\sQ)$, \cite[Theorem 3.5]{GilbargTrudinger}, is replaced by the classical strong maximum principle for $L$-subharmonic functions in $W^{2,d+1}_{\loc}(\sQ)$, Theorem \ref{thm:Classical_strong_max_principle_W2d+1_parabolic_c_geq_zero}. The additional case $c=0$ follows from the obvious fact that $u$ may be replaced by $u-u(P^0)$.
\end{proof}

Remark \ref{rmk:Aronszajn_parabolic} (relaxing the hypothesis of a global maximum to a local maximum) applies equally well here.

\subsection{Weak maximum principle for $L$-subharmonic functions in $C^2(\sQ)$ or $W^{2,d+1}_{\loc}(\sQ)$}
\label{subsec:Weak_maximum_principle_parabolic}
Recall that we reserve the term \emph{domain} for a connected, open subset of $\RR^{d+1}$. We have the following parabolic analogue of Theorem \ref{thm:Weak_maximum_principle_C2_elliptic_domain} in the elliptic case and \cite[Theorem 2.4]{Lieberman} in the classical parabolic case; note that the hypotheses are simpler because $\mydirac_1\!\sQ$ is always non-empty as noted in \eqref{eq:Nondegenerate_boundaryportion_nonempty_parabolic}

\begin{thm}[Weak maximum principle on domains for $L$-subharmonic functions in $C^2(\sQ)$]
\label{thm:Weak_maximum_principle_C2_parabolic_domain}
Let $\sQ\subset\RR^{d+1}$ be a bounded domain. Assume the hypotheses of Theorem \ref{thm:Viscosity_maximum_parabolic} for the coefficients of $L$ in \eqref{eq:Generator_parabolic}. Suppose $u \in C^2(\sQ)\cap \sC^1(\underline\sQ)$ and $\sup_\sQ u<\infty$. If $Lu\leq 0$ on $\sQ$ and $u^* \leq 0$ on $\mydirac_1\!\sQ$, then $u\leq 0$ on $\sQ$.
\end{thm}

%COMMENT In the parabolic case, it does not seem possible for $\partial\sQ\less\mydirac_0\!\sQ$ to be isolated set of points.
%COMMENT In the elliptic case, it is possible for $\partial\sO\less\partial_0\sO$ to be an isolated set of points, in which case a specification of boundary conditions at those points would not be necessary.

\begin{proof}
The proof is the same as that of Theorem \ref{thm:Weak_maximum_principle_C2_elliptic_domain}, except the role of Theorem \ref{thm:Strong_maximum_principle_C2_elliptic} is replaced by Theorem \ref{thm:Strong_max_principle_C2_parabolic_c_geq_zero} and because $\mydirac_1\!\sQ$ is necessarily non-empty for parabolic domains, the parabolic analogue of \eqref{eq:Nonempty_nondegenerate_boundary_elliptic_domain} is always obeyed and we do not require an analogue of the condition \eqref{eq:Empty_nondegenerate_boundary_elliptic_domain}, that $c>0$ at some point of $\sQ$.
\end{proof}

By replacing the role of Theorem \ref{thm:Strong_max_principle_C2_parabolic_c_geq_zero} by that of Theorem \ref{thm:Strong_max_principle_W2d+1_parabolic_c_geq_zero} in the proof of Theorem \ref{thm:Weak_maximum_principle_C2_parabolic_domain}, we obtain the following analogue of the weak maximum principle for $L$-subharmonic functions \cite[Corollary 7.4]{Lieberman}.

\begin{thm}[Weak maximum principle on domains for $L$-subharmonic functions in $W^{2,d+1}_{\loc}(\sQ)$]
\label{thm:Weak_maximum_principle_W2d+1_parabolic_domain}
Let $\sQ\subset\RR^{d+1}$ be a bounded domain. Assume the hypotheses of Theorem \ref{thm:Viscosity_maximum_parabolic} for the coefficients of $L$ in \eqref{eq:Generator_parabolic}, which we require to be measurable.
%COMMENT Although the solution is continuous, c is only defined a.e. on \sQ
Suppose $u\in W^{2,d+1}_{\loc}(\sQ)\cap \sC^1(\underline\sQ)$ and\,\footnotemark[\value{footnote}]
$\sup_\sQ u<\infty$. If $Lu\leq 0$ a.e. on $\sQ$ and $u^* \leq 0$ on $\mydirac_1\!\sQ$, then $u\leq 0$ on $\sQ$.
\end{thm}

One can again immediately deduce versions of the weak maximum principle (Theorems \ref{thm:Weak_maximum_principle_C2_parabolic_domain} or \ref{thm:Weak_maximum_principle_W2d+1_parabolic_domain}) when $\sQ$ is not necessarily connected, by the same argument as used in the proofs of Corollaries \ref{cor:Weak_maximum_principle_C2_elliptic_opensubset} or \ref{cor:Weak_maximum_principle_W2d_elliptic_opensubset} in the elliptic case. Because \eqref{eq:Nondegenerate_boundaryportion_nonempty_parabolic} holds for each connected component of $\sQ$ (necessarily open since $\sQ\subset\RR^{d+1}$ is open and thus locally connected), the proofs are simpler.

\begin{cor}[Weak maximum principle on open subsets for $L$-subharmonic functions in $C^2(\sQ)$]
\label{cor:Weak_maximum_principle_C2_parabolic_opensubset}
Let $\sQ\subset\RR^{d+1}$ be a bounded, open subset. Assume the hypotheses of Theorem \ref{thm:Viscosity_maximum_parabolic} for the coefficients of $L$ in \eqref{eq:Generator_parabolic}. Suppose $u \in C^2(\sQ)\cap \sC^1(\underline\sQ)$ and $\sup_\sQ u<\infty$. If $Lu\leq 0$ on $\sQ$ and $u^* \leq 0$ on $\mydirac_1\!\sQ$, then $u\leq 0$ on $\sQ$.
\end{cor}

\begin{cor}[Weak maximum principle on open subsets for $L$-subharmonic functions in $W^{2,d+1}_{\loc}(\sQ)$]
\label{thm:Weak_maximum_principle_W2d+1_parabolic_opensubset}
Let $\sQ\subset\RR^{d+1}$ be a bounded, open subset. Assume the hypotheses of Theorem \ref{thm:Viscosity_maximum_parabolic} for the coefficients of $L$ in \eqref{eq:Generator_parabolic}, which we require to be measurable. Suppose $u\in W^{2,d+1}_{\loc}(\sQ)\cap \sC^1(\underline\sQ)$ and $\sup_\sQ u<\infty$. If $Lu\leq 0$ a.e. on $\sQ$ and $u^* \leq 0$ on $\mydirac_1\!\sQ$, then $u\leq 0$ on $\sQ$.
\end{cor}

\subsection{Weak maximum principles for $L$-subharmonic functions in $C^2(\sQ)$ or $W^{2,d+1}_{\loc}(\sQ)$ and relaxed hypotheses on the coefficients}
\label{subsec:Relaxing_coefficients_parabolic}
Before we can relax the hypotheses of Theorem \ref{thm:Weak_maximum_principle_C2_parabolic_domain} on the coefficients $a$ and $b$ of $L$ in \eqref{eq:Generator_parabolic}, we shall again need a priori maximum principle estimates. To state these properties in useful generality, we recall the parabolic analogue of Definition \ref{defn:Weak_max_principle_property}.

\begin{defn}[Weak maximum principle property for $L$-subharmonic functions in $C^2(\sQ)$ or $W^{2,d+1}_{\loc}(\sQ)$]
\label{defn:Weak_max_principle_property_parabolic}
\cite[Definition 2.2]{Feehan_parabolicmaximumprinciple}
%COMMENT In parabolic problems, we always have that Sigma \neq \mydirac\!\sQ$
Let $\sQ\subset\RR^{d+1}$ be an open subset, let $\Sigma \subsetneqq \mydirac\!\sQ$ be an open subset, and let $\fK\subset C^2(\sQ)$ (respectively, $W^{2,d+1}_{\loc}(\sQ)$) be a convex cone. We say that an operator $L$ in \eqref{eq:Generator_parabolic} obeys the \emph{weak maximum principle property on $\sQ\cup\Sigma$ for $\fK$} if whenever $u\in \fK$ obeys
$$
Lu \leq 0 \quad \hbox{(a.e.) on } \sQ \quad\hbox{and}\quad u^* \leq 0 \quad\hbox{on } \Sigma,
$$
then
$$
u \leq 0 \quad\hbox{on } \sQ.
$$
\end{defn}

\begin{exmp}[Examples of the weak maximum principle property for $L$-subharmonic functions in $C^2(\sQ)$ or $W^{2,d+1}_{\loc}(\sQ)$]
\label{exmp:Examples_weak maximum principle property_C2_W2d+1_parabolic}
One can find examples (compare \cite[Example 2.3]{Feehan_parabolicmaximumprinciple}) of subsets $\sQ$ and $\Sigma\subseteqq\partial\sQ$, operators $L$, and cones $\fK$ yielding the weak maximum principle property in the following settings.
\begin{enumerate}
\item In \cite[Theorem 2.4]{Lieberman} (respectively, \cite[Corollaries 6.26 or 7.4]{Lieberman}), where $\sQ$ is bounded, one takes $\Sigma = \emptyset$ and $\fK = C^2(\sQ)\cap C(\bar\sQ)$ (respectively, $W^{2,d+1}_{\loc}(\sQ)\cap C(\bar\sQ)$).
\item In \cite[Theorem 8.1.4]{Krylov_LecturesHolder}, where $\sQ$ may be unbounded, one takes $\Sigma = \emptyset$ and $\fK$ to be the set of $u\in C^2(\sQ)$ such that $\sup_\sQ u < \infty$.
\item In Theorem \ref{thm:Weak_maximum_principle_C2_parabolic_domain} (respectively, Theorem \ref{thm:Weak_maximum_principle_W2d+1_parabolic_domain}), where $\sQ$ is a bounded domain, one takes $\Sigma = \mydirac_0\!\sQ$ and $\fK = C^2(\sQ)\cap \sC^1(\underline\sQ)$ (respectively, $W^{2,d+1}_{\loc}(\sQ)\cap \sC^1(\underline\sQ)$) and $\sup_\sQ u < \infty$.
\item In \cite[Theorem 4.1]{Feehan_parabolicmaximumprinciple}, where $\sQ$ is bounded, one takes $\Sigma = \mydirac_0\!\sQ$ and $\fK$ to be the set of $u\in C^2(\sQ)\cap \sC^1(\underline\sQ)$ such that $\lim_{\sQ\ni P\to P^0}\tr(aD^2u)(P)=0$ for all $P^0\in\mydirac_0\!\sQ$ and $\sup_\sQ u < \infty$.
\item In \cite[Theorem 4.3]{Feehan_parabolicmaximumprinciple}, where $\sQ$ may be unbounded, one takes $\Sigma = \mydirac_0\!\sQ$ and $\fK$ to be the set of $u\in C^2(\sQ)\cap \sC^1(\underline\sQ)$ such that $\lim_{\sQ\ni P\to P^0}\tr(aD^2u)(P)=0$ for all $P^0\in\mydirac_0\!\sQ$ and $\sup_\sQ u < \infty$.
\end{enumerate}
\end{exmp}

\begin{rmk}[Weak maximum principle property for viscosity subsolutions]
\label{rmk:Example_weak maximum principle property_viscosity_parabolic}
Suppose that the coefficients of $L$ in \eqref{eq:Generator_parabolic} obey the hypotheses of \cite[Theorem 8.2 and Example 3.6]{Crandall_Ishii_Lions_1992} (see Remark \ref{rmk:Example_weak maximum principle property_viscosity_elliptic}). Then \cite[Theorem 8.2]{Crandall_Ishii_Lions_1992}, when $\sQ$ is bounded, implies that $L$ has the weak maximum principal property when $\Sigma = \emptyset$ and $\fK$ is the set of upper semicontinuous functions on $\bar\sQ$.
\end{rmk}

We have the following analogue of Proposition \ref{prop:Elliptic_weak_max_principle_apriori_estimates}.

%COMMENT Note omission of boundedness requirements
\begin{prop}[Weak maximum principle estimates for $L$-subharmonic functions in $C^2(\sQ)$ or $W^{2,d+1}_{\loc}(\sQ)$]
\label{prop:Parabolic_weak_max_principle_apriori_estimates}
\cite[Proposition 2.6]{Feehan_parabolicmaximumprinciple}
Let $\sQ\subset\RR^{d+1}$ be an open subset and $L$ in \eqref{eq:Generator_parabolic} have the weak maximum principle property on $\sQ\cup\Sigma$ in the sense of Definition \ref{defn:Weak_max_principle_property_parabolic}, for a convex cone $\fK\subset C^2(\sQ)$ (respectively, $W^{2,d+1}_{\loc}(\sQ)$) containing the constant function $1$ and open subset $\Sigma\subsetneqq\mydirac\!\sQ$. Suppose that $u\in \fK$.
\begin{enumerate}
\item\label{item:Subsolution_Lu_leq_zero} If $c \geq 0$ (a.e.) on $\sQ$ and $Lu\leq 0$ (a.e.) on $\sQ$, then
$$
u\leq 0 \vee \sup_{\mydirac\!\sQ\less\Sigma}u^* \quad\hbox{on } \sQ.
$$
\item\label{item:Subsolution_Lu_arb_sign} If $c\geq c_0$ (a.e.) on $\sQ$ for a positive constant $c_0$, then
$$
u\leq 0 \vee \frac{1}{c_0}\esssup_\sQ Lu \vee \sup_{\mydirac\!\sQ\less\Sigma}u^* \quad\hbox{on } \sQ.
$$
\item\label{item:Subsolution_Lu_arb_sign_c_bounded_below} If $c\geq -K_0$ (a.e.) on $\sQ$ for a positive constant $K_0$ and $\sQ\subset(-\infty,T)\times\RR^d$, then
$$
u\leq 0 \vee e^{(K_0+1)(T-t)}\esssup_\sQ Lu \vee \sup_{\mydirac\!\sQ\less\Sigma}u^* \quad\hbox{on } \sQ.
$$
\end{enumerate}
\end{prop}

Item \eqref{item:Subsolution_Lu_arb_sign_c_bounded_below} in Proposition \ref{prop:Parabolic_weak_max_principle_apriori_estimates} follows from Item \eqref{item:Subsolution_Lu_arb_sign} via the change of dependent variable, $v = e^{\lambda t}u$, for $\lambda = K_0+1$. Naturally, $\esssup_\sQ Lu$ may be replaced by $\sup_\sQ Lu$ in Proposition \ref{prop:Parabolic_weak_max_principle_apriori_estimates} when $u\in C^2(\sQ)$.

We provide an application of Proposition \ref{prop:Parabolic_weak_max_principle_apriori_estimates} to the question of uniqueness for solutions to an obstacle problem, via the following analogue of \cite[Theorem 1.3.4]{Friedman_1982}; the proof is identical to that of Theorem \ref{thm:Comparison_principle_elliptic_obstacle_problem}, except that we now appeal to the weak maximum principle property for parabolic operators (Definition \ref{defn:Weak_max_principle_property_parabolic}).

\begin{thm}[Comparison principle and uniqueness for $W^{2,d+1}_{\loc}$ solutions to the obstacle problem]
\label{thm:Comparison_principle_parabolic_obstacle_problem}
\cite[Proposition 3.2]{Feehan_parabolicmaximumprinciple}
Let $\sQ\subset\RR^{d+1}$ be an open subset,
%COMMENT We do not need to assume that \sQ is bounded
let $\Sigma \subsetneqq \mydirac\!\sQ$ be an open subset, and let $\fK\subset W^{2,d+1}_{\loc}(\sQ)$ be a convex cone. For every open subset $\sU\subset\sQ$, let $L$ in \eqref{eq:Generator_parabolic} have the weak maximum principle property on $\sU\cup\Sigma$ in the sense of Definition \ref{defn:Weak_max_principle_property_parabolic}
\footnote{Note that the weak maximum principle property hypothesis on $L$ here is stronger than that in Proposition \ref{prop:Parabolic_weak_max_principle_apriori_estimates}.}.
Let $f\in L^{d+1}_{\loc}(\sQ)$ and $\psi\in L^{d+1}_{\loc}(\sQ)$.
%COMMENT We do not need to assume that psi is bounded above or continuous
Suppose $u\in \fK$ (respectively, $v\in -\fK$) is a solution (respectively, supersolution) to the obstacle problem,
%COMMENT We do not need to assume that u is bounded above
$$
\min\{Lu-f, \ u-\psi\} = 0 \ (\geq  0) \quad\hbox{a.e. on } \sQ.
$$
If $v_*\geq u^*$ on $\mydirac\!\sQ\less\Sigma$, then $v \geq u$ on $\sQ$; if $u, v$ are solutions and $v_* = u^*$ on $\mydirac\!\sQ\less\Sigma$, then $u = v$ on $\sQ$.
\end{thm}

The proofs of Theorems \ref{thm:Weak_maximum_principle_C2_elliptic_domain_relaxed} and \ref{thm:Weak_maximum_principle_W2d_elliptic_domain_relaxed} adapt virtually without change to give

\begin{thm}[Weak maximum principle on domains for $L$-subharmonic functions in $C^2(\sQ)$ and relaxed conditions on $a, b$]
\label{thm:Weak_maximum_principle_C2_parabolic_domain_relaxed}
Let $\sQ\subset\RR^{d+1}$ be a bounded domain and $L$ be as in \eqref{eq:Generator_parabolic}. Assume that the coefficients of $L$ obey \eqref{eq:b_perp_nonnegative_boundary_parabolic}, \eqref{eq:c_lower*_positive_parabolic}, \eqref{eq:c_locally_bounded_above_domain_plus_degen_boundary_parabolic}, \eqref{eq:a_locally_lipschitz_parabolic}, and
\begin{equation}
\label{eq:b_continuous_on_domain_plus_degenerate_boundary_parabolic}
b \in C(\underline\sQ;\RR^d),
\end{equation}
and \emph{one} of the following two conditions:
\begin{gather}
\label{eq:a_lipschitz_on_domain_plus_degenerate_boundary_parabolic}
a \in \sC^{0,1}(\underline\sQ) \quad\hbox{(locally Lipschitz on $\underline\sQ$), or}
\\
\label{eq:a_locally_strictly_parabolic_on_domain_parabolic}
\tag{\ref*{eq:a_lipschitz_on_domain_plus_degenerate_boundary_parabolic}$'$}
a \quad\hbox{obeys \eqref{eq:a_locally_strictly_parabolic} (locally strictly parabolic on $\sQ$)}.
\end{gather}
Suppose $u \in C^2(\sQ)\cap \sC^1(\underline\sQ)$ and $\sup_\sQ u < \infty$. If $Lu\leq 0$ on $\sQ$ and $u^* \leq 0$ on $\mydirac_1\!\sQ$, then $u\leq 0$ on $\sQ$.
\end{thm}

\begin{proof}
The proof is the same as that of Theorem \ref{thm:Weak_maximum_principle_C2_elliptic_domain_relaxed}, except that the roles of Theorem \ref{thm:Weak_maximum_principle_C2_elliptic_domain} and Proposition \ref{prop:Elliptic_weak_max_principle_apriori_estimates} are replaced by those of Theorem \ref{thm:Weak_maximum_principle_C2_parabolic_domain} and Proposition \ref{prop:Parabolic_weak_max_principle_apriori_estimates}.
\end{proof}

\begin{thm}[Weak maximum principle on domains for $L$-subharmonic functions in $W^{2,d+1}_{\loc}(\sQ)$ and relaxed conditions on $a,b$]
\label{thm:Weak_maximum_principle_W2d+1_parabolic_domain_relaxed}
Assume the hypotheses of Theorem \ref{thm:Weak_maximum_principle_C2_parabolic_domain_relaxed} on $\sQ$ and $L$, except that the coefficients of $L$ are now required to be measurable. Suppose $u \in W^{2,d+1}_{\loc}(\sQ)\cap \sC^1(\underline\sQ)$ and $\sup_\sQ u<\infty$. If $Lu\leq 0$ a.e. on $\sQ$ and $u^* \leq 0$ on $\mydirac_1\!\sQ$, then $u\leq 0$ on $\sQ$.
\end{thm}

\begin{proof}
The proof is the same as that of Theorem \ref{thm:Weak_maximum_principle_C2_parabolic_domain_relaxed}, except that the role of Theorem \ref{thm:Weak_maximum_principle_C2_parabolic_domain} is replaced by that of Theorem \ref{thm:Weak_maximum_principle_W2d+1_parabolic_domain}.
\end{proof}

As before, one can deduce versions of the weak maximum principle (Theorems \ref{thm:Weak_maximum_principle_C2_parabolic_domain_relaxed} or \ref{thm:Weak_maximum_principle_W2d+1_parabolic_domain_relaxed}) when $\sQ$ is not necessarily connected, by the same argument as used in the proofs of Corollaries \ref{cor:Weak_maximum_principle_C2_elliptic_opensubset} or \ref{cor:Weak_maximum_principle_W2d_elliptic_opensubset} in the elliptic case. Because \eqref{eq:Nondegenerate_boundaryportion_nonempty_parabolic} holds for each connected component of $\sQ$, the proofs are simpler.

\begin{cor}[Weak maximum principle on open subsets for $L$-subharmonic functions in $C^2(\sQ)$ and relaxed conditions on $a, b$]
\label{cor:Weak_maximum_principle_C2_parabolic_opensubset_relaxed}
Let $\sQ\subset\RR^{d+1}$ be a bounded, open subset. Assume the hypotheses of Theorem \ref{thm:Weak_maximum_principle_C2_parabolic_domain_relaxed} on the coefficients of $L$. Suppose $u \in C^2(\sQ)\cap \sC^1(\underline\sQ)$ and $\sup_\sQ u < \infty$. If $Lu\leq 0$ on $\sQ$ and $u^* \leq 0$ on $\mydirac_1\!\sQ$, then $u\leq 0$ on $\sQ$.
\end{cor}

\begin{cor}[Weak maximum principle on open subsets for $L$-subharmonic functions in $W^{2,d+1}_{\loc}(\sQ)$ and relaxed conditions on $a,b$]
\label{cor:Weak_maximum_principle_W2d+1_parabolic_opensubset_relaxed}
Let $\sQ\subset\RR^{d+1}$ be a bounded, open subset. Assume the hypotheses of Theorem \ref{thm:Weak_maximum_principle_W2d+1_parabolic_domain_relaxed} on the coefficients of $L$. Suppose $u \in W^{2,d+1}_{\loc}(\sQ)\cap \sC^1(\underline\sQ)$ and $\sup_\sQ u<\infty$. If $Lu\leq 0$ a.e. on $\sQ$ and $u^* \leq 0$ on $\mydirac_1\!\sQ$, then $u\leq 0$ on $\sQ$.
\end{cor}

\appendix

\section{Properties of confluent hypergeometric functions}
\label{sec:Properties_confluent_hypergeometric_functions}
Recall that $u \in C^2_s[0,\infty)$ if $u\in C^1[0,\infty)$ and $xu'' \in C[0,\infty)$. As we noted in \S \ref{subsec:Motivation}, there is a sharp distinction between the boundary properties of functions in $C^1[0,\infty)$ and $C^2_s[0,\infty)$. For example, if we choose a constant $p>0$ and set\footnote{This example was suggested by Camelia Pop.}
$$
u(x) = \int_0^x t\sin(t^{-p})\,dt \quad\hbox{for }x >0 \quad\hbox{and}\quad u(0) = 0,
$$
then $u'(x) = x\sin(x^{-p})$ has limit $\lim_{x\downarrow 0}u'(x) = 0$ and so $u' \in C[0,\infty)$. However,
\begin{align*}
  u''(x) &= \sin(x^{-p}) - (1/p)x^{-p}\cos(x^{-p}),
  \\
  xu''(x) &= x\sin(x^{-p}) - (1/p)x^{1-p}\cos(x^{-p}),
\end{align*}
and thus
$$
xu''(x) \sim x^{1-p}\cos(x^{-p}) \quad\hbox{when } x\downarrow 0,
$$
with no limit as $x\downarrow 0$ when $p\geq 1$; in particular, $xu'' \notin C[0,\infty)$ if $p\geq 1$.

We review the properties of the confluent hypergeometric functions used in \S \ref{subsec:Motivation}; see \cite[\S 13]{Olver_Lozier_Boisvert_Clark} for additional details.

\begin{lem}[Properties of confluent hypergeometric functions]
\label{lem:Confluent_hypergeometric_function_properties}
\cite{Olver_Lozier_Boisvert_Clark}
Let $a, b \in \RR$ and let $M(a,b,x)$ and $U(a,b,x)$, for $x\in\RR$, denote the confluent hypergeometric functions \cite[Equations (13.2.2) and (13.2.6)]{Olver_Lozier_Boisvert_Clark}. When $b\neq -n$, for $n=0,1,2,\ldots$, the derivative of $M(a,b,x)$ with respect to $x$ obeys the recurrence relation:
\begin{equation}
\label{eq:Derivative_M}
M'(a,b,x) = \frac{a}{b} M(a+1,b+1,x) \quad\hbox{\cite[\S 13.3.15]{Olver_Lozier_Boisvert_Clark}.}
\end{equation}
The function $M(a,b,x)$ has the following asymptotic behavior as $x\rightarrow 0$:
\begin{align}
\label{eq:Value_at_0_M}
 M(a,b,0)=1\quad\hbox{\cite[\S 13.2.2]{Olver_Lozier_Boisvert_Clark}}
 \quad\hbox{and}\quad
 M(a,b,x)=1+O(x) \quad\hbox{\cite[\S 13.2.13]{Olver_Lozier_Boisvert_Clark}.}
\end{align}
The derivative of $U(a,b,x)$ with respect to $x$ obeys the recurrence relation:
\begin{equation}
\label{eq:Derivative_U}
U'(a,b,x) = -a U(a+1,b+1,x) \quad\hbox{\cite[\S 13.3.22]{Olver_Lozier_Boisvert_Clark}.}
\end{equation}
If $a = -n$ or $a = -n+b-1$, where $n=0,1,2,\ldots$, then $U(a,b,x)$ has the following asymptotic behavior as $x\rightarrow 0$:
\begin{subequations}
\label{eq:Value_at_0_U_a_or_a-b_nonpositive_integer}
\begin{align}
\label{eq:Value_at_0_U_a_nonpositive_integer}
U(-n,b,x) &= (-1)^n(b)_n + O(x) \quad\hbox{\cite[\S 13.2.14]{Olver_Lozier_Boisvert_Clark},}
\\
\label{eq:Value_at_0_U_a-b_negative_integer}
U(-n+b-1,b,x) &= (-1)^n(2-b)_n x^{1-b} + O(x^{2-b}) \quad\hbox{\cite[\S 13.2.15]{Olver_Lozier_Boisvert_Clark},}
\end{align}
\end{subequations}
where $(b)_n := b(b+1)\cdots(b+n)$ and $(b)_0 := 1$. If $a\neq -n$ and $a\neq -n+b-1$, where $n=0,1,2,\ldots$, then $U(a,b,x)$ has the following asymptotic behavior as $x\rightarrow 0$:
\begin{subequations}
\label{eq:Value_at_0_U_other}
\begin{align}
U(a,b,x) &= \frac{\Gamma(b-1)}{\Gamma(a)} x^{1-b} + O(x^{b-2}) \quad\hbox{if $b>2$,} \quad\hbox{\cite[\S 13.2.16]{Olver_Lozier_Boisvert_Clark},}
\\
U(a,2,x) &= \frac{1}{\Gamma(a)} x^{-1} + O(\ln x) \quad\hbox{if $b=2$,} \quad\hbox{\cite[\S 13.2.17]{Olver_Lozier_Boisvert_Clark},}
\\
U(a,b,x) &= \frac{\Gamma(b-1)}{\Gamma(a)} x^{1-b} + \frac{\Gamma(1-b)}{\Gamma(a-b+1)} + O(x^{2-b}) \quad\hbox{if $1<b<2$,} \quad\hbox{\cite[\S 13.2.18]{Olver_Lozier_Boisvert_Clark},}
\\
U(a,1,x) &= -\frac{1}{\Gamma(a)} (\ln x + \psi(a)+2\gamma) + O(x\ln x) \quad\hbox{if $b=1$,} \quad\hbox{\cite[\S 13.2.19]{Olver_Lozier_Boisvert_Clark},}
\\
U(a,b,x) &= \frac{\Gamma(1-b)}{\Gamma(a-b+1)} + O(x^{1-b}) \quad\hbox{if $0<b<1$,} \quad\hbox{\cite[\S 13.2.20]{Olver_Lozier_Boisvert_Clark},}
\\
U(a,0,x) &= \frac{1}{\Gamma(a+1)} + O(x\ln x) \quad\hbox{if $b=0$,} \quad\hbox{\cite[\S 13.2.21]{Olver_Lozier_Boisvert_Clark},}
\end{align}
\end{subequations}
where $\Gamma(x)$ is the Gamma function, $\psi(a):=\Gamma'(a)/\Gamma(a)$, and $\gamma\in\RR$ is Euler's constant \cite[\S 13.1]{Olver_Lozier_Boisvert_Clark}.
\end{lem}

Lemma \ref{lem:Confluent_hypergeometric_function_properties} yields the following regularity properties for $U$ and $M$ at $x=0$. We restrict our attention to the cases $a\geq 0$ and $b\geq 0$, but note that more generally one may allow $a,b\in \CC$ and consider $U(a,b,z)$ and $M(a,b,z)$ as functions of $z\in\CC$ \cite[\S 13]{Olver_Lozier_Boisvert_Clark}.

\begin{cor}[Regularity of confluent hypergeometric functions at $x=0$]
\label{cor:Confluent_hypergeometric_function_boundary_regularity}
Let $a\geq 0$ and $b\geq 0$.
\begin{enumerate}
  \item\label{item:M_smooth} If $b\neq 0$, then $M(a,b,\cdot) \in C^\infty(\RR)$.
  \item\label{item:U_smooth_a_zero} If $a=0$, then $U(0,b,\cdot) \in C^\infty(\RR)$ for any $b\in\RR$.
  \item\label{item:U_smooth_a-b_negative_integer} If $a=b-1-n$ for an integer $n$ obeying $0\leq n\leq b-1$, then $U(b-1-n,b,\cdot) \in C[0,\infty)$ when $0\leq b\leq 1$ and $U(b-1-n,b,\cdot) \in C^1[0,\infty)$ only when $b=0$.
  \item\label{item:U_smooth_a-b_other} If $a\neq b-1-n$ for any integer $n$ obeying $0\leq n\leq b-1$, then $U(a,b,\cdot) \in C[0,\infty)$ when $0< b< 1$ and $U(a,b,\cdot) \notin C^1[0,\infty)$ for any $b\geq 0$.
\end{enumerate}
\end{cor}

\begin{proof}
Item \eqref{item:M_smooth} follows from \eqref{eq:Derivative_M} and \eqref{eq:Value_at_0_M}. Item \eqref{item:U_smooth_a_zero} follows from \eqref{eq:Derivative_U} and \eqref{eq:Value_at_0_U_a_nonpositive_integer}, or direct calculation using \eqref{eq:Kummer_homogeneous_differential_equation}. Item \eqref{item:U_smooth_a-b_negative_integer} follows from \eqref{eq:Derivative_U} and \eqref{eq:Value_at_0_U_a-b_negative_integer}. Finally, Item \eqref{item:U_smooth_a-b_other} follows from \eqref{eq:Derivative_U} and \eqref{eq:Value_at_0_U_other}.
\end{proof}

\section{Straightening the boundary}
\label{sec:Straightening_boundary}
In order to keep our regularity assumption on the boundary portion, $\partial_0\sO$, to a minimum (in particular, no more than $C^{1,\alpha}$), we shall need an enhancement of the usual `straightening the boundary' lemma \cite[\S 4.6]{Adams_1975}, \cite[Appendix C.1]{Evans}, \cite[p. 94]{GilbargTrudinger}. Let $\sO\subset\RR^d$ be an open subset, $d\geq 1$, with topological boundary $\partial\sO$ and $k\geq 1$ be an integer and $\alpha\in(0,1)$. One says that $\sO$ has a boundary portion, $T\subset\partial\sO$, of class $C^{k,\alpha}$ if at each point $x^0\in T$ there exist an open ball, $B\equiv B_r(x^0)\subset\RR^d$, such that $B\cap\partial\sO\subset T$ and a $C^{k,\alpha}$ diffeomorphism, $\Psi$, from $B$ onto an open subset $D\subset\RR^d$ such that \cite[p. 94]{GilbargTrudinger}
\begin{equation}
\label{eq:Straightening_boundary}
\Psi(B\cap\sO) \subset \RR^{d-1}\times\RR_+ \quad\hbox{and}\quad \Psi(B\cap\partial\sO) \subset \RR^{d-1}\times\partial\RR_+.
\end{equation}
One says that $\Psi$ `straightens the boundary' near $x^0$. Clearly, we can assume without loss of generality (by shrinking the radius of $B$ if needed) that $\Psi$ is a $C^{k,\alpha}$ diffeomorphism from $\bar B$ onto $\bar D$ and thus, by extension, a $C^{k,\alpha}$ diffeomorphism of $\RR^d$. The enhancement we shall need is

\begin{lem}[Straightening a $C^{1,\alpha}$ boundary with a $C^k$ diffeomorphism]
\label{lem:Straightening_boundary}
Let $\sO\subset\RR^d$ be an open subset with open boundary portion, $T\subset\partial\sO$, of class $C^{1,\alpha}$ and $x^0\in T$ and $k\geq 1$ be an integer and $\alpha\in(0,1)$. Then there exist an open ball, $B\equiv B_r(x^0)\subset\RR^d$, such that $B\cap\partial\sO\subset T$ and a $C^1$ diffeomorphism, $\Psi$, from $B$ onto an open subset $D\subset\RR^d$ such that \eqref{eq:Straightening_boundary} holds and $\Psi$ is a $C^k$ diffeomorphism from $B\less\partial\sO$ onto $D\less(\RR^{d-1}\times\{0\})$. In particular, $\Psi$ may be chosen to be a $C^\infty$ diffeomorphism from $B\less\partial\sO$ onto its image.
\end{lem}

\begin{proof}
Since the boundary portion, $T$, is of class $C^{1,\alpha}$, there is a $C^{1,\alpha}$ diffeomorphism, $\Phi$, from $\bar B$ to $\bar D$ which straightens the boundary near $x^0$ in the sense of \eqref{eq:Straightening_boundary} with $\Phi$ in place of $\Psi$. We may choose a sequence of $C^{k,\alpha}$ diffeomorphisms, $\{\Psi_n\}_{n\geq 1}$, of $\RR^d$ and which approximate $\Phi$ near $x^0$ in the sense that, for all $n\geq 1$,
\begin{align*}
\|\Phi-\Psi_n\|_{C^1(\Phi^{-1}(\bar D\cap\{|x_d|\leq 1/n\}))} &\leq \frac{1}{n},
\\
\|\Psi_n\|_{C^{1,\alpha}(\bar B)} &\leq K_0,
\\
\|\Psi_n\|_{C^{k,\alpha}(\Phi^{-1}(\bar D\cap\{|x_d|\geq 1/m\}))} &\leq K_m, \quad m=1,2,3,\ldots,
\end{align*}
for positive constants, $K_m$, which are independent of $n\in\NN$. The sequence $\{\Psi_n\}_{n\geq 1}$ converges in $C^1$ on $\bar B\cap\partial\sO$ to $\Phi$ on $\bar B\cap\partial\sO$. By the Arzel\`a-Ascoli Theorem, there is a subsequence, relabeled as $\{\Psi_n\}_{n\geq 1}$, and a $C^1$ map, $\Psi:B\to\RR^d$, which is $C^k$ on $B\less\partial\sO$, such that $\{\Psi_n\}_{n\geq 1}$ converges in $C^1$ to $\Psi$ on $\bar B$ and in $C^k$ to $\Psi$ on each compact subset, $\Phi^{-1}(\bar D\cap\{|x_d|\geq 1/m\})$ for $m\geq 1$. Since $\Phi:B\to\RR^d$ is a $C^1$ diffeomorphism onto its image and $\{\Psi_n\}_{n\geq 1}$ converges to $\Phi$ in $C^1$ on $\bar B\cap\partial\sO$, the Inverse Function Theorem implies that $\Psi$ is a $C^1$ diffeomorphism from a possibly smaller ball, $B_{r'}(x^0)\subseteqq B$, onto its image. By relabeling, we may assume that $B_{r'}(x^0)=B$. Because $\Psi$ is a $C^k$ map on $B\less\partial\sO$, the Inverse Function Theorem also implies that $\Psi$ is a $C^k$ diffeomorphism from $B\less\partial\sO$ onto its image.
\end{proof}

\section{Boundary properties of second-order derivatives}
\label{sec:Boundary_properties_C2s_functions}
A special case of the following result (when $\sO=\RR^{d-1}\times\RR_+$ and $\vartheta(s)=s$) is implied by the proofs of \cite[Proposition I.12.1]{DaskalHamilton1998} and \cite[Lemma 3.1]{Feehan_Pop_mimickingdegen_pde}, where the stronger hypothesis that $u\in C^{2+\alpha}_s(\underline\sQ)$ is assumed; we refer to those articles for the definition of $C^{2+\alpha}_s(\underline\sQ)$. The changes needed to give the statement and proof of the parabolic version of Lemma \ref{lem:Boundary_properties_C2s_functions} are clear, so we just include the elliptic case. Note that we assume here that $\Sigma\subseteqq\partial\sO$ is a $C^2$, not just a $C^1$ boundary portion as in Lemma \ref{lem:Straightening_boundary}, in order to ensure that the property $\vartheta(\dist(\cdot,\Sigma)) D^2u \in C(\sO\cup\Sigma;\sS^+(d))$ is preserved under a $C^2$ diffeomorphism of $\RR^d$.

\begin{lem}[Boundary properties of second-order derivatives]
\label{lem:Boundary_properties_C2s_functions}
Let $\sO\subset\RR^d$ be an open subset, $\Sigma \subseteqq\partial\sO$ a $C^2$ boundary portion, and $\vartheta \in C[0,\infty)$ such that $\vartheta > 0$ on $(0,\infty)$ and $\vartheta(0)=0$.
If $u \in C^1(\sO\cup\Sigma)$, and $\vartheta(\dist(\cdot,\Sigma))D^2u \in C(\sO\cup\Sigma;\sS^+(d))$, and $x^0\in\Sigma$, then
\begin{equation}
\label{eq:Boundary_properties_C2s_functions}
\lim_{\sO\cup\Sigma \ni x \rightarrow x^0} \vartheta(\dist(x,\Sigma))u_{x_ix_j}(x) = 0,
\end{equation}
for $1 \leq i,j \leq d-1$. If in addition,\footnote{This condition also appears as the `Yamada criterion' \cite[Theorem IV.3.2]{Ikeda_Watanabe2}, one of two conditions which together imply pathwise uniqueness and hence uniqueness of a strong solution to a degenerate one-dimensional stochastic differential equation, $dX(s) = b(X(s))\,ds+\sigma(X(s))\,dW(s)$, where the possibly non-Lipschitz diffusion coefficient, $\sigma(x)$, is required to obey $|\sigma(x)-\sigma(y)|\leq \vartheta^{1/2}(|x-y|)$ for all $x,y\in\RR$.}
\begin{equation}
\label{eq:Vartheta_integral_limit}
\int_0^\delta \frac{1}{\vartheta(s)}\,ds = +\infty,
\end{equation}
for every positive constant $\delta$, then \eqref{eq:Boundary_properties_C2s_functions} holds for $1\leq i,j\leq d$.
\end{lem}

\begin{proof}
First, we consider the case $1 \leq i,j\leq d-1$. Because the function $\vartheta u_{x_ix_j}$ is uniformly continuous on compact subsets of $\sO\cup\Sigma$, the limit in \eqref{eq:Boundary_properties_C2s_functions} exists. By applying a $C^2$ diffeomorphism of $\RR^d$ to straighten the boundary, $\Sigma$, near $x^0$, we may assume without loss of generality that $x^0=0$ (origin in $\RR^d$) and $\sO\cup\Sigma = \RR^{d-1}\times\RR_+$, with $\Sigma=\RR^{d-1}\times\{0\}$ and $\dist(x,\Sigma)=x_d$ when $x=(x',x_d) \in \RR^{d-1}\times\RR_+$. We suppose, to obtain a contradiction, that
\begin{equation}
\label{eq:Positive_limit}
\lim_{\sO\cup\Sigma \ni x \rightarrow x^0} \vartheta(x_d)u_{x_ix_j}(x) = \eta \neq 0,
\end{equation}
and we can further suppose, without loss of generality, that $\eta>0$. Hence, there is a constant, $\eps>0$, such that for all $x=(x',x_d)\in\sO\cup\Sigma$ satisfying
\begin{equation}
\label{eq:Box_epsilon}
0<x_d<\eps \quad\hbox{and}\quad |x'| < \eps,
\end{equation}
we have
\begin{equation}
\label{eq:Lower_bound_second_order_derivative}
\frac{\eta}{2\vartheta(x_d)} \leq u_{x_ix_j}(x',x_d).
\end{equation}
Let $x^1$ and $x^2$ be points satisfying \eqref{eq:Box_epsilon} and such that all except the $x_i$-coordinates are identical. By integrating \eqref{eq:Lower_bound_second_order_derivative} with respect to $x_i\in [x_i^1,x_i^2]$, we obtain
\begin{equation}
\label{eq:Lower_bound_first_order_derivative_ratio}
\frac{\eta(x_i^2-x_i^1)}{2\vartheta(x_d)} \leq u_{x_j}(x^2)-u_{x_j}(x^1).
\end{equation}
We may choose $x^1$, $x^2$ such that $x_i^2-x_i^1=\eps/2$ and $0<x_d <\eps/2$. By taking the limit as $x_d$ goes to zero, the left hand side of \eqref{eq:Lower_bound_first_order_derivative_ratio} converges to $+\infty$, while the right hand side remains finite since $u_{x_j}$ is uniformly continuous on compact subsets of $\sO\cup\Sigma$. Therefore, we must have $\eta=0$ and thus \eqref{eq:Boundary_properties_C2s_functions} holds.

For the case where $i=d$ or $j=d$ (including $i=j=d$), we let $x^1$ and $x^2$ be points satisfying \eqref{eq:Box_epsilon} and such that all except the $x_d$-coordinates are identical. We now integrate \eqref{eq:Lower_bound_second_order_derivative}  with respect to $x_d\in [x_d^1,x_d^2]$ to obtain
\begin{equation}
\label{eq:Lower_bound_first_order_derivative_ratio_d}
\frac{\eta}{2}\int_{x_d^1}^{x_d^2}\frac{1}{\vartheta(x_d)}\,dx_d \leq u_{x_j}(x^2)-u_{x_j}(x^1).
\end{equation}
Because of \eqref{eq:Vartheta_integral_limit}, we again obtain a contradiction when $\eta>0$ by taking the limit of \eqref{eq:Lower_bound_first_order_derivative_ratio_d} as $x_d^1\downarrow 0$.
\end{proof}

%%%%%%%%%%%%%%%%%%%%%%%%%%%%%%%%%%%%%%%%%%%%%%%%%%%%%%%%%%%%%%%%%%%%%%%%%%%%%%%
%
%                                bibliography
%
%%%%%%%%%%%%%%%%%%%%%%%%%%%%%%%%%%%%%%%%%%%%%%%%%%%%%%%%%%%%%%%%%%%%%%%%%%%%%%%

%\bibliography{master,mfpde}
\bibliography{/Users/pfeehan/Dropbox/LATEX/Bibinputs/master,/Users/pfeehan/Dropbox/LATEX/Bibinputs/mfpde}
\bibliographystyle{amsplain-nodash}
\end{document}